\documentclass[11pt]{article}%
\usepackage{geometry}                
\usepackage{tikz} 
\usetikzlibrary{positioning}
\usepackage{graphicx}
\usepackage{amssymb}
\usepackage{epstopdf}
\usepackage{subfigure}  
\usepackage[sans]{dsfont}
\usepackage[english]{babel}
\usepackage{latexsym}
\usepackage{mathrsfs}
\usepackage{graphicx}
\usepackage{color}
\usepackage{float}
\usepackage{mathtools}
\usepackage{amsmath}
\usepackage{amsfonts}
\numberwithin{equation}{section}
\usepackage{enumerate}
\usepackage{amsthm}
\usepackage[title]{appendix}

\usepackage[bookmarks=true,colorlinks=true,linkcolor={blue},urlcolor={blue}, citecolor={blue},pdfstartview={XYZ null null 1.22}]{hyperref}%

\usepackage{cite}

\def\be#1\ee{\begin{equation}#1\end{equation}}

\newtheorem{proposition}{Proposition}

\theoremstyle{definition}

\newtheorem{remark}{Remark}

\newcommand{\xx}{{\bf x}}
\newcommand{\yy}{{\bf y}}
\DeclareMathOperator{\sign}{sign}
\newcommand{\vv}{{\bf v}}
\newcommand{\BB}{{\bf B}}

\newcommand{\EE}{{\bf E}}
\def\RR{\mathbb R}

\def\be{\begin{equation}}
	\def\ee{\end{equation}}
\def\bea{\begin{eqnarray}}
	\def\eea{\end{eqnarray}}

\title{Instantaneous control strategies for magnetically confined fusion plasma}

\author{G. Albi\thanks{Department of Computer Science,
University of Verona, Strada le Grazie 15, 37134 Verona, Italy. (giacomo.albi@univr.it)}\and
G. Dimarco\thanks{Department of Mathematics and Computer Science \& Center for Modeling, Computing and Statistics (CMCS), University of Ferrara, via Machiavelli 30, 44121 Ferrara, Italy. (giacomo.dimarco@unife.it)} \and
F. Ferrarese\thanks{Department of Mathematics and Computer Science \& Center for Modeling, Computing and Statistics (CMCS), University of Ferrara, via Machiavelli 30, 44121 Ferrara, Italy. (federica.ferrarese@unife.it)} \and
L. Pareschi\thanks{Maxwell Institute for Mathematical Sciences and Department of Mathematics,
Heriot-Watt University, Edinburgh, United Kingdom (l.pareschi@hw.ac.uk); Department of Mathematics and Computer Science \& Center for Modeling, Computing and Statistics (CMCS), University of Ferrara, via Machiavelli 30, 44121 Ferrara, ITALY. (lorenzo.pareschi@unife.it)}
}

\begin{document}
	\maketitle
	
		\begin{abstract}
		The principle behind magnetic fusion is to confine high temperature plasma inside a device in such a way that the nuclei of deuterium and tritium joining together can release energy. The high temperatures generated needs the plasma to be isolated from the wall of the device to avoid damages and the scope of external magnetic fields is to achieve this goal. In this paper, to face this challenge from a numerical perspective, we propose an instantaneous control mathematical approach to steer a plasma into a given spatial region. From the modeling point of view, we focus on the Vlasov equation in a bounded domain with self induced electric field and an external strong magnetic field. The main feature of the control strategy employed is that it provides a feedback on the equation of motion based on an instantaneous prediction of the discretized system. This permits to directly embed the minimization of a given cost functional into the particle interactions of the corresponding Vlasov model. The numerical results demonstrate the validity of our control approach and the capability of an external magnetic field, even if in a simplified setting, to lead the plasma far from the boundaries.
	\end{abstract}
{\bf Keywords:} Vlasov equation, instantaneous control, plasma physics, magnetic confinement, kinetic equations, particle methods.\\
	\textbf{Mathematics Subject Classification}: 35Q93, 82D10, 35Q83, 65Mxx

\section{Introduction}\label{sec:intro}
Plasma is an electrically conducting fluid where temperature is so high that electrons are separated by their atoms and free to move^^>\cite{Bitt}. This is called the fourth state of the matter and, in this state, a gas looks macroscopically neutral meaning that it is composed by roughly the same number of positive and negative charges. Perturbations from neutrality arises only at the microscopic level, and they are responsible of the passage of currents inside the plasma^^>\cite{Chen}. Since most of the visible universe appears to be in the state of plasma, its behavior and properties are of intense interest to scientists in many disciplines ranging from physics, engineering to mathematics. In this framework, the development of numerical methods for solving plasma physics problems has attracted a lot of attention in the recent years^^>\cite{Cheng, ghizzo1993eulerian,sonnendrucker1999semi,crouseilles2010conservative,filbet2018numerical,yang2014conservative,dimarco2015numerical, SonnLecture,Degond, Degond2}. In particular, it is of great interest the study of magnetized plasma for its application in the so called fusion devices, such as Tokamaks or Stellarators^^>\cite{ravindran2005real,Pnas,fasoli2016computational, davies2023magnetic,Degond2,Grandgirard}. In these machines, a strong magnetic field tries to contain the plasma during the fusion process of deuterium with tritium, and to avoid the direct contact of the charged particles with the walls of the devices which will cause failure of the system. In order to describe such phenomenon, there exists many different mathematical models and associated numerical methods^^>\cite{Pnas,BadsiMod,Bessemodel,MacroDegond,ModMethDel,Grandgirard} which are able to characterize at different levels the features observed during the time evolution of a plasma. The choice of the appropriate one depends on the spatial or temporal scale at which researchers are interested when studying the fusion phenomenon^^>\cite{ModMethDel}. Typical examples are macroscopic MHD equations or microscopic kinetic models, with or without collisions. From the numerical approximation perspective, a very important role has been played in the recent past by the so-called asymptotic preserving methods^^>\cite{Degond,Degond2,filbet2016asymptotically,filbet2017asymptotically,DimarcoRispoli,DPacta}. These schemes are able to efficiently deal with the different physical scales involved in the dynamics, and are able to treat the problem of the quasi-neutrality in plasma^^>\cite{DimarcoCrouseilles,quasinVignal,Jun,Cro}. We mention here also recent works discussing the presence of uncertainties in kinetic models of plasma physics^^>\cite{medaglia2023stochastic, Frank2021}.

In this work, we consider specifically one of the models commonly used to describe the evolution of a plasma when far from equilibrium, i.e. when collisions are not enough to reach an equilibrium state. The model is based on the Vlasov equation which describes the evolution of charged particles in an electromagnetic self-consistent or externally applied field. For one single species of the plasma, this reads as 
\begin{equation}\label{eq:vlasov}
	\frac{\partial{f(t,\xx,\vv)}}{\partial{t}} +   \vv \cdot \nabla_\xx f(t,\xx,\vv) + F(t,\xx,\vv) \cdot \nabla_\vv f(t,\xx,\vv)   = 0,
\end{equation} 
where 	\begin{equation}\label{eq:density}
	f = f(t,\xx,\vv), \qquad f: \mathbb{R}_+\times \mathbb{R}^{d_x}\times \mathbb{R}^{d_v},
\end{equation} 
is the so-called distribution function giving the probability for the species, either ions or electrons, of being in a certain position of the $d_x$ dimensional space and of the $d_v$ dimensional velocity space, and where $F(t,\xx,\vv)$ represents a force field which can take different forms. For example, for the Vlasov-Poisson case, one has
\begin{equation}\label{eq:Poisson}
	F(t,\xx,\vv) = \frac{q}{m}\EE(t,\xx), \quad \EE(t,\xx) = -\nabla_\xx \phi(t,\xx), \quad -\Delta_\xx \phi(t,\xx) = \frac{\rho(t,\xx)-\rho_0(t,\xx)}{\varepsilon_0},
\end{equation}
where $q$ is the elementary charge, $m$ the mass of a single particle, $\varepsilon_0$ the permittivity, $\EE(t,\xx)$ the electric field, $\phi(t,\xx)$ the electric potential,
\begin{equation}\label{eq:charge_density}
	\rho(t,\xx) = \int_{\mathbb{R}^{d_v}} f(t,\xx,\vv)d\vv
\end{equation}
the charge density, and $\rho_0(t,\xx)$ a static neutralizing background. On the other hand, in the Vlasov-Maxwell case one has
\begin{equation}\label{eq:F}
	F(t,\xx,\vv) = \EE(t,\xx) + \vv \times \BB(t,\xx),
\end{equation} 
where $\BB(t,\xx)$ represents the magnetic field, and where the equation \eqref{eq:vlasov} is coupled with the solution of the Maxwell equations 
\begin{equation}\label{eq:Maxwell}
	\begin{split}
		&\frac{\partial \EE(t,\xx)}{\partial t} = c^2 \nabla_\xx \times \BB(t,\xx) -\frac{\bf{J}}{\varepsilon_0}, \qquad \nabla_\xx \cdot \BB(t,\xx) = 0,\\
		&\frac{\partial \BB(t,\xx)}{\partial t} = -\nabla_\xx \times \EE(t,\xx), \qquad \nabla_\xx \cdot \EE(t,\xx) = \frac{\rho(t,\xx)-\rho_0(t,\xx)}{\varepsilon_0},  
	\end{split}
\end{equation}
with $\mathbf{J}=q\int_{\mathbb{R}^{d_v}} f(t,\xx,\vv) d\vv$ the current density, $c$ the speed of light, and the compatibility condition 
\begin{equation}\label{eq:compatibility}
	\frac{\partial \rho}{\partial t}+\nabla_\xx \cdot\mathbf{J}=0.
\end{equation}
In this work, as first step towards more realistic settings, we set ourselves in an intermediate situation between the Poisson and the Maxwell systems in which the magnetic field is only external and given while the electric field is obtained from the solution of the Poisson equation, i.e.
\begin{equation}\label{eq:our model}\begin{split}
	&-\Delta_\xx \phi(t,\xx) = \frac{\rho(t,\xx)-\rho_0(t,\xx)}{\varepsilon_0}, \quad \BB(t,\xx)=\BB_{ext}(t,\xx), \\ 	
	& F(t,\xx,\vv) = \EE(t,\xx) + \vv \times \BB_{ext}(t,\xx), \quad \EE(t,\xx)= -\nabla_\xx \phi(t,\xx).
		\end{split}
\end{equation}
The numerical solution of the multi-dimensional Vlasov-Poisson system \eqref{eq:vlasov}-\eqref{eq:Poisson}, or the Vlasov-Maxwell system \eqref{eq:vlasov}-\eqref{eq:Maxwell} or of the Vlasov-Poisson with external magnetic field \eqref{eq:our model}, is in itself a major challenge because of the high dimensionality, the intrinsic structural properties and the temporal and spatial scale involved. 

In general, the aforementioned dynamics can be discretized using various numerical techniques, such as semi-Lagrangian schemes^^>\cite{crouseilles2010conservative,russo2009semilagrangian,yang2014conservative}, finite difference, finite volume, or low-rank methods^^>\cite{crouseilles2004numerical,Hu}, and Discontinous-Galerkin schemes^^>\cite{de2012discontinuous,FilbetGal,Gamba}. These methods have demonstrated their ability to provide accurate solutions while maintaining some of the structural properties of the system. However, their computational cost increase significantly in multi-dimensional scenarios. In contrast, particle based methods, like Particle-In-Cell (PIC) method^^>\cite{chacon2016optimization,filbet2016asymptotically,filbet2017asymptotically,hairer2018energy,gu2022hamiltonian,PICDel,DimarcoCaflisch,CaflischPIC}, offer a flexible and efficient approach for plasma simulations, particularly in scenarios with complex dynamics and high-dimensional spaces. 
Another class of schemes, hybrid schemes, couples different models trying to combine the advantages of deterministic and particle solvers, as seen in^^>\cite{DimarcoCaflisch2,Rosin,Rey,Samaey,SamaeyII} and related works.

In the present paper, we focus on PIC methods with the primary goal of designing efficient instantaneous control strategies. These strategies are based on the external magnetic field $\BB_{ext}(t,\xx)$, obtained as a solution of an optimality principle, and aim to minimize the mass hitting the boundaries and/or the thermal energy near the walls. We subsequently examine the numerical outcomes resulting from the implementation of these strategies. While there have been attempts in the literature to study the magnetic confinement of plasma through numerical simulations, most of these efforts have focused on hydrodynamic models^^>\cite{Blum}. The control of microscopic models describing systems of charged particles and based on kinetic equations is a direction that has been largely unexplored in the past. We mention recent references^^>\cite{Einkemmer2024, bartsch2023controlling}, and also refer to^^>\cite{Caflisch21} for the case of collisional kinetic equations.

A fundamental aspect when designing a numerical method to control a plasma simulation is the presence of an external magnetic field, which acts as a control on the system, necessary for the plasma to assume a desired configuration.
To this aim, the general form of the functional we aim to minimize in the following reads 
\begin{equation}\label{eq:control_pb}
			\mathcal{J}(\BB_{ext};f^0;f) : = \int_{0}^T \left( \mathcal{D}(f,\psi)(t) + \frac{\gamma}{2} \int_{\Omega_x} \Vert \BB_{ext}(t,\xx) \Vert^2  d\xx \right)dt,
\end{equation}
where  $\Omega_x \subseteq \RR^{d_x}$ represents the  space domain, $\mathcal{D}(f,\psi)(\cdot)$ is a running cost function, and  $\gamma$ acts as a weight penalizing the magnitude of the control $\Vert \BB_{ext} \Vert$.  For instance, by choosing
\begin{equation}
 \mathcal{D}(f,\psi)(t)=\Big \Vert \int_{\Omega} \psi(\xx,\vv)     \left( f(t,\xx,\vv)-\hat{f}(t,\xx,\vv) \right)   d\xx d\vv \Big \Vert^2,
\end{equation}
where $\Omega= \Omega_x \times \Omega_v$ with $\Omega_v\subseteq \RR^{d_v}$ and $\hat{f}$ a certain target distribution, we are looking to steer the moments of the distribution functions towards a desired state. As an example, the specific case $\psi(\xx,\vv) = 1$ corresponds to force the density of the plasma to be as close as possible to a given configuration. Given the functional \eqref{eq:control_pb}, the general form of the control problem we aim to study is given by 
\begin{equation}\label{eq:control_pb2}
	\begin{split}
		&\min_{\BB_{ext}\in \mathcal{B}^{adm}}\mathcal{J}(\BB_{ext}; f^0;f) \cr
		&\, \textrm{subject to}\\
		&\qquad \partial_t f(t,\xx,\vv) +   \vv \cdot \nabla_\xx f(t,\xx,\vv) + \left(\EE(t,\xx) + \vv \times \BB_{ext}(t,\xx)\right) \cdot \nabla_\vv f(t,\xx,\vv)   = 0,\\ 
		&\qquad-\Delta_\xx \phi(t,\xx) = \frac{\rho(t,\xx)-\rho_0(t,\xx)}{\varepsilon_0}, \qquad \EE(t,\xx) = -\nabla_\xx \phi(t,\xx).
	\end{split}
\end{equation}
The above problem has been analyzed theoretically in^^>\cite{knopf2020optimal,weber2021optimal,Caprino,Kwan}, where the magnetic field has been computed as the superposition of the fields that are generated from different coils. 
A simplified configuration has been also analyzed theoretically in^^>\cite{bartsch2023controlling}, deriving the control following a Lagrangian approach, and by solving the subsequent optimality system. Another similar research direction has been recently explored in^^>\cite{Einkemmer2024}, where the control is performed over the electric field.
However, up to our knowledge, instantaneous control methods of the type described in^^>\cite{Burger,Albi,Fornasier} have never been investigated in the setting of plasma physics described by means of kinetic models which is the direction we aim to explore in this work. In more details we assume, as it will be clarified later, that the external magnetic field may take different values on different parts of the domain, and, for this specific choice, we derive an instantaneous feedback control for the plasma which permits to drive the system towards the desired state. 
Furthermore, the feedback control is derived at both the particle and mesoscopic levels, where it is possible to demonstrate that by carefully designing the time discretization and introducing an appropriate scaling, the discretize-then-optimize (DtO) approach and the optimize-then-discretize (OtD) approach become equivalent. For analogous results in this direction over different models, see, for example,^^>\cite{hager2000runge,sanz2016symplectic,albi2019linear,herty2013implicit}.

The rest of the paper is organized as follows. In Section \ref{sec:formulation} we describe the geometrical setting and we introduce the details of the control problem. In Section \ref{sec:Numerical_methods}, we give a general overview of the type of numerical scheme we choose to approximate our model. In particular, we will focus on  Particle-In-Cell methods together with well suited semi-implicit discretizations. In Section \ref{sec:control}, we discretize the control problem and we derive different control strategies showing the consistency between discretize then optimize, and optimize then discretize approaches. In Section \ref{sec:Numerical_tests}, we propose different numerical examples in which we apply the control to the plasma dynamics.  A concluding Section \ref{sec:conclusion} summarizes the work done and provides some future directions of investigation. In the Appendix \ref{sec:appendix}, 
for validation purposes, we analyze through some tests the proposed numerical method. 

\section{Problem setting}\label{sec:formulation}
We restrict ourselves to a two dimensional setting in phase space, i.e. $d_x=d_v=2$. This situation mimics the evolution of a plasma inside a three dimensional axisymmetric toroidal device. To give a precise definition of our simplified setting, we first consider a two dimensional horizontal section of the three dimensional torus shown on the left of Figure \ref{fig:toro}, which is obtained from the intersection of the $(X,Y)$-plane with the solid. We then focus on a portion of this section, enlightened in red. This section is successively approximated by a rectangle in a new reference frame $(X_\perp,Y_\perp)$ to simplify the description of the computational domain, as shown in the right picture in Figure \ref{fig:toro}. 
For accurate description of the dynamics, directly evaluated over the section of the torus, we refer to^^>\cite{yang2014conservative}.
 \begin{figure}[h!]
\hspace{-2.3cm}
	\includegraphics[width=1.5\linewidth]{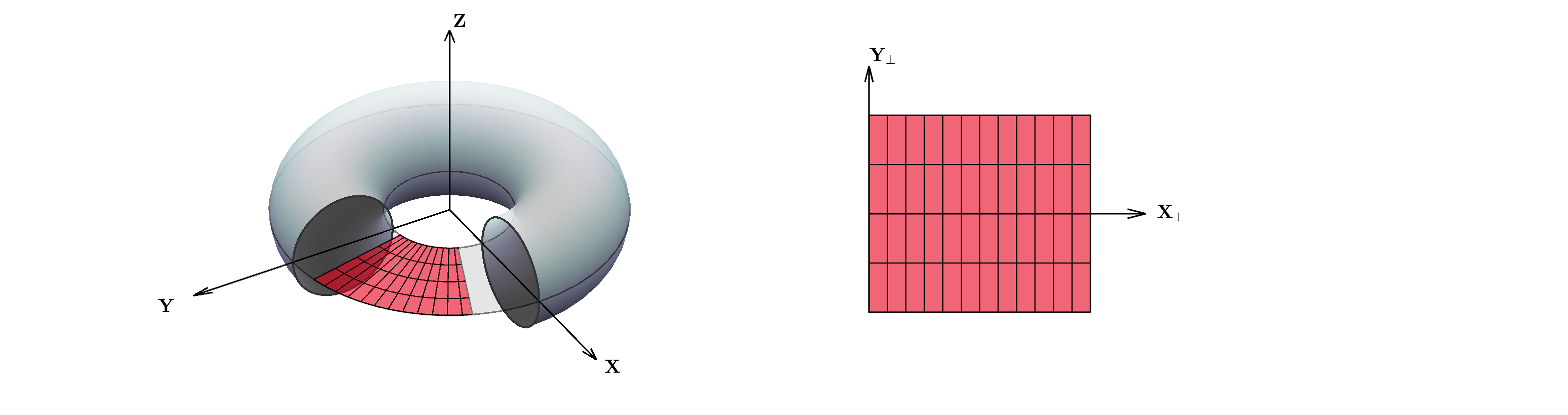}
	\caption{On the left the three dimensional torus, on the right the simplified geometry considered.}
	\label{fig:toro}
\end{figure}

Over this depicted geometry, the external magnetic field is supposed to be such that
\begin{equation}\label{eq:magnetic_field}
	\BB_{ext}(t,\xx) = \left( 0,0, B^{ext}(t,\xx)\right).
\end{equation} 
Hence, the two dimensional Vlasov-Poisson equations with such external magnetic field reads as 
\begin{equation}\label{eq:vlasov_poisson}
	\begin{split}
		&  \partial_t f(t,\xx_\perp,\vv_\perp) +   \vv_\perp \cdot \nabla_{\xx_\perp} f(t,\xx_\perp,\vv_\perp) +\\ &+ \left(  \EE(t,\xx_\perp) +  \vv_\perp \times \BB_{ext}(t,\xx_\perp)\right) \cdot \nabla_{\vv_\perp} f(t,\xx_\perp,\vv_\perp)   = 0,\\
		& \EE(t,\xx_\perp) = -\nabla_{\xx_\perp} \phi(t,\xx_\perp),\qquad -\Delta_{\xx_\perp} \phi(t,\xx_\perp) = \rho(t,\xx_\perp),
	\end{split}
\end{equation}
where $\xx_\perp = (x,y)$, $\vv_\perp= (v_x,v_y)$ denote the coordinates orthogonal to the magnetic field, and where we took $\rho_0=0$ and the scaling parameter $\varepsilon_0=1$ in the Poisson equation. In \eqref{eq:vlasov_poisson}, $f(t,\xx_\perp,\vv_\perp)$ represents an ensemble of charged particles lying on the $x-y$ plane where we impose periodic boundary conditions in $x$ direction while, if not otherwise state, reflective boundary conditions are imposed in $y$ direction, mimicking the wall of the device. The aim then is to control instantaneously the dynamics described by \eqref{eq:vlasov_poisson} through an external magnetic field in order to let the charged particles to assume a desired configuration and to stay as far as possible from the walls. To reach this goal, the idea developed consists in using the external magnetic field $\BB_{ext}(t,\xx_\perp)$ as a control variable able to steer the system. However, in order to cope with a realistic setting, instead of considering the possibility for $\BB_{ext}(t,\xx_\perp)$ to take pointwise values, we introduce a space discretization grid with $N_c$ cells, and we derive a control $B_k^{ext}(t)$ taking constant values in each of the cells $C_k\subset \Omega_x$,  such that $\bigcup_{k=1}^{N} C_k=\Omega_x$, $C_k\cap C_\ell=\emptyset$ for all $k\neq \ell$, and $k, \ell\in\{1,\ldots,N_c\}$, with in the following $\Omega_k = C_k \times \Omega_v$ the state space of the single cell. We stress that the choice of $N_c$ is not due to the numerical setting while instead it has to be considered as the given physical setting of the underlying problem.
In fact, from the practical point of view, one cannot expect that the coils used to generate the external magnetic field would be capable of realizing very complex and pointwise field structures. Of course, as it is natural to expect, relaxing this constraint and allowing a pointwise magnetic field would make the control task easier as detailed in the numerical test section. 

Hence, we reformulate the problem at the continuous level and over a finite time horizon $[0,T]$ as follows
\begin{equation}\label{eq:continuos_pb}
	\min_{B^{ext}\in \mathcal{B}_{adm}}  \sum_{k=1}^{N_c} \mathcal{J}_k(B_k^{ext};f_k,f^0_k),\qquad 
\textrm{s.t.}~\eqref{eq:vlasov_poisson},
\end{equation}
where  $f_k= f_k (t,\xx_\perp,\vv_\perp)$ corresponds to the normalized particle density restricted to a single cell $C_k$: 
\[
f_k (t,\xx_\perp,\vv_\perp) = \frac{f(t,\xx_\perp,\vv_\perp)}{\rho_k(t)},\qquad \rho_k(t) = \int_{\Omega_k}f(t,\xx_\perp,\vv_\perp)\,d\xx_\perp\, d\vv_\perp,
\]
with $\rho_k(t)>0$ the total cell density and with $B^{ext}=(B^{ext}_1,\ldots,B^{ext}_{N_c})$ now representing the vector of $z$  components of $\BB_{ext}(t,\xx_\perp)$ defined in \eqref{eq:magnetic_field} within each cell $C_k$, $\mathcal{B}_{adm}$ the set of admissible controls such that
 $\mathcal{B}_{adm} = \{B_k^{ext} | B_k^{ext}\in[-M,M], M>0,\, k = 1,\ldots, N_c\}$, and where, for each $k = 1,\ldots, N_c$, the cost functional is defined as follows
\begin{equation}\label{eq:J}
	\begin{split}
		\mathcal{J}_k(B^{ext}_k; f_k,f_k^0) =\int_{0}^T\left(    \sum_{ \ell \in \{\textrm{x},\textrm{v}\}}\mathcal{D}_k(f_k,\psi_{\ell})(t) + \frac{\gamma}{2} \Vert B_k^{ext}(t)\Vert^2 \right)\, dt,
\end{split}
\end{equation}
with $\gamma>0$ a penalization parameter of the control, and where the terms of the running cost $\mathcal{D}_k(f_k,\psi_{\textrm{x}})(t)$, $\mathcal{D}_k(f_k,\psi_{\textrm{v}})(t)$ aims at enforcing a specific configuration in the space and velocity distribution read as follows
\begin{equation}\label{eq:function_D_ell}
	\begin{split}
		\mathcal{D}_k(f_k,\psi_\ell) = \frac{\alpha_\ell}2\| m_k(f_k,\psi_\ell)(t) - \hat \psi_{\ell,k}\|^2 +  \frac{\beta_\ell}2\sigma^2_k(f_k,\psi_\ell)(t),	\qquad \ell = \{\textrm{x},\textrm{v}\},
	\end{split}
\end{equation}
with $\psi_{\ell} = \psi_{\ell}(\xx_\perp,\vv_\perp)$ a function of the state variables,$ \hat \psi_{\ell,k} \equiv \hat\psi_{\ell,k}(\xx_\perp,\vv_\perp) $ the target state, $\alpha_\ell, \beta_\ell \geq 0$, and the following moment functions 
\begin{equation}\label{eq:mean_var}
	\begin{split}
    {m}_{k}(f_k,\psi_\ell)(t)&=\int_{\Omega_k} \psi_{\ell}(\xx_\perp,\vv_\perp)f_k(t,\xx_\perp,\vv_\perp) d\xx_\perp d\vv_\perp,\cr
	\sigma^2_{k}(f_k,\psi_\ell)(t) &=  \int_{\Omega_k}\| \psi_{\ell}(\xx_\perp,\vv_\perp) - {m}_{k}(f_k,\psi_\ell)(t)\|^2 f_k(t,\xx_\perp,\vv_\perp) d\xx_\perp d\vv_\perp.
	\end{split}
\end{equation}
Now, the purpose of the first term in \eqref{eq:function_D_ell} is to force the moments of the distribution function in each cell $C_k$, such as mass and momentum, towards a desired values $\hat\psi_{\ell}$. At same time, the second term aims at forcing the distribution function in such a way that the variance, in space and/or in velocity, is as small as possible. In other words, one of the scope of this second term is to avoid particles with large velocity lying on the tail of the distribution function to escape from the control and eventually to reach the walls. 
\begin{remark}
	In the rest of the paper we will use for simplicity the following notation $B(t,\cdot) = B^{ext}(t,\cdot)$, $\BB(t,\cdot) = \BB_{ext}(t,\cdot)$, $\xx = \xx_\perp$ and $\vv = \vv_\perp$.
\end{remark}

\section{Particle-in-cell (PIC) Methods}\label{sec:Numerical_methods}
In this section we focus on the discretization of the Vlasov-Poisson equation \eqref{eq:vlasov_poisson}. To this aim, we will rely on Particle-in-Cell (PIC) methods^^>\cite{chacon2016optimization,hairer2018energy,filbet2017asymptotically,SonnLecture} which will be used to study the motion of $N$ particles approximating the single species plasma density $f(t,\xx,\vv)$. We consequently define the approximated distribution function at time $t^{n+1}$ as 
\begin{equation}\label{eq:approx_density}
	f^N(t^{n+1},\xx,\vv) = \sum_{m=1}^N\omega_m \delta(\xx-\xx_m^{n+1})\delta(\vv-\vv_m^{n+1}),
\end{equation}
where $\delta(\cdot)$ is the Dirac-delta distribution, and $\omega_m$ is a weight related to the mass of each particle. Then, the trajectory of single particle $m=1,\ldots,N$ is computed from the characteristic curves originating from the Vlasov equation
\begin{equation}\label{eq:char_curves} 
	\begin{split}
		& \frac{d\xx_m(t)}{dt} = \vv_m(t), \qquad \xx_m(0) = \xx^0_m,  \\
		& \frac{d\vv_m(t)}{dt} = \vv_m(t) \times \BB(t,\xx_m(t)) + \EE(t,\xx_m), \qquad \vv_m(0) = \vv^0_m.
	\end{split}
\end{equation}
In order to integrate equation \eqref{eq:char_curves}, we first introduce a spatial mesh of size $\Delta x$, with $m_x$ nodes in the $x$ direction, and of size $\Delta y$, with $m_y$ nodes in the $y$ direction. In the following we will denote by $\bf{X}_j$ the nodes of the space grid assuming, for simplicity, to have a single index $j=1,\ldots,M_c$, with $M_c$ the total number of cells. The value of the electric field $\EE(t,\bf{X}_j)$ on the mesh points is computed through the solution of the Poisson equation by means of a finite difference method, 
while the computation of the magnetic field will be detailed later, and it will be the consequence of the minimization of the cost functional \eqref{eq:J}. 

At each time $t$, the approximated density of the particles which is needed to solve the Poisson equation is reconstructed on each spatial cell $\bf{X}_j$ considering the updated particles positions and velocities. The new position is due to a shift at velocity $\vv_m(t)$ in the physical space, while the new velocity being due to the force term acting at the particles position $\xx_m$. The value of the electric field at particle position $\xx_m$ corresponds to the value assumed by the electric field in the cell $j$ to which the particle with position $\xx_m$ belongs. 
The system is initialized by sampling the initial position and velocity of each particle from a given initial distribution: $f(t=0,\xx,\vv)=f_0(\xx,\vv)$. To do so, we first determine the mass of each particle by computing the total mass of the plasma, and by dividing it by the total number of particles $N$ used to approximate the distribution function, i.e.
\begin{equation}\label{eq:mtot}
	\rho_{tot} = \int_{\Omega} f_0(\xx,\vv) d\xx d\vv, \qquad \omega_m = \frac{\rho_{tot}}{N}.
\end{equation}
Then, the number of particles in each cell is given by
\begin{equation}
	N_j = \Big \lfloor \frac{\rho_j}{N} \Big \rceil,
\end{equation} 
where, to avoid any bias, $\lfloor \cdot \rceil$ denotes the stochastic rounding, and $\rho_j$ is given by
\begin{equation}\label{eq:rho_i}
	\rho_j = \int_{\Omega_v} f_0(\xx_j,\vv) d\vv, \quad	\forall j=1,..,M_c.
\end{equation}
Finally, we set uniformly in space the positions of the $N_j$ particles inside each cell $j=1,\ldots,M_c$. Following a similar argument, we associate to each particle a given velocity accordingly to a grid in velocity space with $m_{v_x}\times m_{v_y}$ nodes, used to approximate the initial distribution in velocity. Let observe that the choice of assigning positions and velocities randomly or deterministically at $t=0$ will lead to different convergence rates of the PIC methods in the long time behavior as detailed in the appendix \ref{sec:appendix}.

Now, for what concerns the time discretization, among all the possible methods, we focus on a semi-implicit strategy originally proposed in^^>\cite{filbet2016asymptotically} to discretize the system of equations \eqref{eq:char_curves}. This scheme was developed with the aim of handling strong magnetic field which may be the case in our setting. In fact, the magnetic field used to control the system may in principle take large value during the time evolution of the plasma when trying to force particles not to deviate from a desired direction. Thus, a semi-implicit method guarantees stability for a wider range of time steps and consequently is more suited for our scopes. We then consider a time interval $[0,T]$ divided into $N_t$ intervals of size $h$ and we define $t^n = n h$ and for any $m=1,\ldots,N$, one has the following first order in time scheme
\begin{equation}\label{eq:filbet_scheme}
	\begin{split}
		&\xx_m^{n+1} = \xx_m^n + h \vv_m^{n+1},\\
		&\vv_m^{n+1} = \vv_m^n + h \vv_m^{n+1}\times \BB(t^n,\xx_m^n) +  h \EE(t^n,\xx_m^n),
	\end{split} 
\end{equation}
where $\BB(t^n,\xx_m^n)$ is the external magnetic field, computed as a result of the control problem which will be set and solved in the next Section \eqref{sec:control}.

A second order in time extension of the above first order semi-implicit method can be used as well^^>\cite{filbet2016asymptotically}. This reads
  \begin{equation}\label{eq:second_order_first_stage}
	  	\begin{split}
		  		&\xx_m^{(1)} = \xx_m^n + \frac{h}{2} \vv_m^{(1)},\\
		  		&\vv_m^{(1)} = \vv_m^n + \frac{h}{2} \left( \vv_m^{(1)} \times \BB(t^n,\xx_m^n) + \EE(t^n,\xx_m^n)\right),
		  	\end{split}
	  \end{equation}
where the apex $(1)$ denotes the first stage, while the second stages, for $m=1,\ldots,N$, reads 
  \begin{equation}\label{eq:second_order_second_stage}
		\begin{split}	
				&\xx_m^{(2)} = \xx_m^n + \frac{h}{2} \vv_m^{(2)},\\
				&\vv_m^{(2)} = \vv_m^n + \frac{h}{2} \left( \vv_m^{(2)} \times \BB(t^{n+1},2\xx_m^{(1)}-\xx_m^n) + \EE(t^{n+1},2\xx_m^{(1)}-\xx_m^n)\right) .
			\end{split}
	\end{equation}
	Finally, the numerical solution at time $t^{n+1}$ reads,
	\begin{equation}\label{eq:second_order_scheme}
	\begin{split}
		&\xx_m^{n+1} = \xx_m^{(1)}+\xx_m^{(2)}-\xx_m^n,\\
		&\vv_m^{n+1} = \vv_m^{(1)}+\vv_m^{(2)}-\vv_m^n,\\
	\end{split}
	\end{equation}
	for any $m=1,\ldots,N$.
A complete analysis of this family of semi-implicit PIC numerical methods can be found in^^>\cite{filbet2016asymptotically}.

\section{Derivation of the instantaneous control strategy} \label{sec:control}
As mentioned in Section \ref{sec:formulation}, we consider a space discretization grid with  $N_c$ cells of size $\Delta^c  x \times \Delta^c y$ with $\Delta^c x\gg\Delta x$, and $\Delta^c y\gg\Delta y$, 
and we provide a feedback control $B_k$, $k=1,\ldots,N_c$ taking constant values in each cell $C_k$ at every instant of time $t^n$, and based on a one-step prediction of the dynamics. We stress that we consider the mesh $\Delta^c  x \times \Delta^c y$ to be given, depending on the physics of the problem while the mesh $\Delta  x \times \Delta y$ to depend on the level of resolution aimed in the approximation of the Vlasov equation \eqref{eq:our model}. Once the physical and numerical setting have been defined, the idea used to determine the expression of the magnetic field consists in  splitting the time interval $[0,T]$ into small time intervals of length $h$, and to solve a sequence of optimal control problems of the general form
\begin{equation}\label{eq:min_prob}
	\min_{B\in \mathcal{B}_{adm}}   \sum_{k=1}^{N_c}\mathcal{J}^{h}_{k}(B_k; f_k,f_k^{0})=\sum_{k=1}^{N_c}\int_{t}^{t+h} \left( \sum_{ \ell \in \{\textrm{x},\textrm{v}\}}\mathcal{D}_k(f_k,\psi_{\ell})(\tau) + \frac{\gamma}{2} \Vert B_k(\tau)\Vert^2 \right)\, d\tau.
\end{equation}
One can successively proceed as discussed in the Introduction \ref{sec:intro} to derive the control strategy with the discretize-then-optimize (DtO) approach or with the optimize-then-discretize (OtD) one. In the following we follow the DtO path while in the final part of the section we will discuss the analogies and differences with the OtD method.

\subsection{Discretize then optimize (DtO)}
In this setting  due to the implicit nature of the discretized dynamics given in \eqref{eq:filbet_scheme}, one can first notice that it is not possible to solve the one step optimal control problem analytically. For this reason, we derive an approximated expression of the magnetic field using first a different time discretization for the velocity term of the Vlasov equation \eqref{eq:vlasov}, fully explicit, while keeping implicit the time integration for the particle position and then we plug this result into the semi-implicit scheme \eqref{eq:filbet_scheme}. This procedure provides a sub-optimal feedback control with respect to the original problem \eqref{eq:J}, containing an error of order $h$. Nonetheless, as shown in the numerical test Section, the control design guarantees that the particles remain far from the walls for all time of the simulations.  In order to enforce such behavior we will consider the following choice for $\psi_\ell=\psi_\ell(\xx,\vv)$ in \eqref{eq:min_prob}
\begin{equation}\label{eq:states}
	\psi_{\textrm{v}}(\xx,\vv) = v_y,  \quad		\hat \psi_{\textrm{v},k}(\xx,\vv) = \hat v_{y,k}, \quad		\psi_{\textrm{x}}(\xx,\vv) = y,  \quad	\hat \psi_{\textrm{x},k}(\xx,\vv) = \hat{y}_k.
\end{equation}
Now, based on the particle-based approximation induced by the PIC-scheme introduced in Section \ref{sec:Numerical_methods} we first restrict the optimization problem \eqref{eq:min_prob} to the domain
$\Omega_k$.
Hence, we formulate the discretized version of  \eqref{eq:continuos_pb} in a short time horizon $[t, t + h]$  as follows 
\begin{equation}\label{eq:min_prob_cell}
	\min_{B_k\in [-M,M]} \mathcal{J}^{N,h}_{k}(B_k; f_k^N,f_k^{N,0}),
\end{equation}
subject to a semi-implicit in time discretized Vlasov dynamics, fully explicit for the velocity terms, reading
\begin{equation}\label{eq:explicit_scheme}
	\begin{split}
		&x_i^{n+1} = x_i^n + h v_{x_i}^{n+1},\\
		&y_i^{n+1} = y_i^n + h v_{y_i}^{n+1},\\
		&v_{x_i}^{n+1} = v_{x_i}^n+h v_{y_i}^nB_k + h E_{x_i}^n,\\
		&v_{y_i}^{n+1} = v_{y_i}^n - h v_{x_i}^nB_k + h E_{y_i}^n,
	\end{split}
\end{equation}

Now, by considering also a discretized version of \eqref{eq:min_prob} using the rectangle rule for approximating the integral in time, the functional in \eqref{eq:min_prob_cell} reads as follows
\begin{equation}\label{eq:discr_J_1}
	\mathcal J_k^{N,h}(B_k;f_k^N,f_k^{N,0}) =h\left(\mathcal{D}_k(f_k^N,\psi_{\textrm{x}})(t^{n+1}) +	\mathcal{D}_k(f_k^N,\psi_{\textrm{v}})(t^{n+1}) + \frac{\gamma}{2} \Vert B_k\Vert^2 \right),
\end{equation}
 and where 
 \begin{equation}\label{eq:function_D_ell_discr}
	\begin{split}
		\mathcal{D}_k(f_k^N,\psi_\ell) = \frac{\alpha_\ell}2\| m_k(f_k^N,\psi_\ell) (t^{n+1}) - \hat \psi_{\ell,k}\|^2 +  \frac{\beta_\ell}2\sigma^2_k(f_k^N,\psi_\ell) (t^{n+1}),	\qquad \ell = \{\textrm{x},\textrm{v}\},
	\end{split}
\end{equation}
is equivalent to  \eqref{eq:function_D_ell}, with evaluation over the empirical density $f_k^N(t,x,v)$.
Thus, by setting $\psi_\textrm{x} = y^{n+1}$, $\psi_{\textrm{v}} = v_y^{n+1}$, $\hat{\psi}_{\textrm{x},k} = \hat{y}_k$ and $\hat{\psi}_{\textrm{v},k} = \hat{v}_{y_k}$, target states, according to \eqref{eq:states} and by direct computation over the empirical densities, we can rewrite the functional in \eqref{eq:discr_J_1} as
\begin{equation}\label{eq:discr_J}
		\begin{split}
	\mathcal J_k^{N,h}(B_k)&= \frac{h\alpha_\emph{v}}{2}  \Vert \bar{v}_{y,k}^{n+1} -\hat{v}_{y_k}\Vert^2 + \frac{h\beta_\emph{v}}{2 N_k} \sum_{i\in C_k} \Vert v_{y_i}^{n+1}-\bar{v}_{y,k}^{n} \Vert^2 +\\
	&\qquad+ \frac{h\alpha_\emph{x}}{2}  \Vert \bar{y}_{k}^{n+1} -\hat{y}_k\Vert^2 + \frac{h\beta_\emph{x}}{2 N_k} \sum_{i\in C_k} \Vert y_{i}^{n+1}-\bar{y}_k^{n} \Vert^2+  \frac{h\gamma}{2}  \Vert B^n_k \Vert^2,
		\end{split}
	\end{equation}
  with 
  \begin{equation}\label{eq:mean_quantities}
	\begin{split} 	
		\bar{y}_{k}  = \frac{1}{N_{k}} \sum_{j\in C_k} y_{j} ,\qquad
		\bar{v}_{y,k}  = \frac{1}{N_{k}} \sum_{j\in C_k} v_{y_j},
	\end{split}
\end{equation}
denoting the mean position and velocity over cell $C_k$, $k=1,\ldots,N_c$. 
Now, the following Proposition holds true.
\begin{proposition}
	Assume the parameters to scale as 
\begin{equation}\label{eq:scaling}
	\alpha_\emph{x} \rightarrow \frac{\alpha_\emph{x}}{h}, \qquad \beta_\emph{x} \rightarrow \frac{\beta_\emph{x}}{h}, \qquad \gamma \rightarrow \gamma h,  
\end{equation}
then the feedback control at cell $C_k$ associated to \eqref{eq:discr_J} reads as follows
	\begin{equation}\label{eq:L2_control_space_velocity}
		B_k = \mathbb{P}_{[-M,M]}\left( \frac{\mathcal{R}^{N,n}_{\emph{v},k} + \mathcal{R}^{N,n}_{\emph{x},k} }{\gamma +\mathcal{Q}^{N,n}_{\emph{v},k}  + \mathcal{Q}^{N,n}_{\emph{x} ,k}} \right), 
	\end{equation}
where $\gamma>0$, 
\begin{equation}\label{eq:terms_in_B}
	\begin{split}
	&\mathcal{R}^{N,n}_{\emph{v},k}   = \alpha_\emph{v}  (\bar{v}_{y,k}^n + h \bar{E}_{y,k}^n -\hat{v}_{y_k}) \bar{v}^n_{x,k} + \frac{\beta_\emph{v} }{N_k} \sum_{i=1}^{N_k} \left[ (v_{y_i}^n + h E_{y_i}^n - \bar{v}_{y,k}^n)v_{x_i}^n \right],\\
	&\mathcal{R}^{N,n}_{\emph{x},k}   =   \alpha_\emph{x} (\bar{y}_{k}^n + h(\bar{v}_{y,k}^n+ h\bar{E}_{y,k}^n) -\hat{y}_k) \bar{v}^n_{x,k} + \frac{\beta_\emph{x}}{N_k} \sum_{i=1}^{N_k} \left[ (y_i^n + h(v_{y_i}^n + h E_{y_i}^n)-\bar{y}_{k}^n)v_{x_i}^n \right],\\
	&\mathcal{Q}^{N,n}_{\emph{v},k}  =  h \left( \alpha_\emph{v} (\bar{v}_{x,k}^n)^2 + \frac{\beta_\emph{v}}{N_k} \sum_{i=1}^{N_k} (v_{x_i}^n)^2\right) ,\quad
	\mathcal{Q}^{N,n}_{\emph{x},k} =   h^2 \left( \alpha_\emph{x} (\bar{v}_{x,k}^n)^2 + \frac{\beta_\emph{x}}{N_k} \sum_{i=1}^{N_k} (v_{x_i}^n)^2\right),
	\end{split}
\end{equation}
and with $\mathbb{P}_{[-M,M]}(\cdot)$ denoting the projection over the interval $[-M,M]$. In the limit $h\to 0$ the control at the continuous level reads, 
\begin{equation}\label{eq:L2_control_continuos}
	B_k(t) = \mathbb{P}_{[-M,M]} \left(\frac1\gamma\left(	\mathcal{R}^N_{\emph{v},k} (t) + 	\mathcal{R}^{N}_{\emph{x},k} (t)\right)\right),
\end{equation}
with 
\begin{equation}
	\begin{split}
		&	\mathcal{R}^{n}_{\emph{v},k} (t) = \alpha_\emph{v} (\bar{v}_{y,k}(t) -\hat{v}_{y_k})\bar{v}_{x,k}(t) + \frac{\beta_\emph{v}}{N_k} \sum_{i\in C_k} (v_{y_i}(t) -\bar{v}_{y,k}(t)) v_{x_i}(t),\\
		&	\mathcal{R}^{n}_{\emph{x},k} (t) = \alpha_\emph{x} (\bar{y}_{k}(t)-\hat{y}_{k})\bar{v}_{x,k}(t) + \frac{\beta_\emph{x}}{N_k} \sum_{i\in C_k} (y_{i}(t) -\bar{y}_{k}(t)) v_{x_i}(t).
	\end{split}
\end{equation}
\end{proposition}

\begin{proof}
	We introduce the augmented Lagrangian
	\begin{equation}\label{eq:lagr_L2}
		\mathcal{L}(B_k,\lambda_k) = \mathcal{J}_k^{N,h}(B_k)+\lambda_k(|B_k|-M),
	\end{equation}
	with $\lambda_k(t)$ the Lagrangian multiplier.  
	Then, for each $k=1,\ldots,N_c$, we solve the optimality system 
\begin{equation}\label{eq:Lagr_system}
	\begin{cases}
		\partial_{B_k}\mathcal{L}(B_k,\lambda_k)=0\\
		\left(\partial_{\lambda_k}\mathcal{L}(B_k,\lambda_k) = 0\ \text{ if }\lambda_{k}> 0\right) \text{ or } \left( \partial_{\lambda_k}\mathcal{L}(B_k,\lambda_k) <0 \text{ if } \lambda_{k} = 0\right).
	\end{cases}
	\end{equation} 
	From the first equation in \eqref{eq:Lagr_system} we get 
	\begin{equation}\label{eq:partial_Bk}
		\begin{split}
			\partial_{B_k} \mathcal{L} =&  h \alpha_\emph{v} \left( \bar{v}_{y,k}^n -h  \bar{v}_{x,k}^n B_k + h  \bar{E}_{y,k}^n -\hat{v}_{y_k} \right) \left( -h \bar{v}_{x,k}^n\right) + \\
			& + \frac{h\beta_\emph{v}}{N_k}\sum_{i\in C_k}  \left\lbrace  \left(  v_{y_i}^n -h  v_{x_i}^n B_k + h  E_{y_i}^n  -\bar{v}_{y,k}^n \right) \left( - hv_{x_i}^n\right) \right\rbrace  +\\&+h\alpha_\emph{x} \left(\bar{y}_{k}^n+ h(\bar{v}_{y,k}^n -h  \bar{v}_{x,k}^n B_k + h  \bar{E}_{y,k}^n) -\hat{y}_{k} \right) \left( - h^2\bar{v}_{x,k}^n\right) + \\
			& + \frac{h\beta_\emph{x}}{N_k}\sum_{i\in C_k}  \left\lbrace  \left( y_i^n+h( v_{y_i}^n -h  v_{x_i}^n B_k + h  E_{y_i}^n) -\bar{y}_{k}^n   \right) \left( -h^2 v_{x_i}^n\right) \right\rbrace   + \\& +h\gamma B_k  +\lambda_k\sign(B_k) = 0.
		\end{split}
	\end{equation} 
	Then, if $\lambda_k > 0$, from  $\partial_{\lambda_k} \mathcal{L}(B_k,\lambda_k)=0$, we get $\vert B_k \vert=M$. Consequently, by considering the two cases $B_k=M$ and $B_k=-M$ separately, and by assuming the parameters to scale as in \eqref{eq:scaling}, we get by setting $B_k=-M$ in \eqref{eq:partial_Bk}, and since $\lambda_k>0$
	\begin{equation}
		\begin{split}
			\frac{\lambda_k}{h^2} =  &\alpha_\emph{v} \left( \bar{v}_{y,k}^n + h  \bar{v}_{x,k}^n M + h  \bar{E}_{y,k}^n -\hat{v}_{y_k} \right) \left( -\bar{v}_{x,k}^n\right) + \\
			& + \frac{\beta_\emph{v}}{N_k}\sum_{i\in C_k}  \left\lbrace  \left(  v_{y_i}^n +h  v_{x_i}^n M + h  E_{y_i}^n   -\bar{v}_{y,k}^n \right) \left( -v_{x_i}^n\right) \right\rbrace  +\\&+\alpha_\emph{x} \left( \bar{y}_{k}^n + h(\bar{v}_{y,k}^n +h  \bar{v}_{x,k}^n M + h  \bar{E}_{y,k}^n) -\hat{y}_{k} \right) \left( -\bar{v}_{x,k}^n\right) + \\
			& + \frac{\beta_\emph{x}}{N_k}\sum_{i\in C_k}  \left\lbrace  \left( y_i^n+h( v_{y_i}^n +h  v_{x_i}^n M + h  E_{y_i}^n)-\bar{y}_{ k}^n   \right) \left( - v_{x_i}^n \right) \right\rbrace   - \gamma M > 0,
		\end{split}
	\end{equation}  
	from which we get 
	\begin{equation}
\frac{\mathcal{R}^{n}_{\textrm{v},k} + \mathcal{R}^{n}_{\textrm{x},k}  }{\gamma +\mathcal{Q}^{n}_{\textrm{v},k}  + \mathcal{Q}^{n}_{\textrm{x},k}} < -M,
	\end{equation} 
	with 
	$\mathcal{R}^{n}_{\textrm{v},k} $, $\mathcal{R}^{n}_{\textrm{x},k} $, $\mathcal{Q}^{n}_{\textrm{v},k} $, $\mathcal{Q}^{n}_{\textrm{v},k} $ defined as in \eqref{eq:terms_in_B}. 
	If conversely $\lambda_k > 0$ and $B_k=M$, by scaling the parameters as in \eqref{eq:scaling}, following a similar argument we have
	\begin{equation}
\frac{\mathcal{R}^{n}_{\textrm{v},k} + \mathcal{R}^{n}_{\textrm{x},k}  }{\gamma +\mathcal{Q}^{n}_{\textrm{v},k}  + \mathcal{Q}^{n}_{\textrm{x},k}}> M.
	\end{equation}
	If finally $\lambda_k=0$, and by assuming the parameters to scale as in \eqref{eq:scaling}, from the first equation in \eqref{eq:Lagr_system} we get 
	\begin{equation}\label{eq:control_L2} 
		B_k =\frac{\mathcal{R}^{n}_{\textrm{v},k} + \mathcal{R}^{n}_{\textrm{x},k}  }{\gamma +\mathcal{Q}^{n}_{\textrm{v},k}  + \mathcal{Q}^{n}_{\textrm{x},k}},
	\end{equation}
	and from the second equation in \eqref{eq:Lagr_system}, i.e. from $\partial_{\lambda_k} \mathcal{L}(B,\lambda)<0$, we get 
	\begin{equation}
		-M< B_k < M.
	\end{equation}
	All in all we get $B_k$ defined as in \eqref{eq:L2_control_space_velocity}. Finally, in the limit $h\to 0$ we recover equation \eqref{eq:L2_control_continuos}.  
\end{proof}
The expression obtained in \eqref{eq:L2_control_space_velocity} is the suboptimal control we will plug into the second order discretized Vlasov equation \eqref{eq:second_order_first_stage}-\eqref{eq:second_order_second_stage}-\eqref{eq:second_order_scheme}  to steer particles into the desired position.
\begin{remark}
		The feedback control obtained in \eqref{eq:L2_control_space_velocity} can be equivalently derived as an instantaneous control of the following functional on the interval $[t,t+h]$, without rescaling the parameters as in \eqref{eq:scaling},
	 \begin{equation}\label{eq:J_alt}
	 	\begin{split}
	 	\mathcal{J}^{N,h}_k(B_k) &= 	\frac{\alpha_\emph{v}}{2}  \Vert \bar{v}_{y,k}(t+h) -\hat{v}_{y_k}\Vert^2 + 
	 		 \frac{\beta_\emph{v}}{2 N_k} \sum_{i\in C_k} \Vert v_{y_i}(t+h)-\bar{v}_{y,k}(t) \Vert^2+\cr
	 		& +\int_t^{t+h} \left(\frac{\alpha_\emph{x}}{2}  \Vert \bar{v}_{y,k}(\tau)      - \hat{\mathcal{V}}_k (\tau) \Vert^2 + \frac{\beta_\emph{x}}{2 N_k} \sum_{i\in C_k} \Vert v_{y_i}(\tau) - \bar{\mathcal{V}}_{k,i} (\tau) \Vert^2\right)\,d\tau\cr
	 		&\qquad+\frac{\gamma}{2}\int_t^{t+h} \Vert B_k(\tau) \Vert^2\,d\tau,
	 	\end{split}
	 \end{equation}	
where we introduced the term
 \[
 \hat{\mathcal{V}}_{k} (\tau) =  \frac{\hat y_k -\bar y_k(\tau)}{h},\quad  \bar{\mathcal{V}}_{k,i} (\tau) = \frac{\bar y_k(\tau) -y_i(\tau)}{h}.
 \]
Hence, we consider the following discretization
	 \begin{equation}\label{eq:disc_J_alt}
	\begin{split}
		\mathcal{J}^{N,h}_k(B_k) & = 	\frac{\alpha_\emph{v}}{2}  \Vert \bar{v}_{y,k}^{n+1}-\hat{v}_{y_k}\Vert^2 + 
		\frac{\beta_\emph{v}}{2 N_k} \sum_{i\in C_k} \Vert v_{y_i}^{n+1}-\bar{v}_{y,k}^n \Vert^2+\cr
		& +\frac{h\alpha_\emph{x}}{2}  \Vert \bar{v}_{y,k}^{n+1}     - \hat{\mathcal{V}}_k ^n \Vert^2 + \frac{h\beta_\emph{x}}{2 N_k} \sum_{i\in C_k} \Vert v_{y_i}^{n+1} - \bar{\mathcal{V}}_{k,i} ^n \Vert^2+\frac{h\gamma}{2} \Vert B_k \Vert^2,
	\end{split}
\end{equation}	
and we observe that the following reformulation is viable for the third term
\[
\begin{split}
&\frac{h\alpha_\emph{x}}{2} \Vert \bar{v}_{y,k}^{n+1}     - \hat{\mathcal{V}}_k ^n \Vert^2 
=\frac{h\alpha_\emph{x}}{2} \left \Vert \frac{\bar{y}_k^{n+1}-\bar{y}_k^{n}}{h}-      \frac{\hat y^n_k - \bar y^n_k}{h}\right\Vert^2
 =\frac{\alpha_\emph{x}}{2h}\left \Vert \bar{y}_k^{n+1}-\tilde y^n_k \right\Vert^2.
 \end{split}
\]
The same argument holds for the forth term of \eqref{eq:disc_J_alt}.
\end{remark}

\subsection{Optimize then discretize  (OtD)}
In the previous section we proceeded by finding instantaneous control from the optimization of the discretized dynamics based on the PIC method. In what follows, we instead derive an instantaneous control directly from the continuous problem \eqref{eq:continuos_pb}, showing successively consistency with the discretized dynamics by selecting an appropriate time discretization.

A first step toward the solution of the PDE-constrained optimal control problem \eqref{eq:continuos_pb} is the derivation of first order necessary conditions, as done for example in ^^>\cite{knopf2020optimal,bartsch2023controlling}.  
Here we consider the evolution of the Vlasov-Poisson system \eqref{eq:vlasov_poisson} restricted to $\Omega_k$ in the short time frame $[t,t+h]$, and since we are interested in a instantaneous control we neglect the contribution of the Poisson equation.
Hence we write the associated Lagrangian as follows
\begin{equation}\label{eq:lagrangian_cont}
	\begin{split}
		\mathcal{L}(B_k,f_k,q_k) &=\int_{t}^{t+h}\left(\sum_{\ell\in\left\{\textrm{x,v}\right\}}\mathcal{D}_k(f_k,\psi_{\ell})(\tau) + \frac{\gamma}{2} \Vert B_k(\tau)\Vert^2 \right)\, d\tau
		\cr
		&	+\int_{t}^{t+h}\int_{\Omega_k}q_k\left(\partial_t f_k + \vv\cdot \nabla_{x} f_k + (\EE(t,\xx)+\vv\times \BB_k)\cdot\nabla_{v}f_k\right)  d\xx d\vv\, d\tau.
	\end{split}
\end{equation}
where $q_k=q_k(t,\xx,\vv)$ denotes the Lagrangian multiplier.

Necessary conditions associated to \eqref{eq:continuos_pb} can be derived as first variations of \eqref{eq:lagrangian_cont}, similarly to^^>\cite{knopf2020optimal}.
Hence, we retrieve the following adjoint equation
\begin{align} \label{eq:optimality_cont}	
	&\partial_tq_k +  \vv \cdot \nabla_{x} q_k+  (\EE(t,\xx)+ \vv \times \BB_k(t))\cdot \nabla_{v} q_k  = \sum_{\ell}\ \mathcal{S}_k(t,\psi_\ell),
\end{align}
which  is solved backward  in time with terminal condition $q_k(t+h,\xx,\vv)=0$, and where
the source terms $\mathcal{S}_k(t,\psi_\ell)$ are the variations of \eqref{eq:function_D_ell} with respect to the density $f_k(t,\cdot)$. 
\begin{align}\label{eq:S_ell}
	\mathcal{S}_k(t,\psi_\ell)& =  \alpha_\ell \psi_\ell(x,v)\left(m_k[\psi_\ell](t) - \hat \psi_{\ell,k}\right)+\frac{\beta_\ell}{2}\|m_k[\psi_\ell](t)-\psi_\ell(\xx,\vv)\|^2.
\end{align}

The optimality condition associated to the variation \eqref{eq:lagrangian_cont} with respect to the magnetic field writes
\begin{align}\label{eq:control_cont}
	B_k(t) = \mathbb{P}_{[-M,M]}\left(\frac{1}{\gamma}\int_{\Omega_k}\vv\times \nabla_{v}q_k(t,\xx,\vv) f_k(t,\xx,\vv)\, d\xx\,d\vv\right),
\end{align}
with $\mathbb{P}_{[-M,M]} $ the projection operator previously defined.

In order to retrieve consistency with the instantaneous control in \eqref{eq:L2_control_continuos} we need to discretize the optimality system in time in a coherent way with respect to the forward dynamics of the particle scheme \eqref{eq:explicit_scheme}, since in general the two approaches do not commute^^>\cite{sanz2016symplectic,hager2000runge}.
Then we introduce the time splitting scheme for the adjoint equation as follows
\begin{equation}\label{eq:adjoint_disc}
\begin{aligned}
	q_k^*&= 
	q_k^{n+1}+  h( \EE^n(\xx)+v\times \BB_k^n)\cdot\nabla_{v} q_k^{n+1}  - h\sum_{\ell}\ \mathcal{S}_k(t^{n+1},\psi_\ell)\\
	q_k^n &= q_k^* +h\vv\cdot\nabla_x q_k^*
\end{aligned}
\end{equation}
with terminal condition $q_k^{n+1}(\xx,\vv)= 0$,
and 
\begin{align}\label{eq:control_cont_disc}
	B_k(t^n) = \mathbb{P}_{[-M,M]}\left(\frac{1}{\gamma}\int_{\Omega_k}\vv\times \nabla_{v}q_k^n(\xx,\vv) f^n_k(\xx,\vv)\, d\xx,d\vv\right).
\end{align}
Since the terminal conditons is zero, we have that the first step of the splitting scheme \eqref{eq:adjoint_disc} corresponds to
\begin{equation*}
		q_k^*= -h\sum_{\ell}\ \mathcal{S}_k(t^{n+1},\psi_\ell) = -h\left(\mathcal{S}_k(t^{n+1},v_y)+\mathcal{S}_k(t^{n+1},y)\right).
\end{equation*}
Then, the term in the integral of \eqref{eq:control_cont_disc} can be rewritten as 
\[
\begin{split}
	 \vv\times \nabla_{v}q_k^n&=  \vv\times \nabla_{v}\left(q_k^*+h\vv\cdot \nabla_x q_k^*\right) \cr
	 &= -h\sum_{\ell}\vv\times \nabla_{v}\left(\mathcal{S}_k(t^{n+1},\psi_\ell)+h\vv\cdot \nabla_x \mathcal{S}_k(t^{n+1},\psi_\ell)\right),
\end{split}
\]
and considering the choice for $\psi_\ell(x,v)$ designed in \eqref{eq:states}, we have
\[
\begin{split}
	\vv\times \nabla_{v}q_k^n& = hv_x\partial_{v_y} \mathcal{S}_k(t^{n+1},v_y) +h^2v_x\partial_{y} \mathcal{S}_k(t^{n+1},y).
\end{split}
\]
where we can compute directly from \eqref{eq:S_ell} the partial derivatives
\begin{align*}
\partial_{v_y} \mathcal{S}_k(t^{n+1},v_y)&= \alpha_\textrm{v} \left(m_k(f_k,v_y)(t^{n+1}) - \hat v_{y,k}\right)+\beta_\textrm{v}\left(v_y-m_k(f_k,v_y)(t^{n+1})\right),\\
\partial_{y} \mathcal{S}_k(t^{n+1},y)&= \alpha_\textrm{x} \left(m_k(f_k,y)(t^{n+1}) -\hat y_{k}\right)+\beta_\textrm{x}\left(y-m_k(f_k,y)(t^{n+1})\right).
\end{align*}
Finally introducing the scaling \eqref{eq:scaling} over the paramters $\alpha_\textrm{x},\beta_\textrm{v}$ and $\gamma$ and letting $h\to 0$ we can retrieve the instantaneous control from \eqref{eq:control_cont_disc} as follows
\begin{align}\label{eq:control_cont_disc_2}
	B_k(t) = \mathbb{P}_{[-M,M]}\left(\int_{\Omega_k}\frac{1}{\gamma}\left(v_x	\mathcal{R}_{\textrm{v},k}(t,v_y) + v_x	\mathcal{R}_{\textrm{x},k}(t,y)\right)f_k(t,\xx,\vv)\, d\xx\,d\vv\right).
\end{align}
where $\mathcal{R}_\ell$ are defined as
\begin{align*}
	\mathcal{R}_{\textrm{v},k}(t,v_y)&:=\partial_{v_y} \mathcal{S}_k(t,v_y),\qquad	\mathcal{R}_{\textrm{x},k}(t,y):=\partial_{y} \mathcal{S}_k(t,y).
\end{align*}
which corresponds to the continuous version of the particle feedback control in \eqref{eq:L2_control_continuos}.

\begin{remark}
We remark that in the Lagrangian \eqref{eq:lagrangian_cont} we neglected the contribution of the Poisson equation, which depends itself on the density $f_k(t,\xx,\vv)$, and which would provide in the adjoint equation an additional source term as follows
\begin{align} \label{eq:optimality_cont2}	
	&\partial_tq_k +  \vv \cdot \nabla_{x} q_k+  (\EE(t,\xx)+ \vv \times \BB_k(t))\cdot \nabla_{v} q_k  = \sum_{\ell}\ \mathcal{S}_k(t,\psi_\ell) + \mathcal{P}(f_k,q_k)(t,\xx),
\end{align}
where $ \mathcal{P}(f_k,q_k)(t,\xx)$ in the two dimensional setting corresponds to the following non local term
\begin{align}
	\mathcal{P}(f_k,q_k)(t,\xx)& =- \mu\int_{\Omega_k} \frac{\xx-\yy}{\|\xx-\yy\|^2}\cdot \nabla_{v} f_k(t,\yy,{\bf {w}})q_k(t,\yy,{\bf {w}}) d\yy d{\bf {w}},
\end{align}
with $\mu>0$ a normalization constant. We refer to^^>\cite{knopf2020optimal} for detail computation over the complete Lagrangian. However, this additional term does not affect the derivation of the instantaneous control \eqref{eq:control_cont_disc_2} since its contribution vanishes at the level of the discretization used in \eqref{eq:adjoint_disc} due to the zero terminal condition.
\end{remark}

\section{Numerical experiments}\label{sec:Numerical_tests}
In this section we present several numerical results to test the behavior of the Vlasov system under the feedback control \eqref{eq:L2_control_space_velocity} induced by the external magnetic field. We simulate the dynamics by means of the second order PIC method in \eqref{eq:second_order_first_stage}-\eqref{eq:second_order_second_stage}-\eqref{eq:second_order_scheme}, and we study the capability of the control to achieve the desired confinement in two classical settings. The first problem concerns the simulation of a regular flow, while the second problem regards a well-known type of instability arising in fluids. We expect the first case to be easier to confine the particles than in the second case due to higher mixing. For both situations, we show the case with and without control, where this latter is constructed in such a way to limit the accumulation of the mass close to the boundaries of the domain and correspondingly to avoid the thermal energy to grow excessively. In Appendix \ref{sec:appendix} some convergence numerical tests for the Vlasov-Poisson system are shown for code validation purposes.

\subsection{Two stream plasma test case}
We consider a two dimensional domain $(x,y)\in [0,40]\times[-1.5,1.5]$ and we impose periodic boundary conditions in $x$ and reflective boundary conditions in $y$. 
\begin{figure}[h!]
	\centering
	\includegraphics[width=0.328\linewidth]{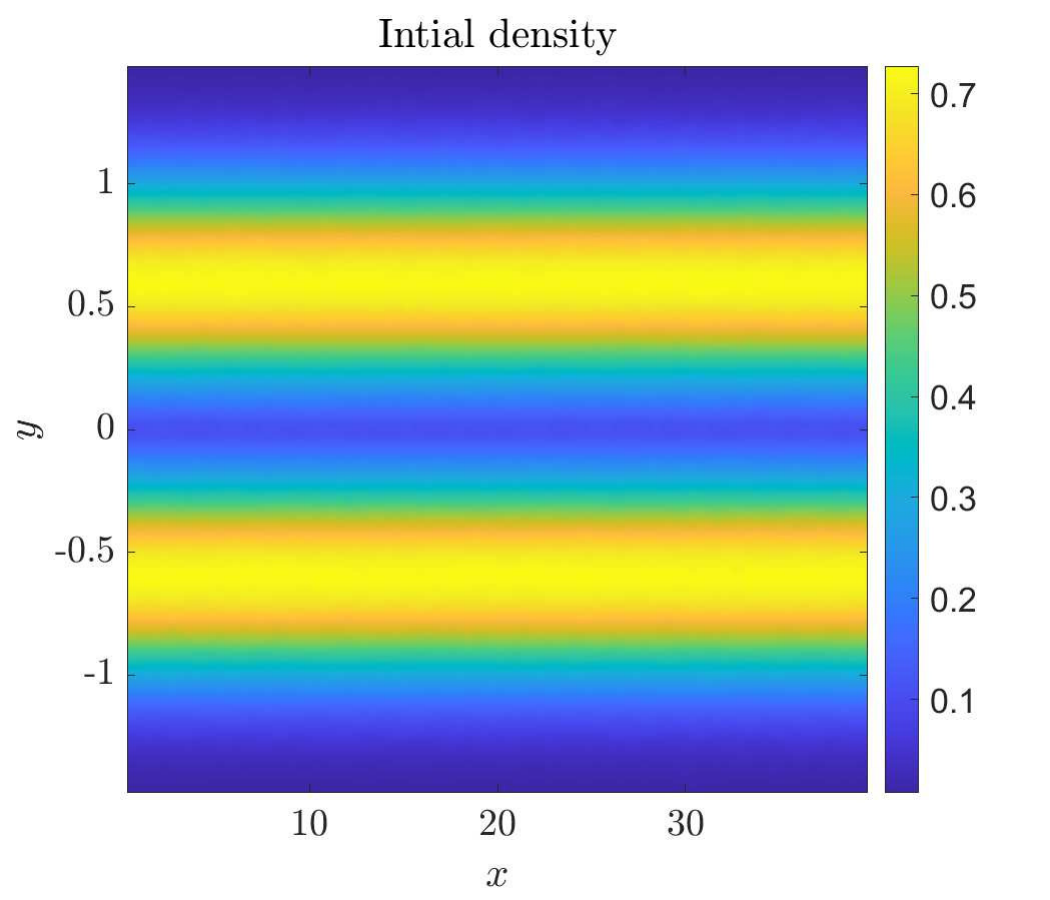}
	\includegraphics[width=0.328\linewidth]{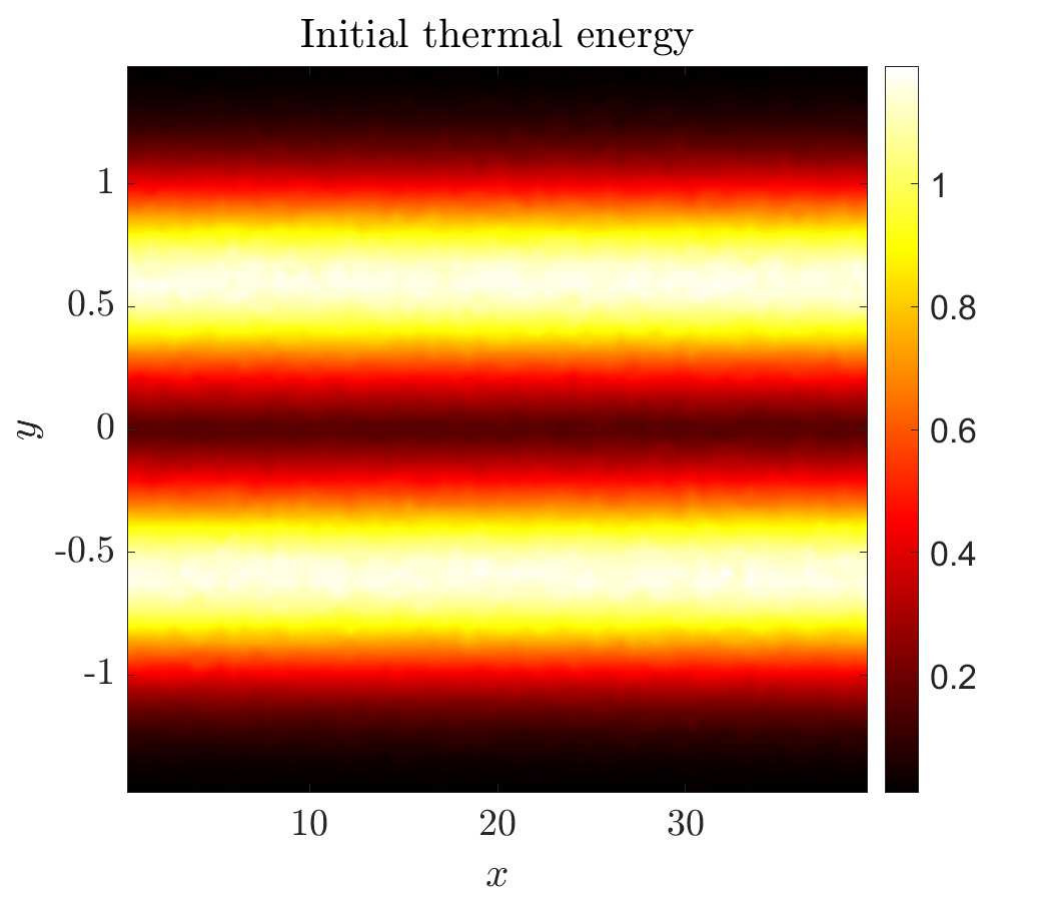}\\ 
	\includegraphics[width=0.328\linewidth]{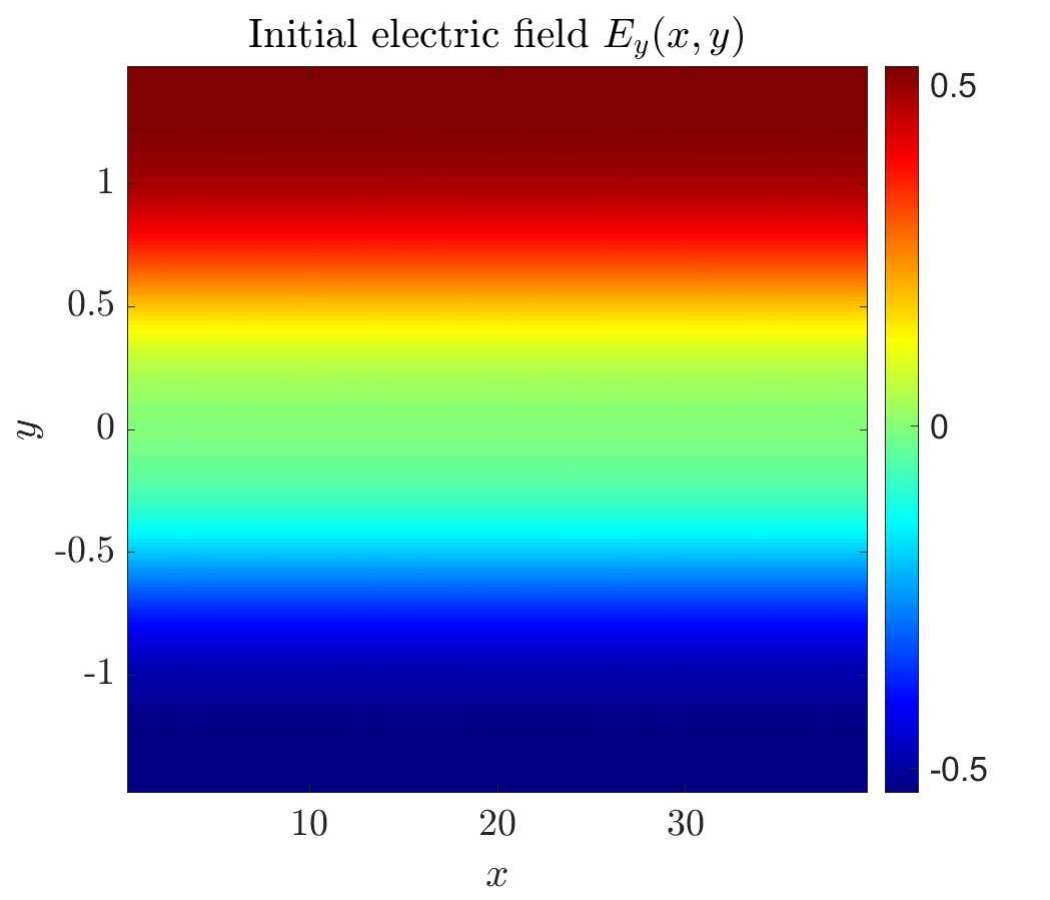}
	\includegraphics[width=0.328\linewidth]{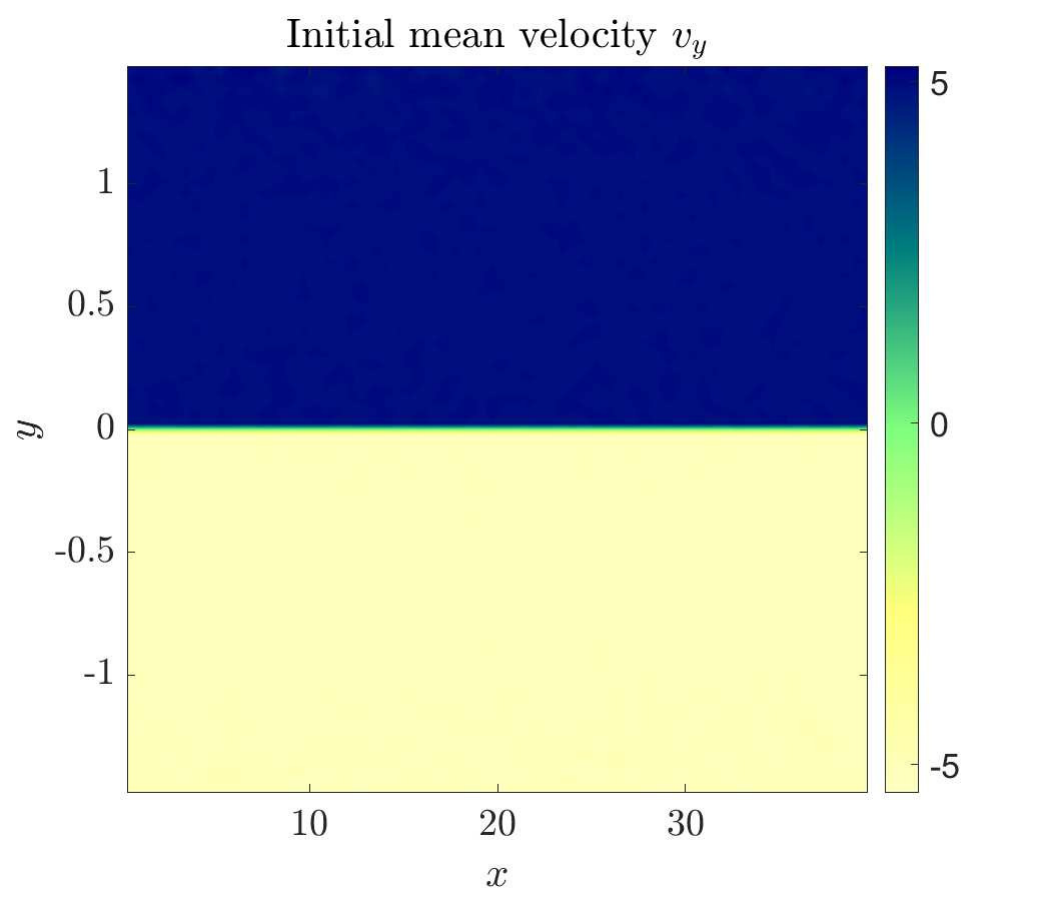}
	\caption{Two stream plasma test. Initial density top left. Initial thermal energy (top right). Initial electric field (bottom left) and velocity (bottom right) in the $y$-direction.}
	\label{fig:ic_ex}
\end{figure}
As initial density we take the sum of two Gaussian distributions in space 
\begin{equation}\label{eq:rh0_example} 
\begin{split}
	\rho_0^\pm(\xx) =\frac{1}{\sqrt{2\pi\sigma}} \exp\left(-\frac{(y\pm c_y)^2}{2\sigma^2}\right),
\end{split}
\end{equation}
with $c_y = 1$ and $\sigma = 0.3$ and where the apex $\pm$ indicates whether the particles are lying on the positive or the negative part of the domain in the $y$-direction. We assume a plasma in thermodynamic equilibrium at $t=0$ and distributed accordingly to 
\begin{equation}\label{eq:f0} 
	f_0(\xx,\vv) = f_0^+(t,\xx,\vv) \chi(y\geq0)+ f_0^-(t,\xx,\vv)\chi(y<0),
\end{equation}
with
\begin{equation}\label{eq:fo_p_m}
	f_0^\pm(\xx,\vv) =\frac{\rho_0^\pm(\xx)}{ 2\pi T_0(\xx)}\exp\left(-\frac{(v_y\pm u_y)^2 + v_x^2}{2T_0(\xx)}\right),
\end{equation}
where the velocity in the $y-$direction is $u_y = \pm5$, denoting that the particles are in average pointing to the walls of the device at initial time. The initial temperature is 
\begin{equation}\label{eq:temp_sigma}
	\begin{split}
	T_0(\xx) =  1.5+&0.1\sin\left(\frac{2\pi(y-0.3)}{1.2}\right) \chi(y>=0.3)+\\&-0.1\sin\left(\frac{2\pi(y-0.3)}{1.2}\right)\chi(y<-0.3).
	\end{split} 
\end{equation}
We sample $N=10^7$ particles from the initial distribution with position $(x_i^0,y_i^0)$, and velocity $(v_{x_i}^0,v_{y_i}^0)$ accordingly to \eqref{eq:fo_p_m}. In Figure \ref{fig:ic_ex} the initial configuration is shown as well as the initial self-induced electric field. 

We now start by assuming to have no control on the system, we instead set a constant external magnetic field $B(t,\xx)=1$ acting on the particles. We then let the dynamics to evolve up to time $T=100$, choosing $h=0.001$ as time step. We also take $m_x\times m_y$, with $m_x=m_y = 64$ cells for the space discretization, which are needed to reconstruct the density field from the particle distribution and to compute the electric field. 
\begin{figure}[h!]
	\centering
	\includegraphics[width=0.328\linewidth]{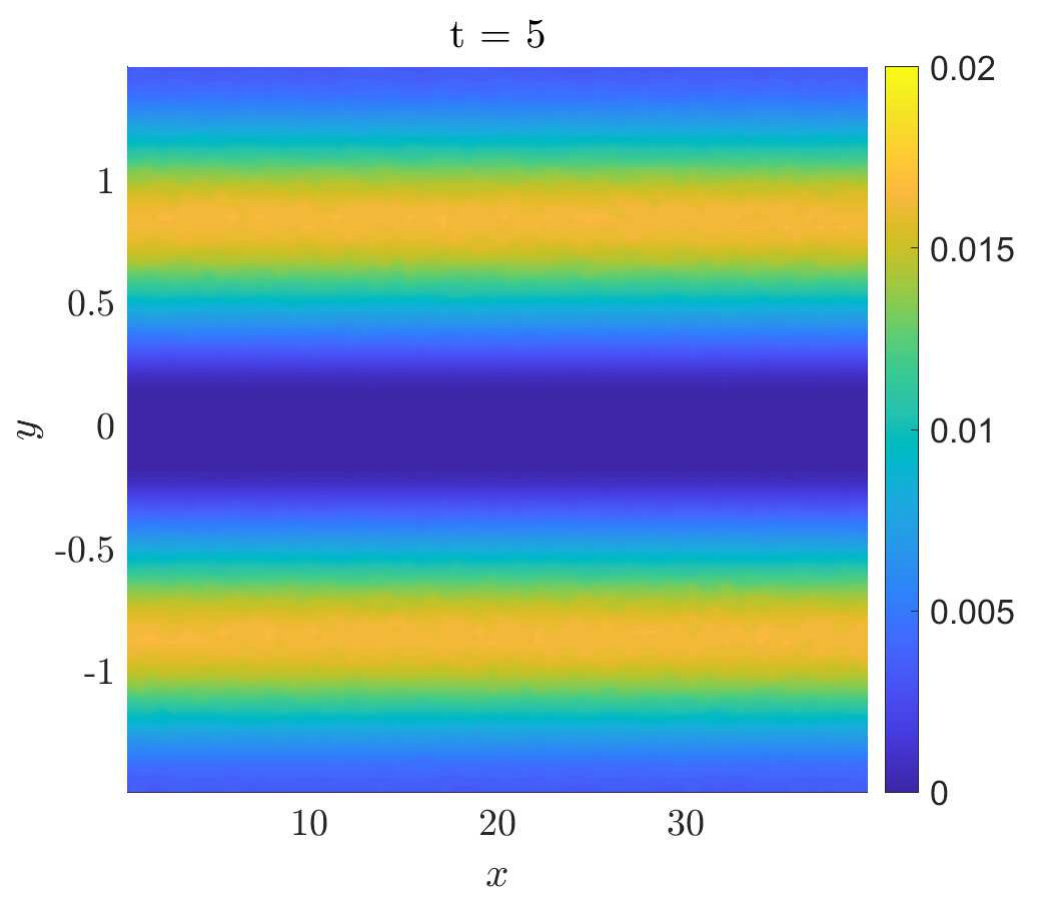}
	\includegraphics[width=0.328\linewidth]{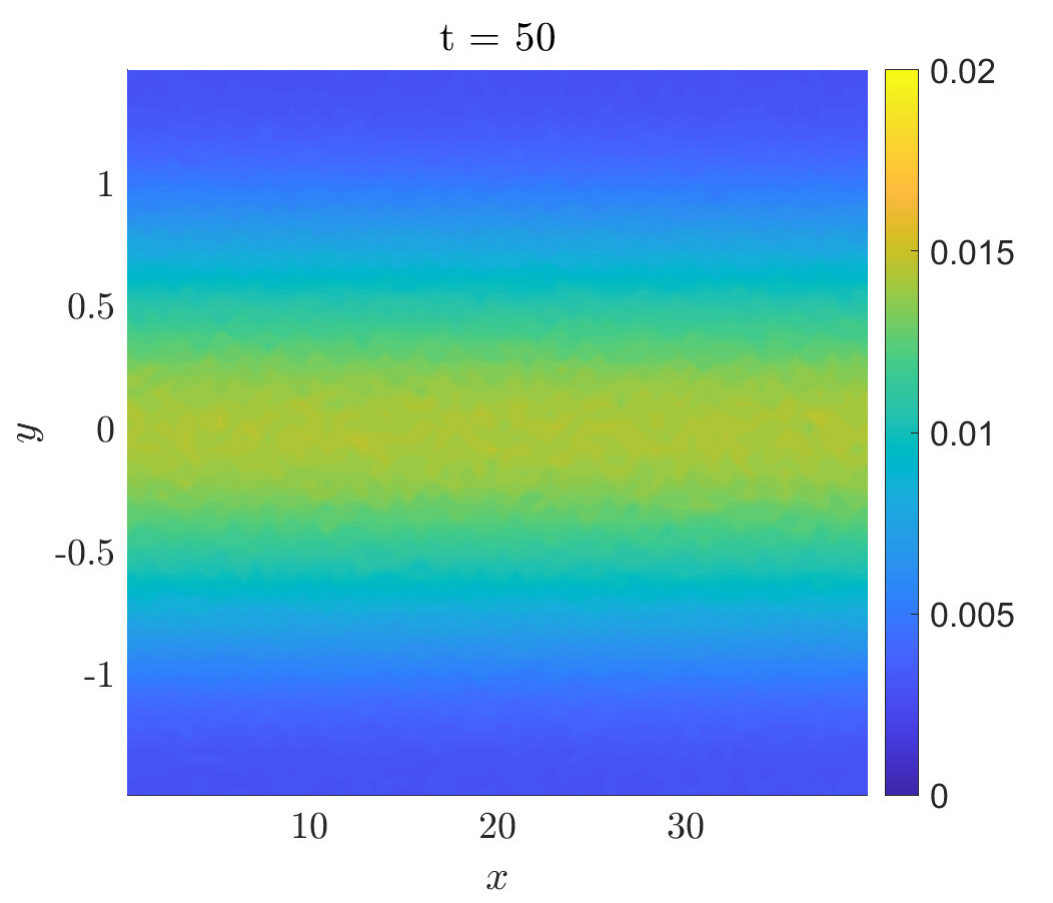}
	\includegraphics[width=0.328\linewidth]{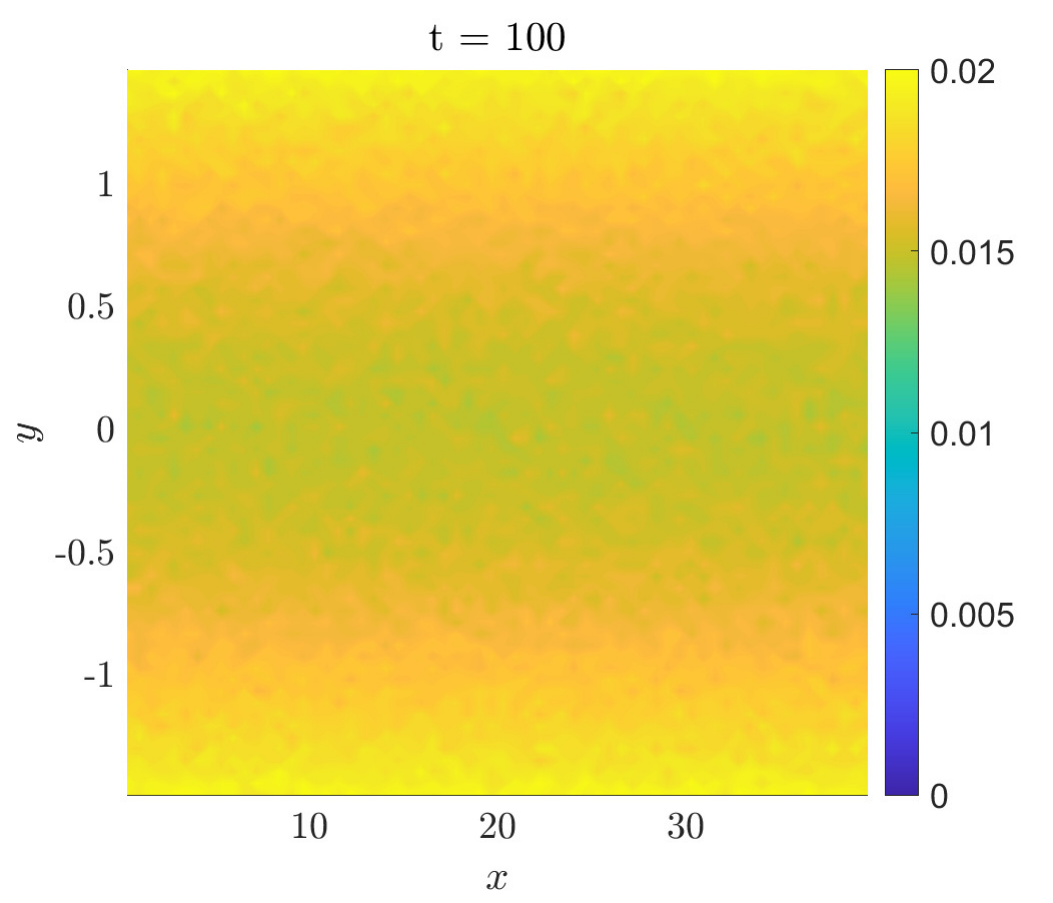}\\
	\includegraphics[width=0.328\linewidth]{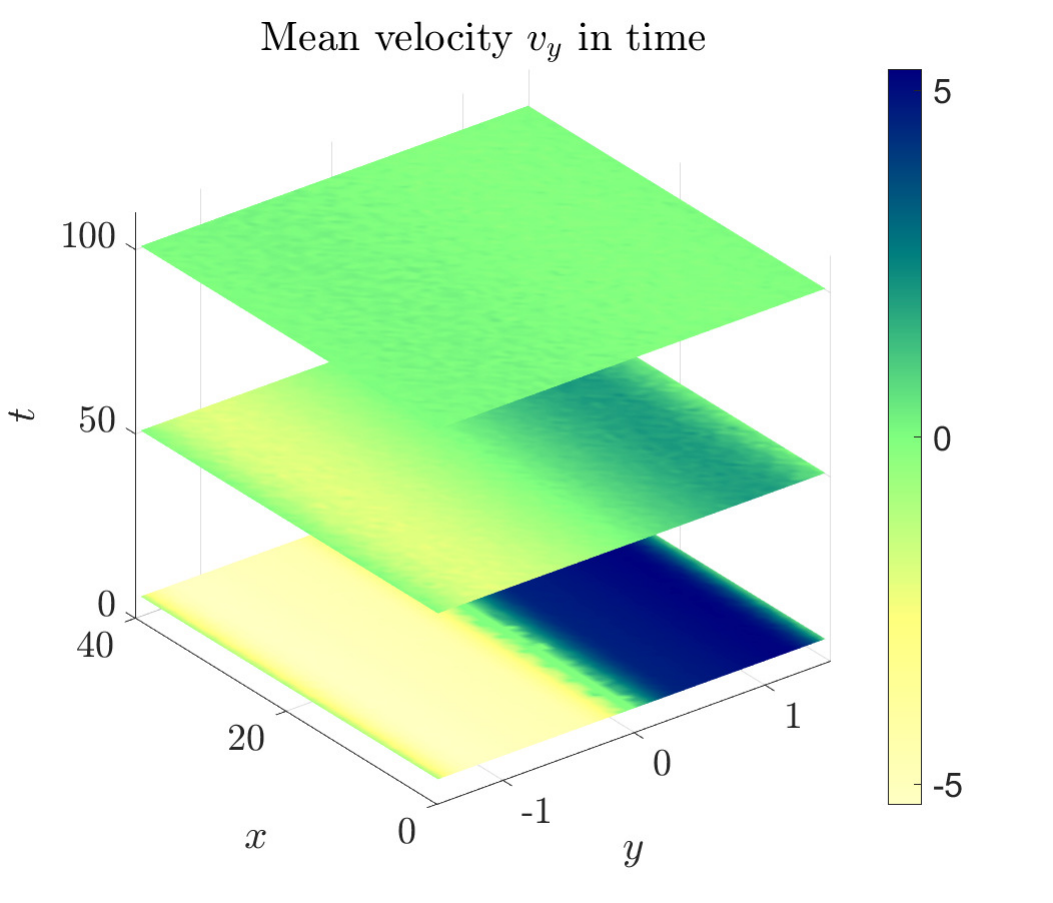}
	\includegraphics[width=0.328\linewidth]{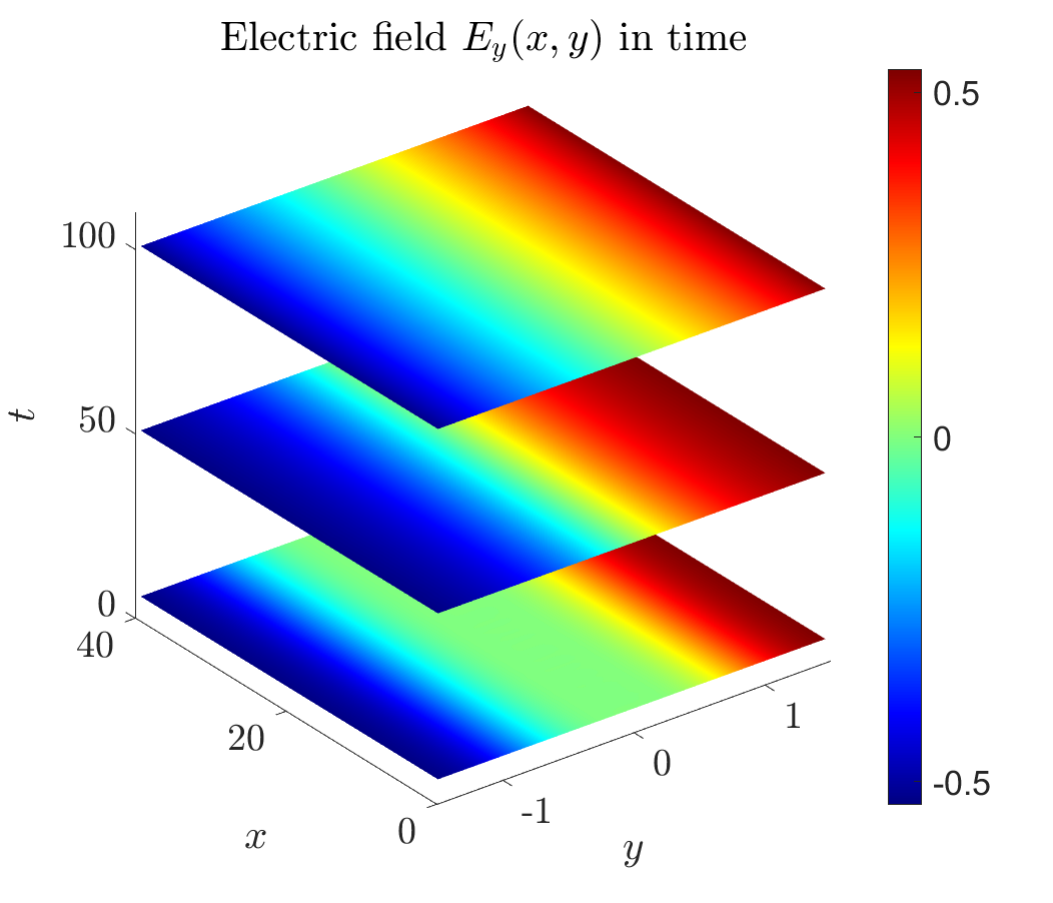}
	\caption{Two stream plasma test with constant magnetic field $B=1$. First row: mass density taken at time $t=5$, $t =50$ and $t=100$. Second row: mean velocity in the $y$-direction (on the left), electric field in the $y$-direction (on the right).  }
	\label{fig:no_control_ex}
\end{figure}

Given the above detailed configuration, in Figure \ref{fig:no_control_ex}, in the first row there are shown three different snapshots of the density at time $t=5$, $t =50$ and $t=100$. 
The dynamics is such that particles move from the centre to the $y-$boundaries and they are then reflected back again toward the centre due to the choice of the reflective boundary conditions in $y$. After some time a spread of the density in all domain is observed (see again Figure \ref{fig:no_control_ex} on the right). In the second row we plot the mean velocity and the electric field in the $y$-direction over time. 
In order to test after the effectiveness of the control, we compute also in this uncontrolled case, the percentage of mass hitting the walls over time as well as the thermal energy of the cells close to the boundaries. The mass at the boundaries at time $t^n$ is computed by 
\begin{equation}\label{eq:mass_boundary}
	\rho_{b}(t^n) =  \int_{\Omega_b} f^N(t^n,\xx,\vv) d\xx d\vv, 
\end{equation}
with $f^N(\cdot)$ defined as in \eqref{eq:approx_density} and $\Omega_b=[-1.5,-0.625] \cup [0.625,1.5]$ a region close to the $y$ boundaries where, in particular, the width in the $y-$direction corresponds to the size of one cell $\Delta_y$.  
The thermal energy at the boundaries at time $t^n$ is instead computed as 
\begin{equation}\label{eq:energy}
	\begin{split}
		&\rho_{b}(t^n)T_b(t^n) =\frac{1}{2} \int_{\Omega_b} \vert \vv\vert^2 f^N(t^n,\xx,\vv) d\xx d\vv+\\&-\frac{1}{2\rho_b(t^n)}\left(\left( \int_{\Omega_b} v_x f^N(t^n,\xx,\vv) d\xx d\vv\right)^2+\left( \int_{\Omega_b} v_y f^N(t^n,\xx,\vv) d\xx d\vv\right)^2\right), 
	\end{split}
\end{equation}
with $f^N(\cdot)$ defined as before and again $\Omega_b=[-1.5,-0.625] \cup [0.625,1.5]$.

In Figure \ref{fig:no_control_mass_energy_ex} on the left the density percentages and, on the right, the thermal energy at the boundaries as a function of time are respectively shown. In the images, the mass density $\rho_i(t)$ represents the internal mass, i.e. $\rho_i(t)=\rho_{tot}-\rho_b(t)$. As expected when particles hits the boundaries, the local energy increases while when far from them it drops down. 
\begin{figure}[h!]
	\centering
	\includegraphics[width=0.35\linewidth]{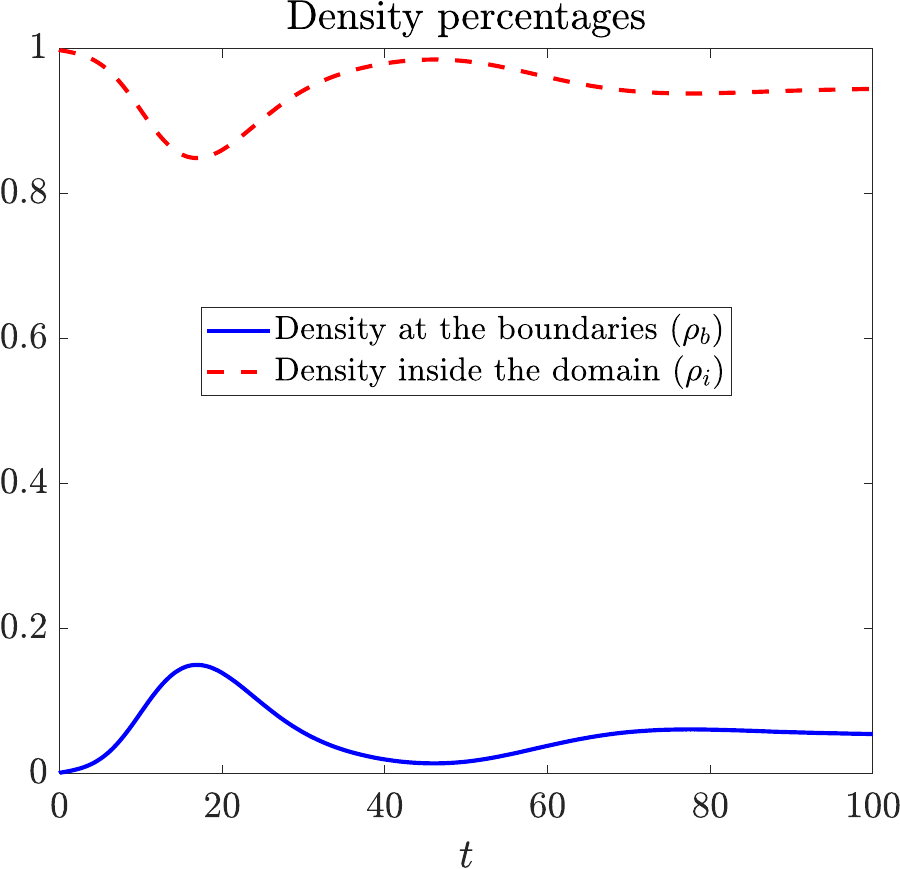}
	\includegraphics[width=0.36\linewidth]{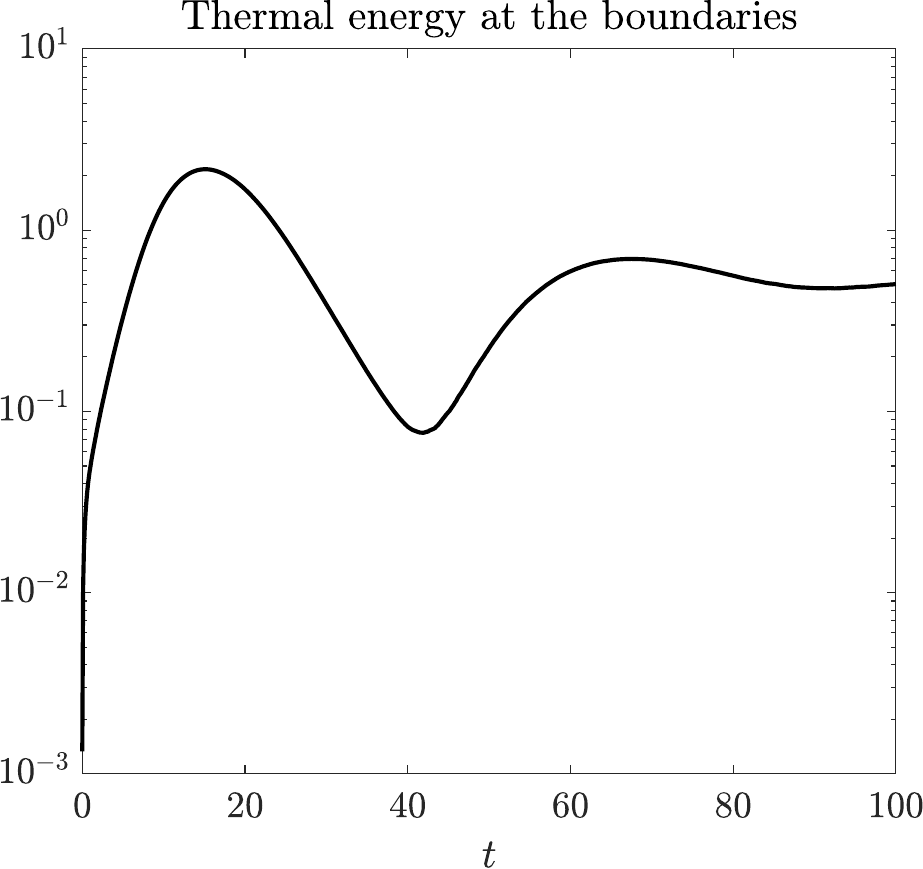}
	\caption{Two stream plasma test with constant magnetic field $B=1$. Density percentages defined as in equation \eqref{eq:mass_boundary} (on the left). Thermal energy at the boundaries defined as in \eqref{eq:energy} (on the right).}
	\label{fig:no_control_mass_energy_ex}
\end{figure}

\paragraph{Controlling the velocity and the position field.}
We suppose now that the magnetic field is obtained as the solution of a control problem aiming at minimizing the percentage of mass which hits the lower and upper boundaries \eqref{eq:L2_control_continuos}. This result is obtained by forcing the velocities and the positions of the particles with a suitable $\BB_{ext}$. We suppose the spatial domain, on which the control is active, to be divided into $K_x\times K_y=N_c$ cells $C_k$, we set $K_x=1$ and $K_y=2$, and we consider the instantaneous control \eqref{eq:L2_control_space_velocity} which aim at controlling both the particles positions and velocities. In details, we choose $\alpha_\emph{x} = \alpha_\emph{v} = 1.5$, $\beta_\emph{x} =\beta_\emph{v}= 0.1$, and $\gamma = 10^{-3}$ with targets $\hat{y}_k=0$ for any $k=1,2$ and $\hat{v}_{y_1} =1$, $\hat{v}_{y_2} = -1$, in such a way to force the mass to move toward the center of the domain and to remain in that region over time.
In Figure \ref{fig:L2_space_velocity_ex} in the first row, the density field at different times $t = 5$, $t =50$ and $t=100$ is shown. In the second row, the correspondent value of the velocity mean in the $y$ direction, of the electric field in the $y$ direction, and of the magnetic field in time are depicted. We observe that the instantaneous control realize a sort of bang-bang strategy which is activated once that the mass move sufficiently far away from the center of the domain.  
\begin{figure}[h!]
	\centering
	\includegraphics[width=0.328\linewidth]{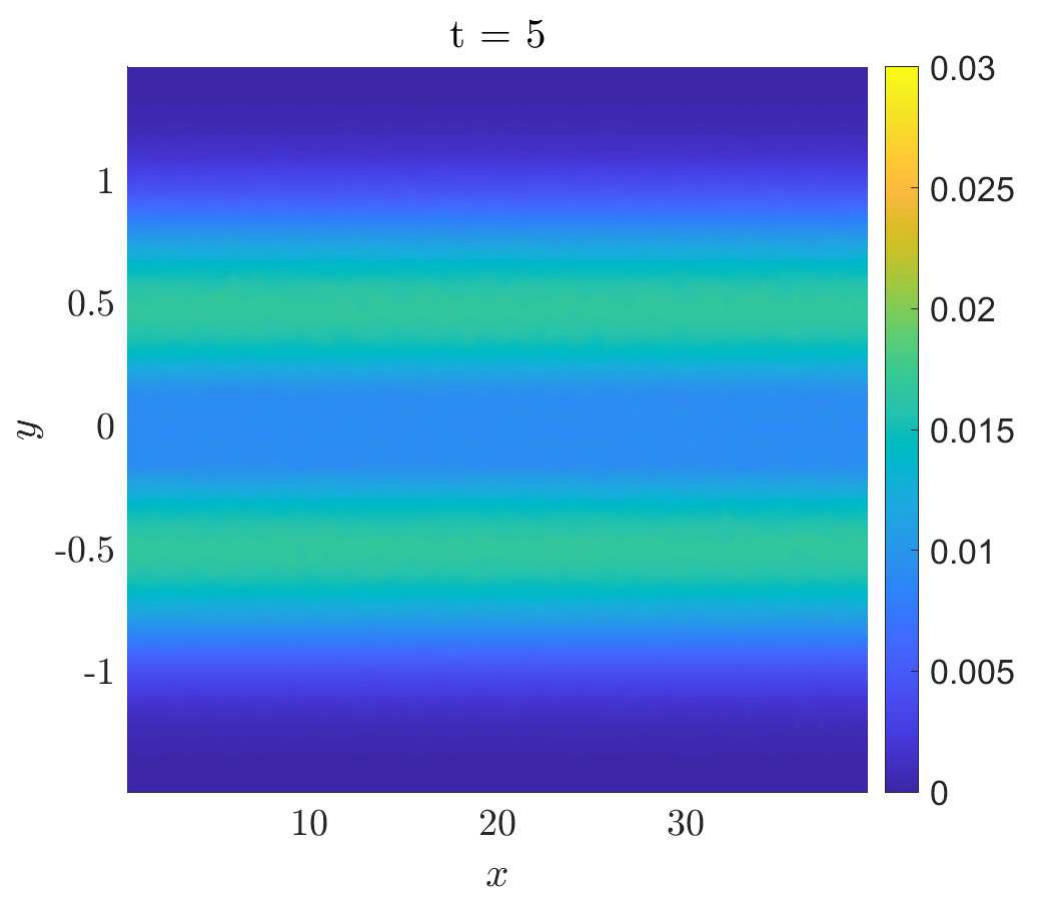}
	\includegraphics[width=0.328\linewidth]{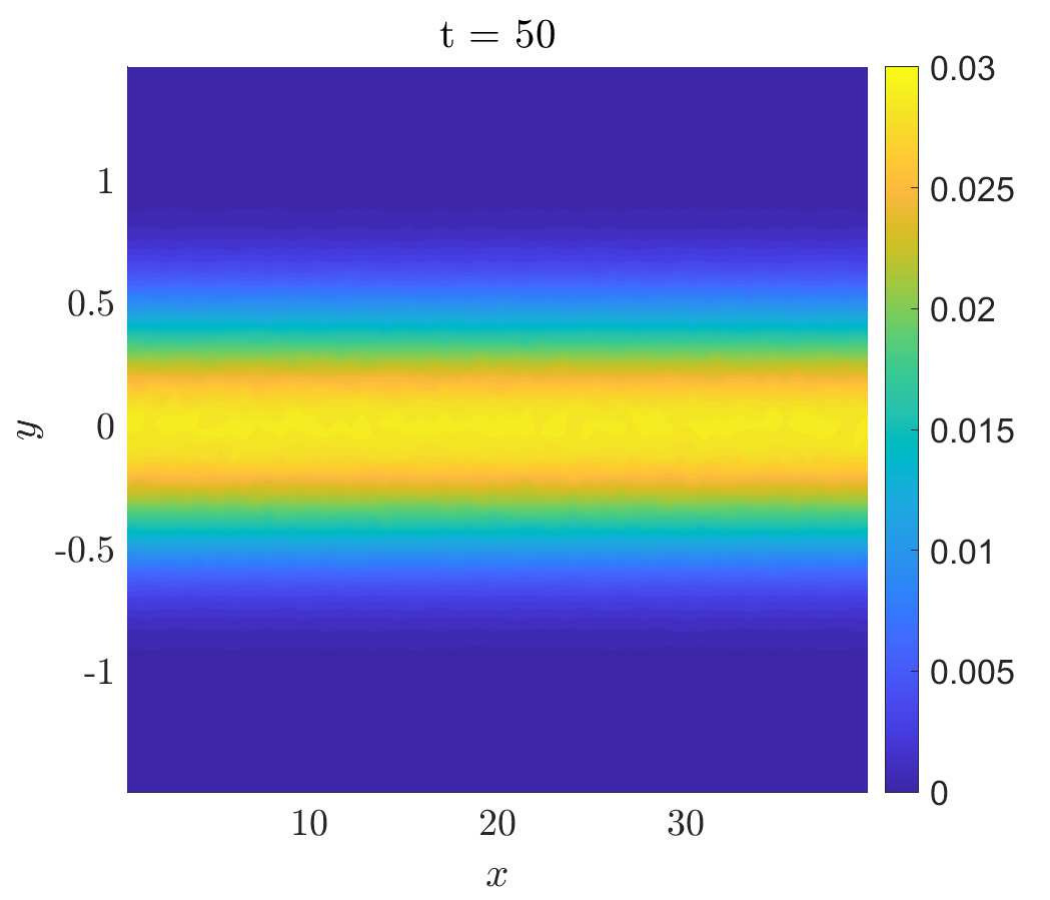}	\includegraphics[width=0.328\linewidth]{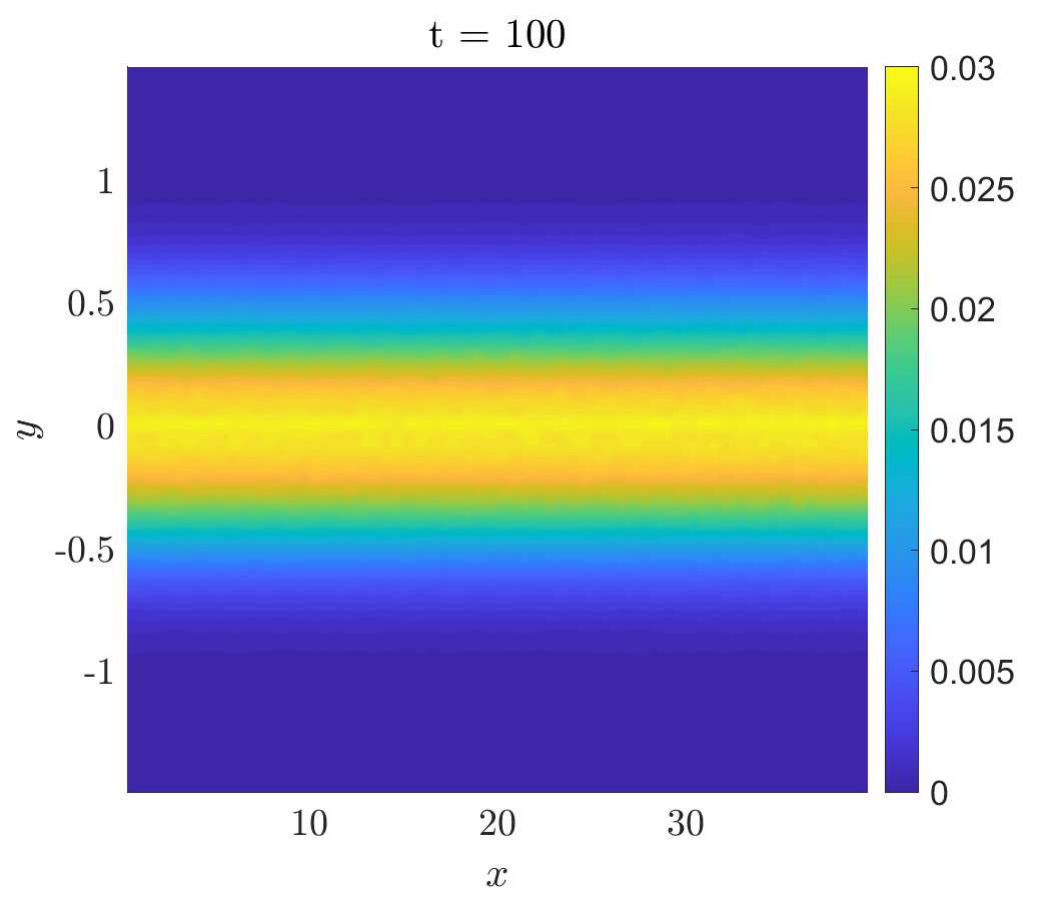}\\	
	\includegraphics[width=0.328\linewidth]{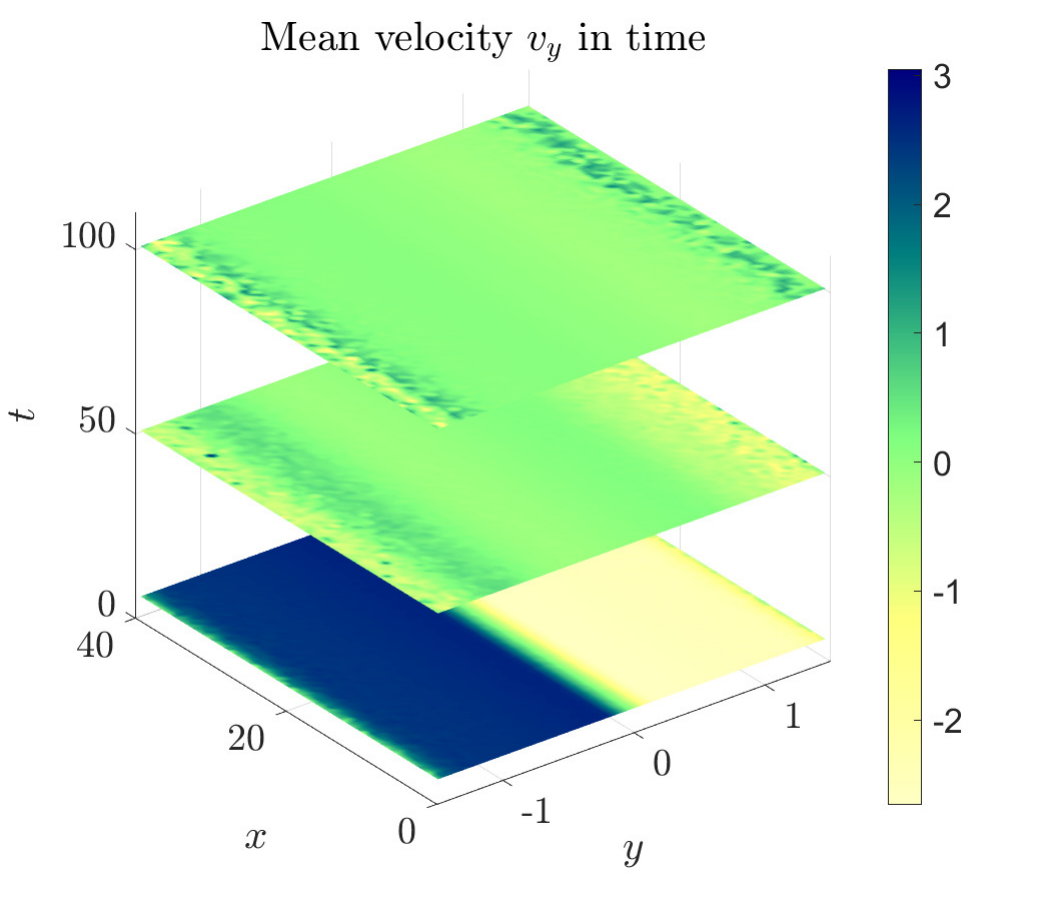}
	\includegraphics[width=0.328\linewidth]{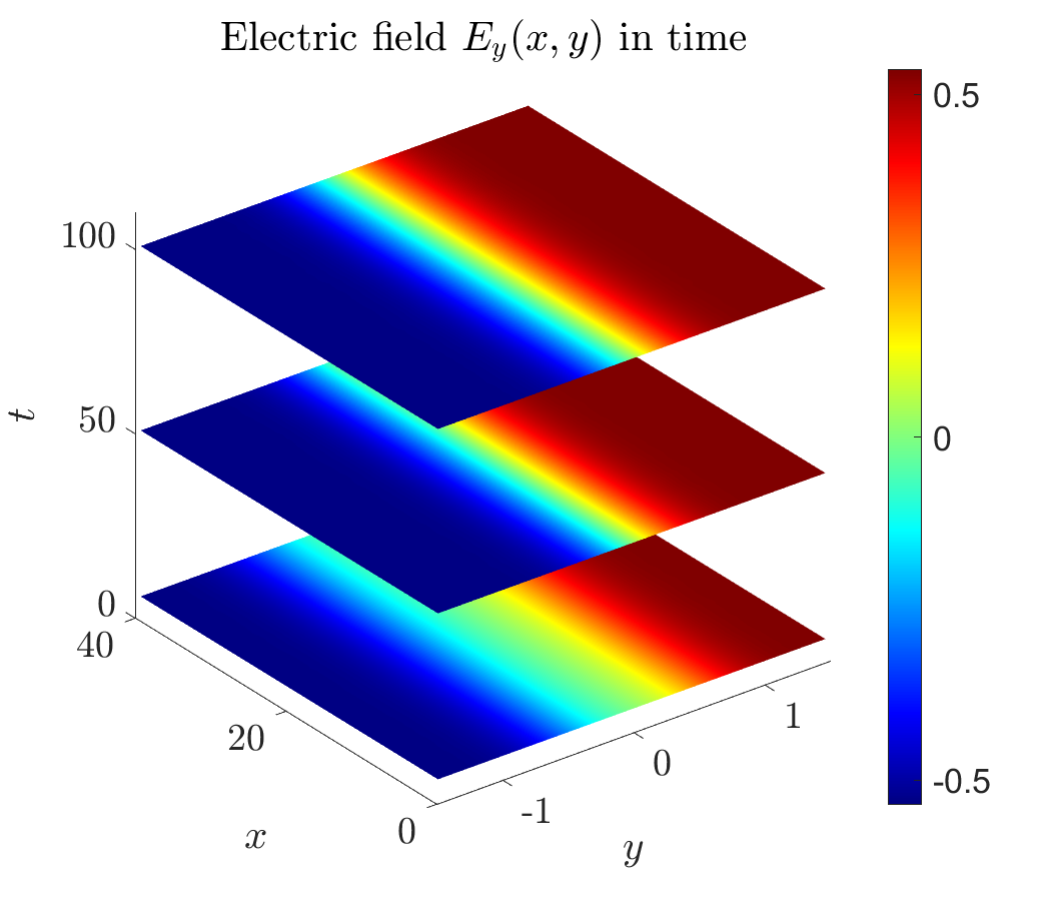}		\includegraphics[width=0.328\linewidth]{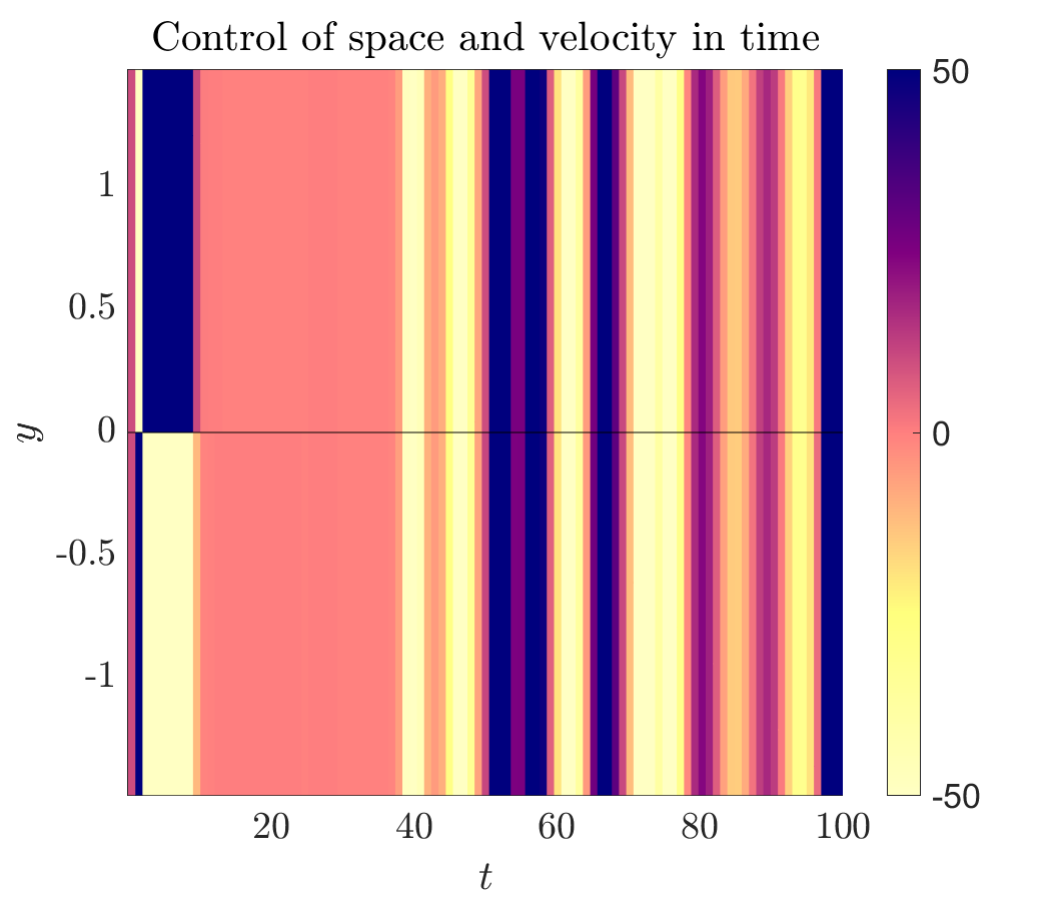}
	\caption{Two stream plasma test with control. We set $\alpha_\emph{x} = \alpha_\emph{v} = 1.5$, $\beta_\emph{x} =\beta_\emph{v}= 0.1$, and $\gamma = 10^{-3}$. First row: mass density taken at time $t=5$, $t =50$ and $t=100$. Second row: mean velocity in the $y$-direction (on the left),  electric field in the $y$-direction (in the middle), magnetic field (on the right).  }
	\label{fig:L2_space_velocity_ex}
\end{figure}
In Figure \ref{fig:L2_space_velocity_mass_energy_ex}, we plot the density percentages in time and the thermal energy at the boundaries. As for the previous cases, even if part of the mass still hits the boundaries the energy decreases toward zero, in other words, the control is effective. 
\begin{figure}[h!]
	\centering
	\includegraphics[width=0.4\linewidth]{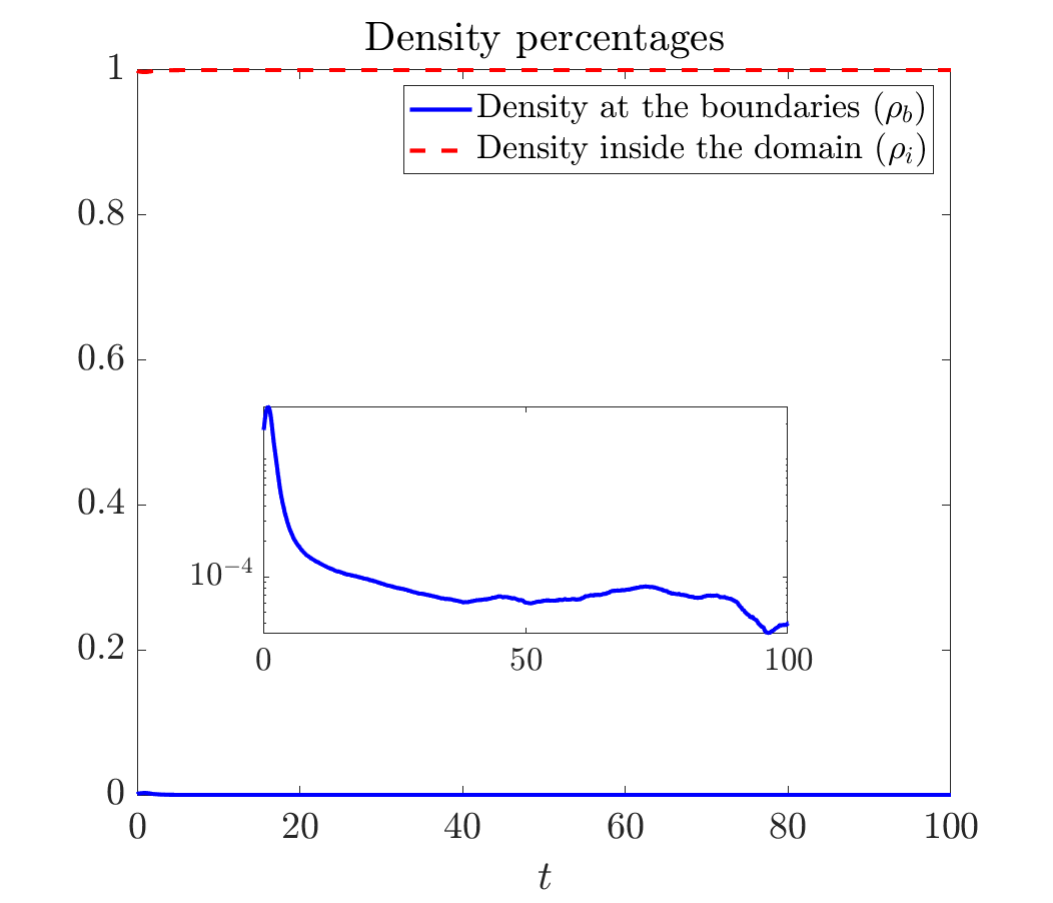}
	\includegraphics[width=0.4\linewidth]{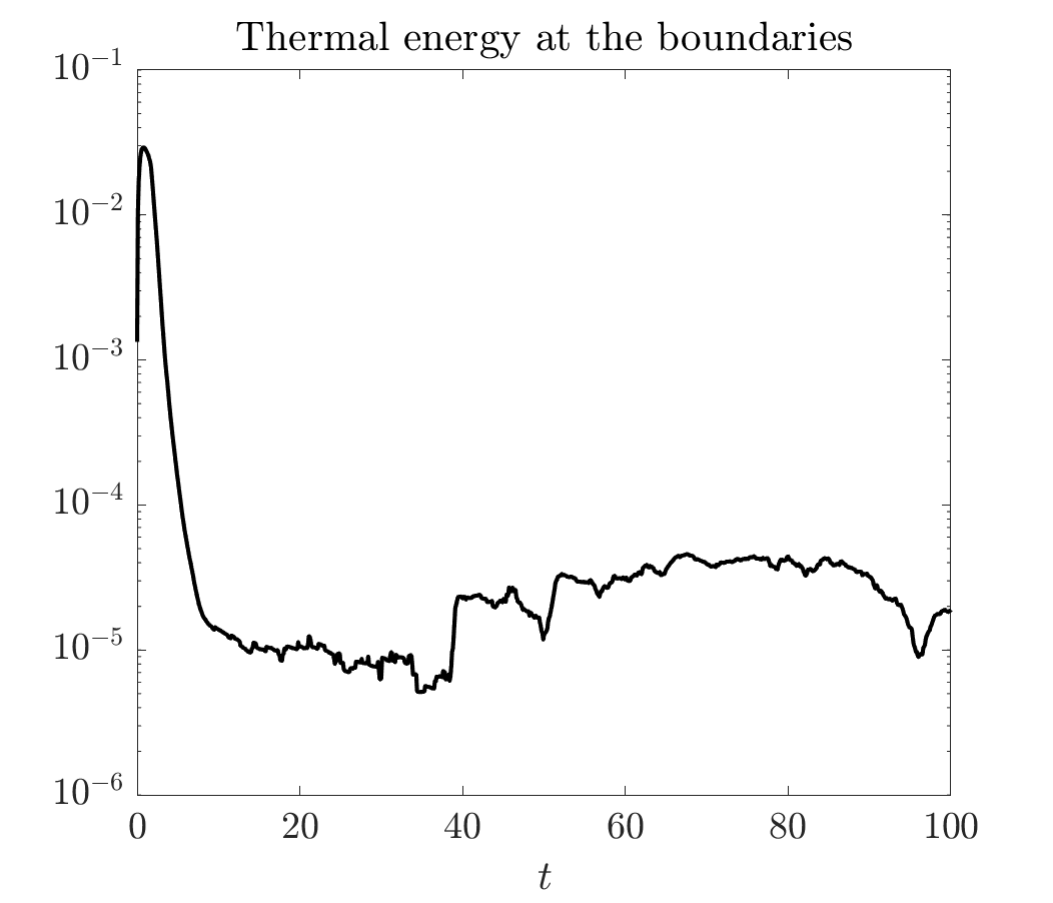}
	\caption{Two stream plasma test with control. We set $\alpha_\emph{x} = \alpha_\emph{v} = 1.5$, $\beta_\emph{x} =\beta_\emph{v}= 0.1$, and $\gamma = 10^{-3}$. Density percentages defined as in equation \eqref{eq:mass_boundary} (on the left). Thermal energy at the boundaries defined as in \eqref{eq:energy} (on the right).}
	\label{fig:L2_space_velocity_mass_energy_ex} 
\end{figure}

\subsection{Kelvin-Helmholtz instability.}
In this section, we focus on Kelvin-Helmholtz instability for charged particles^^>\cite{crouseilles2010conservative,sonnendrucker1999semi,chacon2016optimization,ghizzo1993eulerian} and we perform a similar analysis of the one detailed in the previous part by comparison of the controlled and uncontrolled cases. We choose periodic boundary conditions in $x$ and Dirichlet boundary conditions in $y$, and for the uncontrolled case we assume to have a strong magnetic field $B$ acting on the orthogonal plane. We consider a domain $x\in[0,40]$, $y\in[-5,5]$ and an initial density give by
\begin{equation}\label{eq:kelvin_initial_density}
	\rho_0(\xx) = \frac{1.5}{2\pi} \mbox{sech}\Big(\frac{y}{0.9}\Big) (1+\epsilon_0\cos({3k_0x})+\epsilon_1\sin(k_0x))), 
\end{equation}
\begin{figure}[h!]
	\centering
	\includegraphics[width=0.328\linewidth]{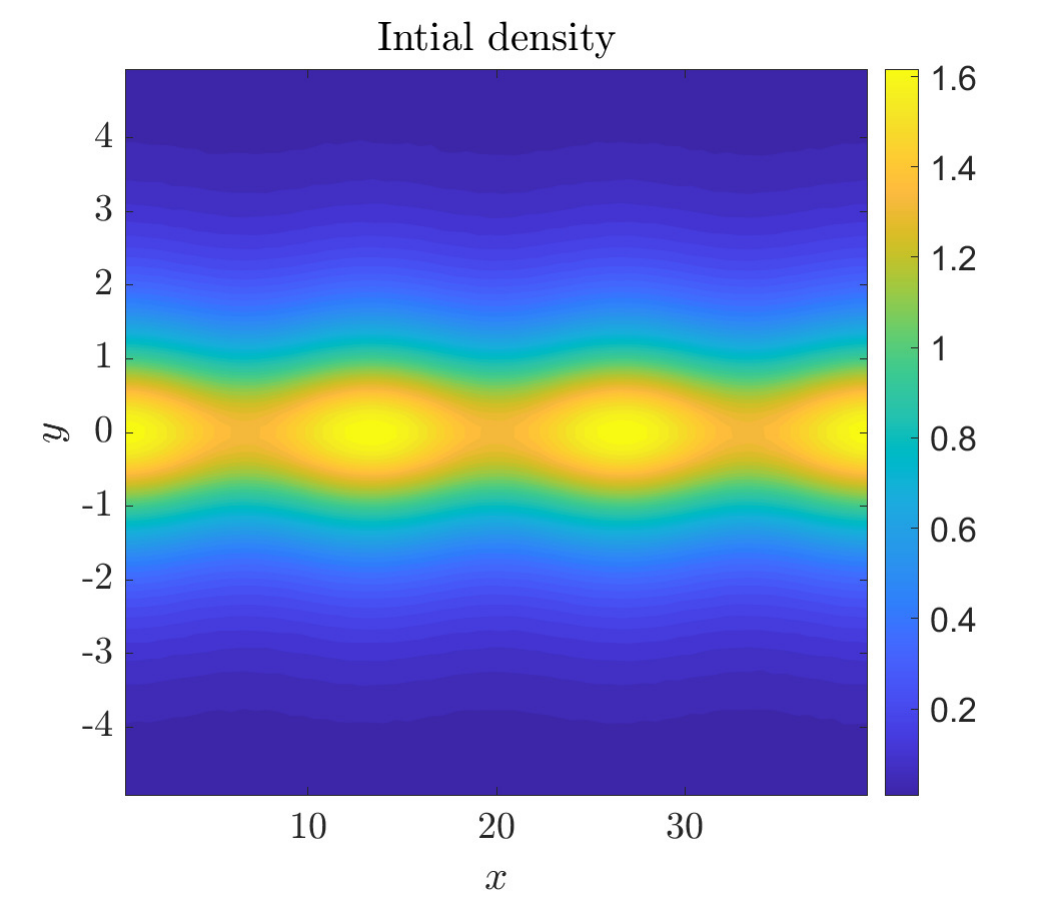}
	\includegraphics[width=0.328\linewidth]{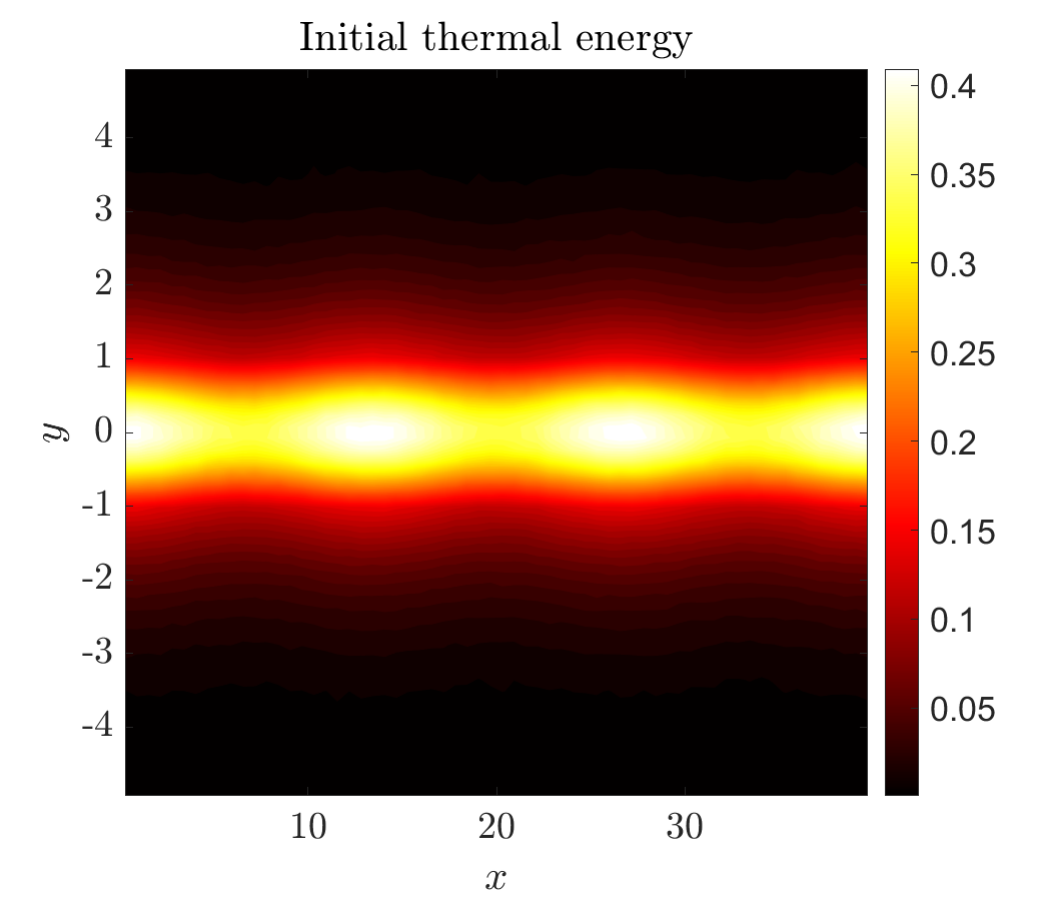}
	\includegraphics[width=0.328\linewidth]{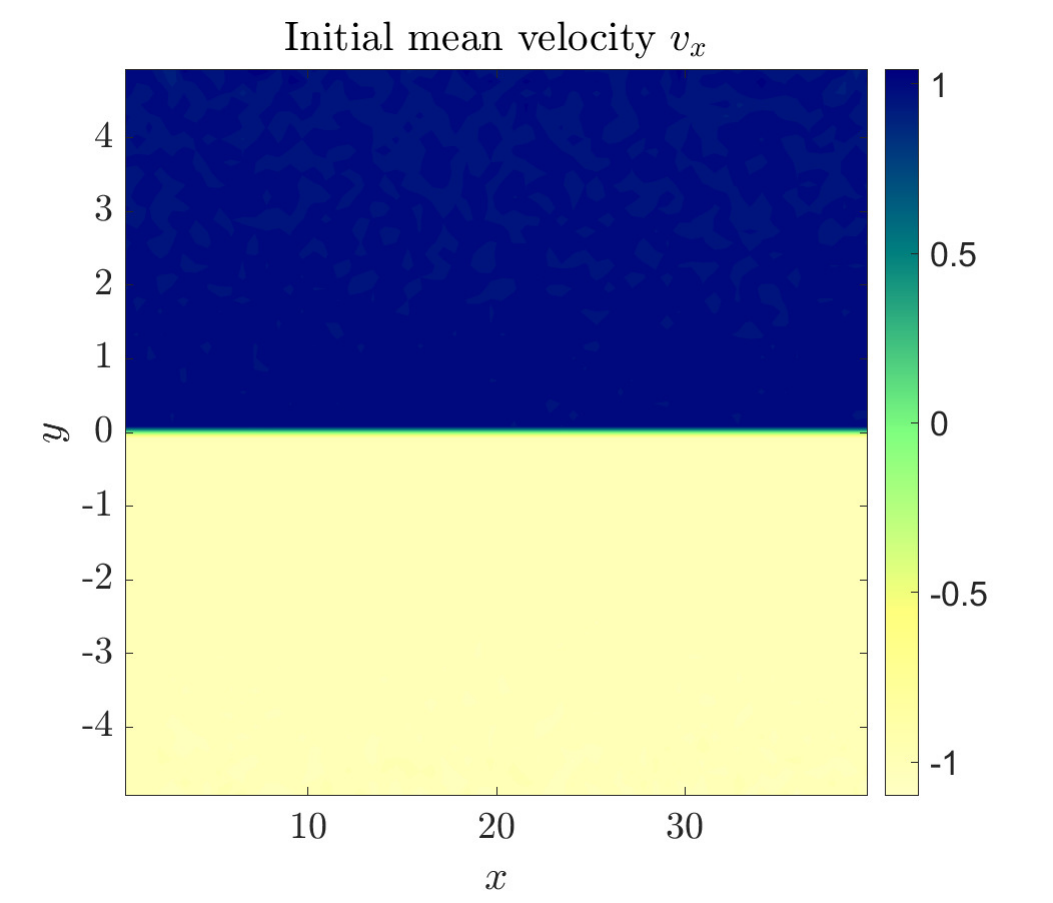}\\
	\includegraphics[width=0.328\linewidth]{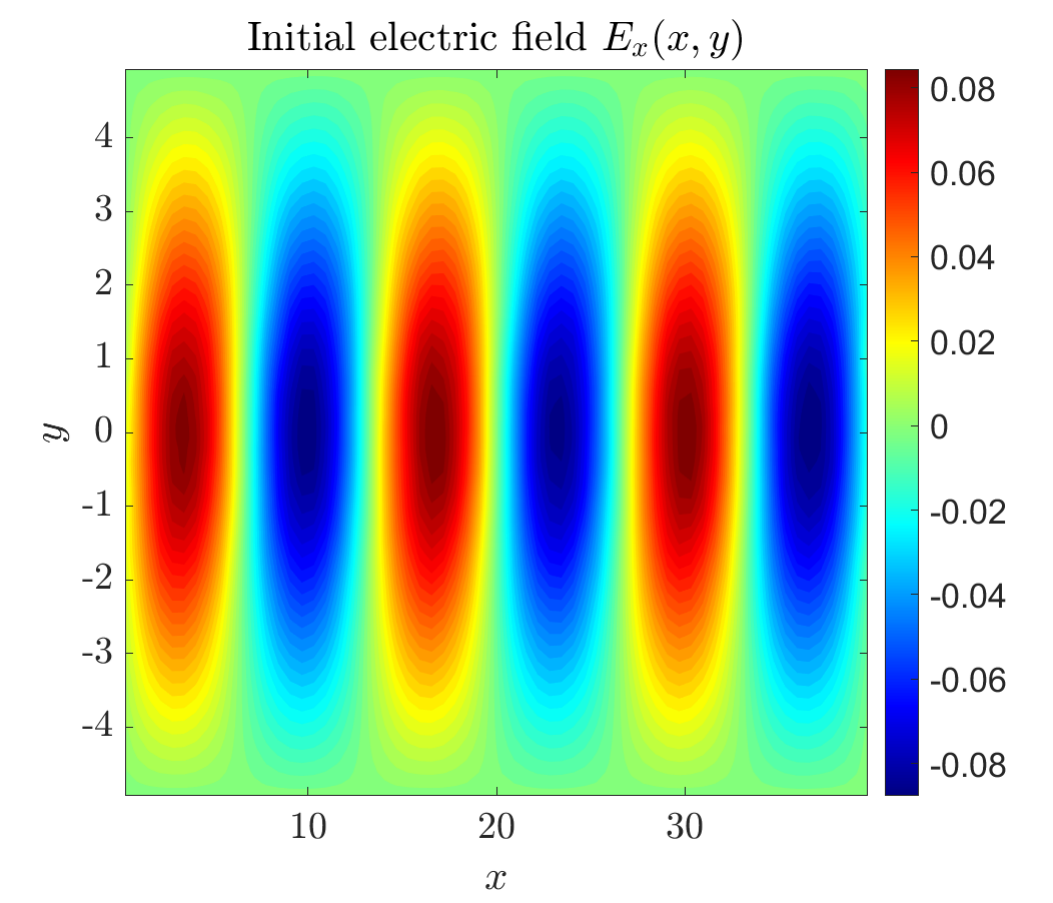}
	\includegraphics[width=0.328\linewidth]{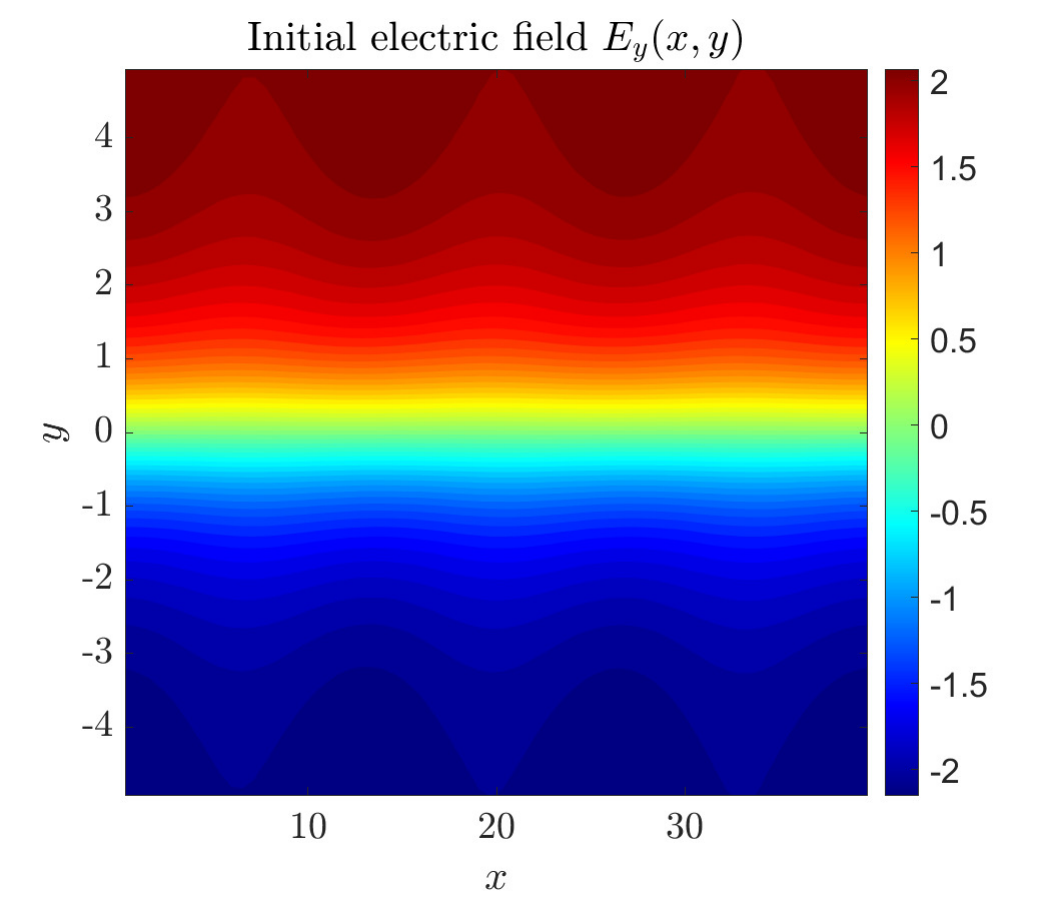}
	\caption{Kelvin-Helmholtz instability test. Initial density (top left). Initial temperature (top centre). Initial mean velocity in the $x$- direction (top right). Initial electric field in $x$ and $y$ direction (bottom).}
	\label{fig:ic_kelvin}
\end{figure}
with $k_0 = 0.15$,  $\epsilon_0=0.1$, $\epsilon_1=0.001$. We also take
\begin{equation}\label{eq:f0_kelvin} 
	f_0(\xx,\vv) = f_0^+(t,\xx,\vv) \chi(y\geq0)+ f_0^-(t,\xx,\vv)\chi(y<0),
\end{equation}
with 
\begin{equation}\label{eq:fo_kelvin_p_m}
	f_0^\pm(\xx,\vv) =\frac{\rho_0(\xx)}{ 2\pi T_0(\xx)}\exp\left(-\frac{(v_x\pm u_x)^2 + v_y^2}{2T_0(\xx)}\right),
\end{equation}
and $u_x = 1$, denoting the mean velocity in the $x$ direction. The initial temperature is given by
\begin{equation}\label{eq:temp_sigma_kelvin}
	\begin{split}
		T_0(\xx) =  0.15+0.1\cos\Big(\frac{\pi y}{2}\Big) \chi(-1\leq y\leq 1).
	\end{split} 
\end{equation}
\begin{figure}[h!]
	\centering
	\includegraphics[width=0.328\linewidth]{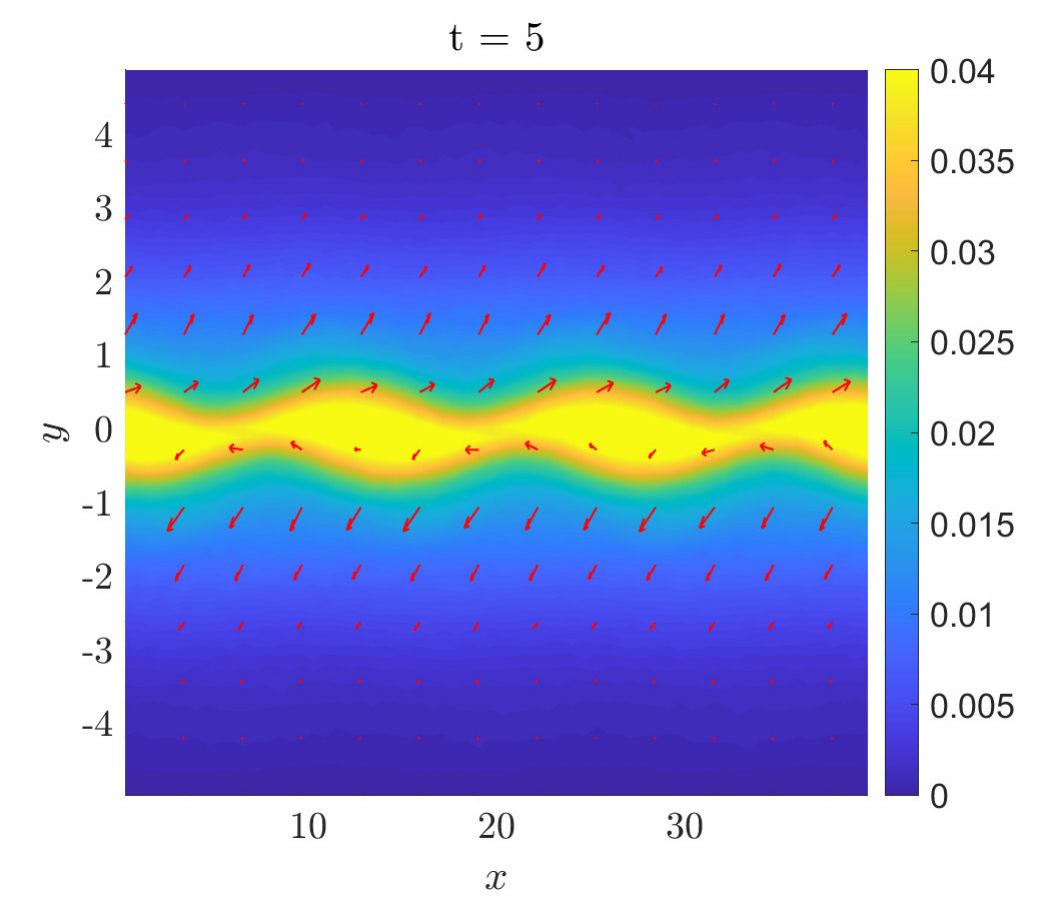}
	\includegraphics[width=0.328\linewidth]{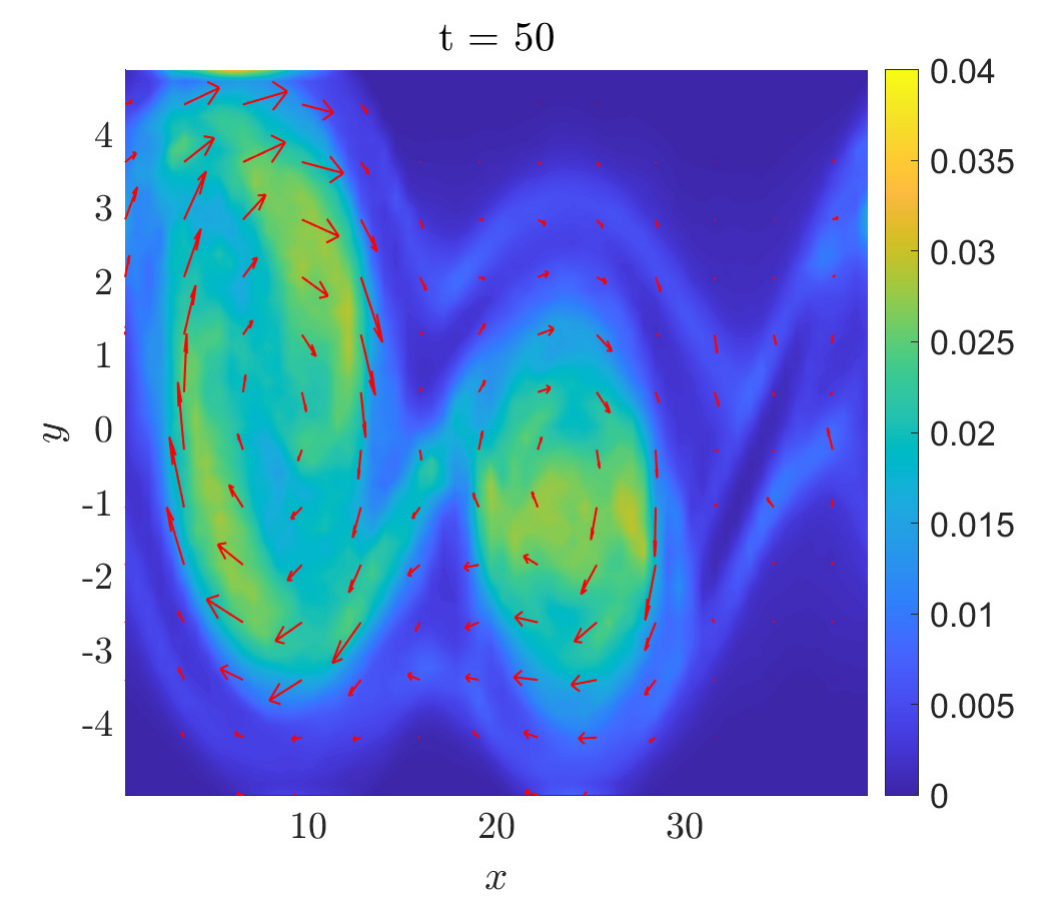}
	\includegraphics[width=0.328\linewidth]{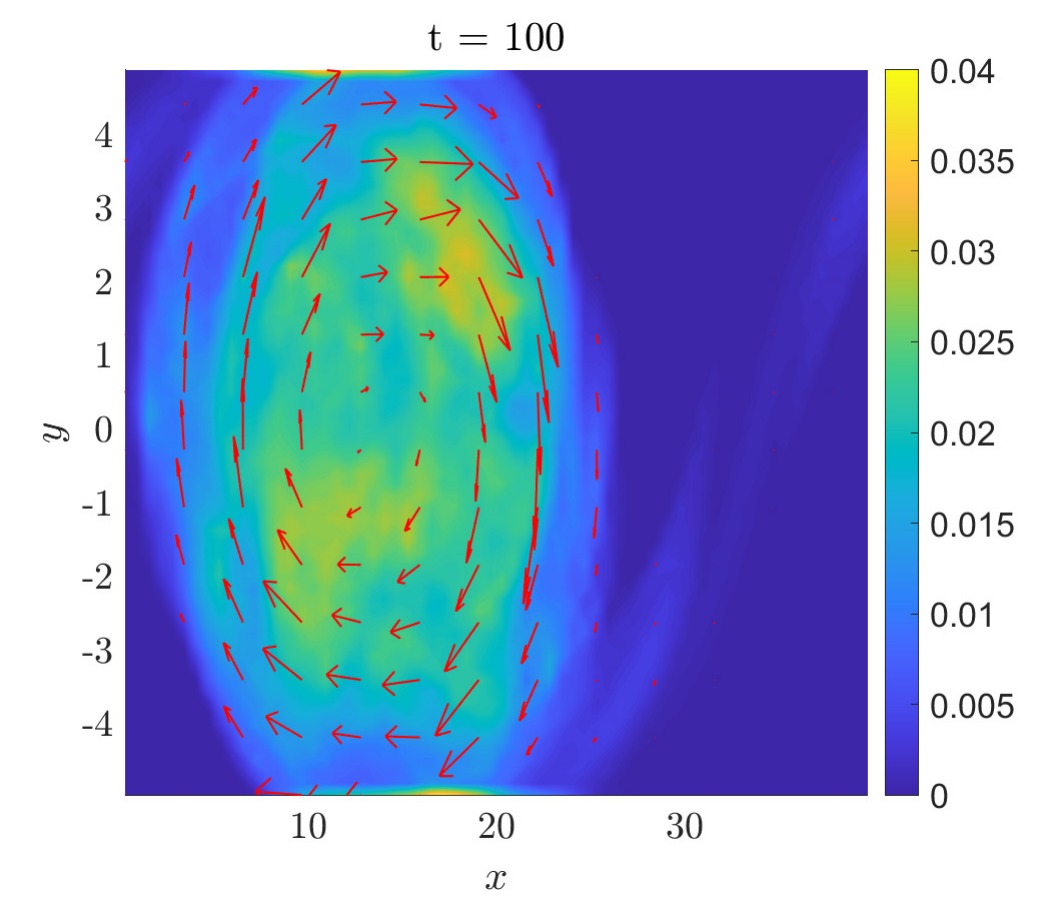}\\
	\includegraphics[width=0.328\linewidth]{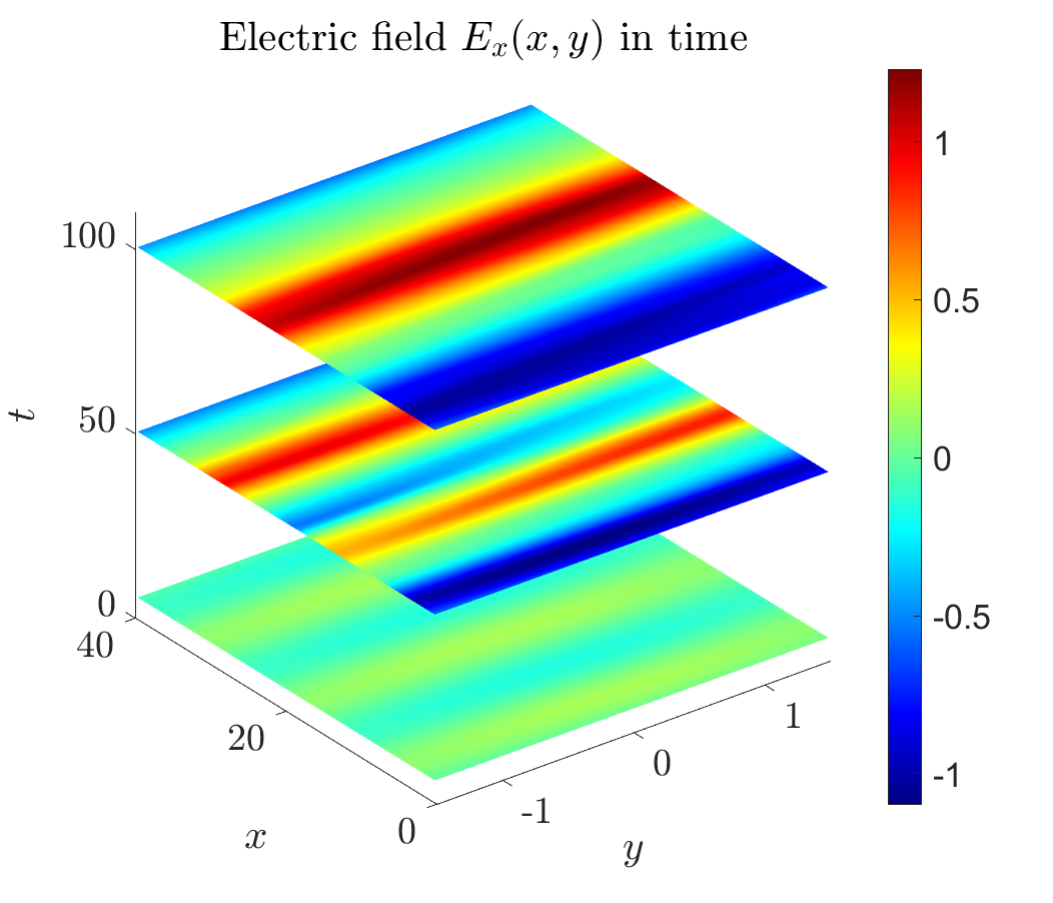} 	
	\includegraphics[width=0.328\linewidth]{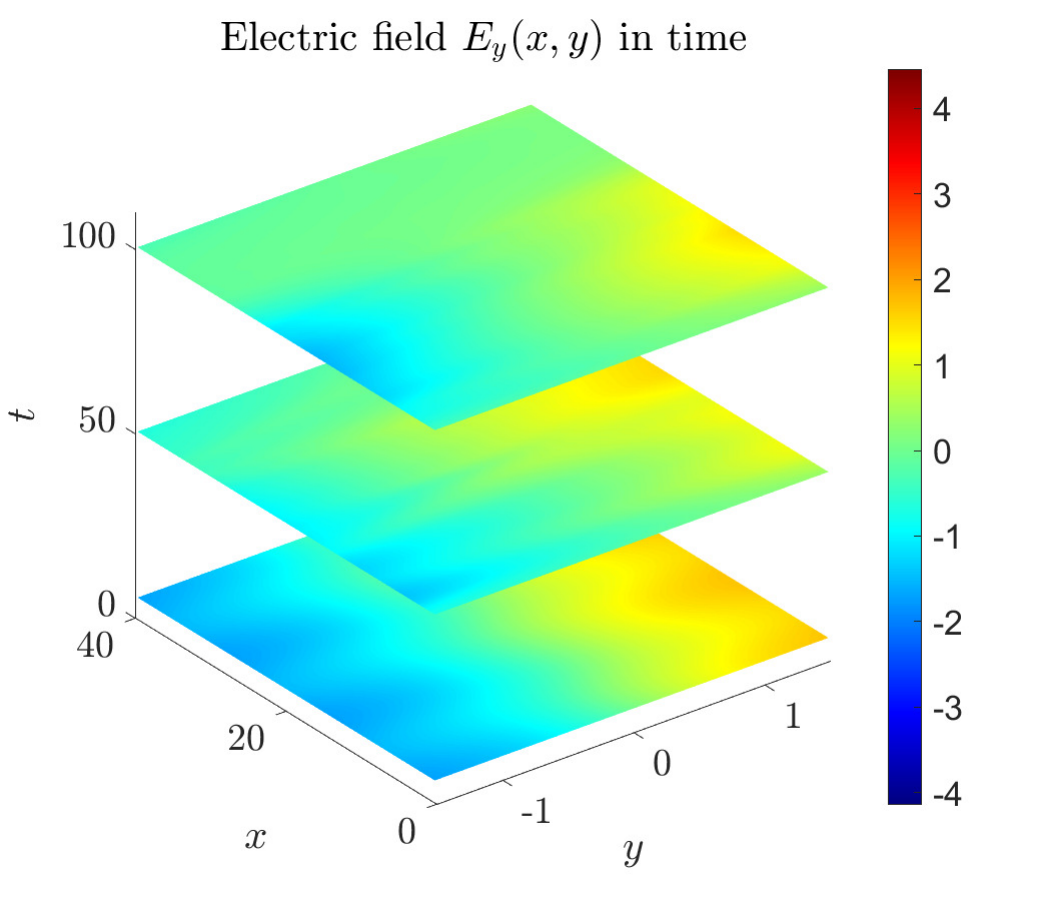}\\
	\centering
	\includegraphics[width=0.328\linewidth]{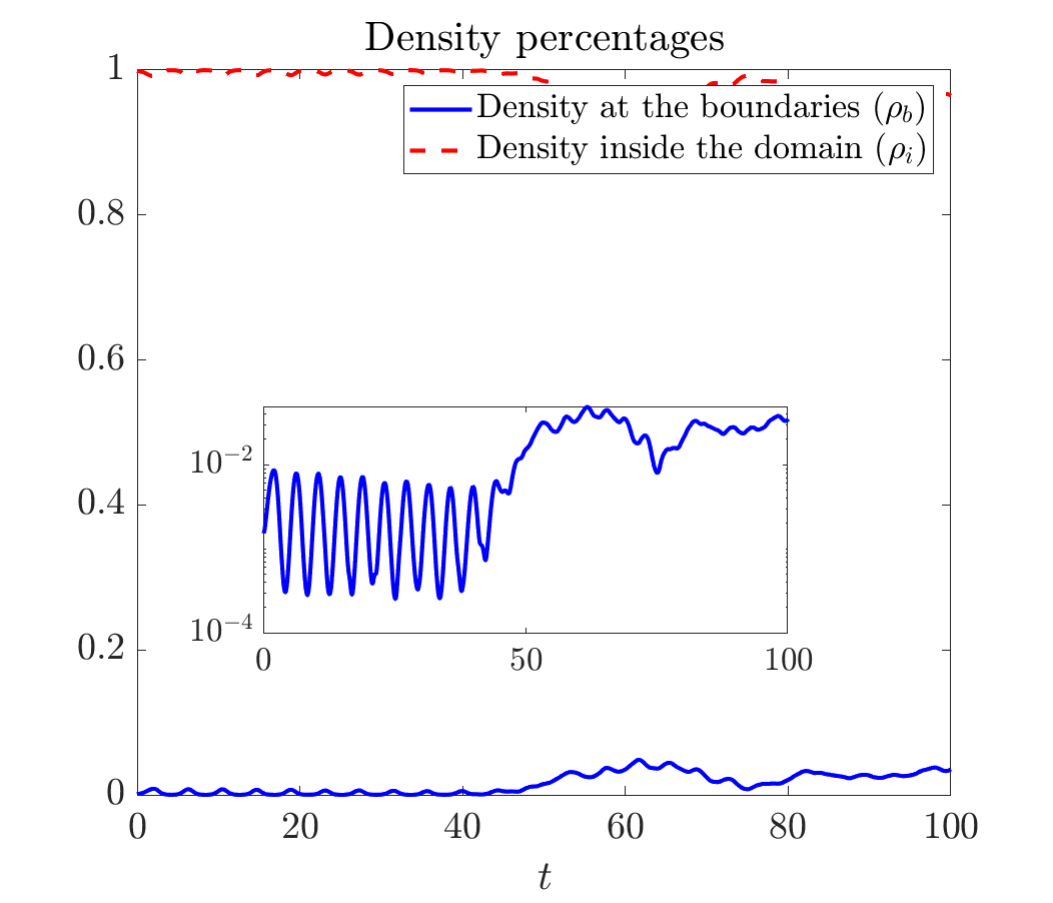}
	\includegraphics[width=0.328\linewidth]{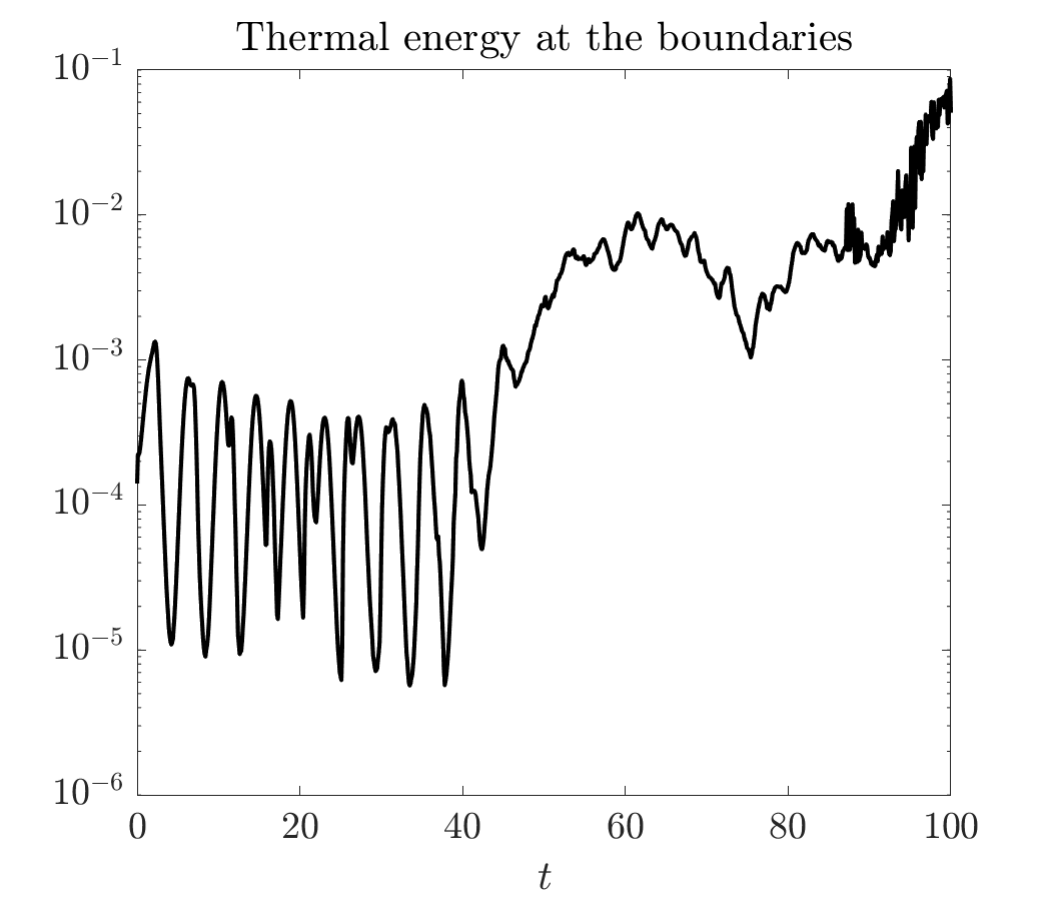}
	\caption{Kelvin-Helmholtz instability test: uncontrolled case. First row: density profile at time $t=5$, $t=50$ and $t=100$. In red the velocity vector field. Second row: value of the electric field in time. Third row: density percentages at the boundaries (on the left) and thermal energy at the boundary over time (on the right). }
	\label{fig:no_control}
\end{figure}
Simulations are performed using $N=10^7$ particles sampled from the above initial distribution, with  $m_x\times m_y$, $m_x=m_y = 64$ cells (see Figure \ref{fig:ic_kelvin}). The final time is $T=100$ and we choose $h=0.1$ as time step. We assume $N_c=K_x\times K_y$ cells, with $K_x=1$, $K_y=10$, for the definition of the control, and we measure the density and the thermal energy on the boundaries taking $\Omega_b = [-5,-4.8] \cup [4.8,5]$. 
In Figure \ref{fig:no_control}, in the first row, the density field at time $t=5$, $t =10$ and $t=100$ for the uncontrolled case with $B = 1.5$ are shown. In red the velocity vector field is also represented. In the second row, the corresponding electric field, and in the third row, the density and thermal energy at the boundaries over time are depicted. 
The case of the instantaneous control defined as in \eqref{eq:L2_control_space_velocity} is shown in Figure \ref{fig:kelvin_L2velocity}. We set $\alpha_\emph{x} = \alpha_\emph{v} = 1.5$, $\beta_\textrm{x} =\beta_\emph{v}= 0.1$, and $\gamma = 10^{-4}$, with targets $\hat{y}_k=0$, for $k=1,\ldots,10$, and $\hat{v}_{y_k}=1, \ k=1,\ldots,5$ and $\hat{v}_{y_k}=-1, \ k=6,\ldots,10$ to confine the mass at the centre of the domain as for the previous case. The results clearly show the capability of the instantaneous control strategy chosen to confine the plasma and to reduce the thermal energy close to the walls.
\begin{figure}[h!]
	\centering
	\includegraphics[width=0.328\linewidth]{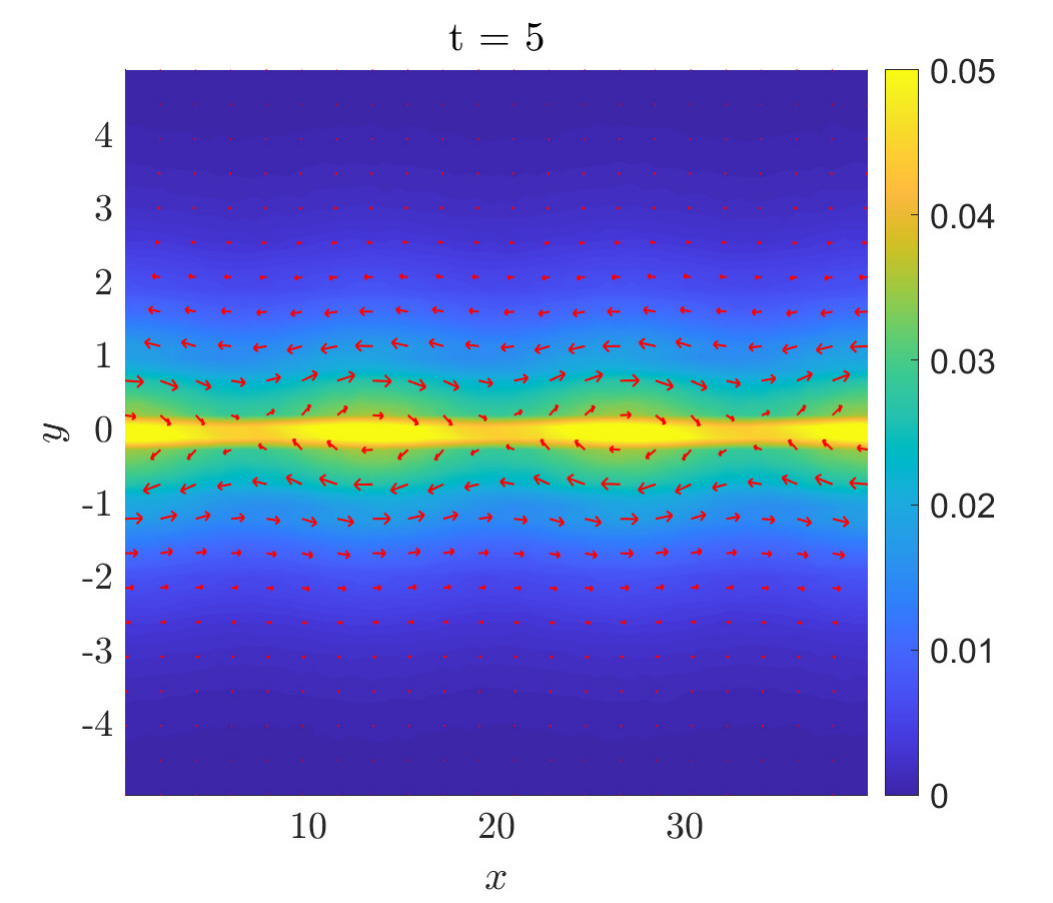} 	
	\includegraphics[width=0.328\linewidth]{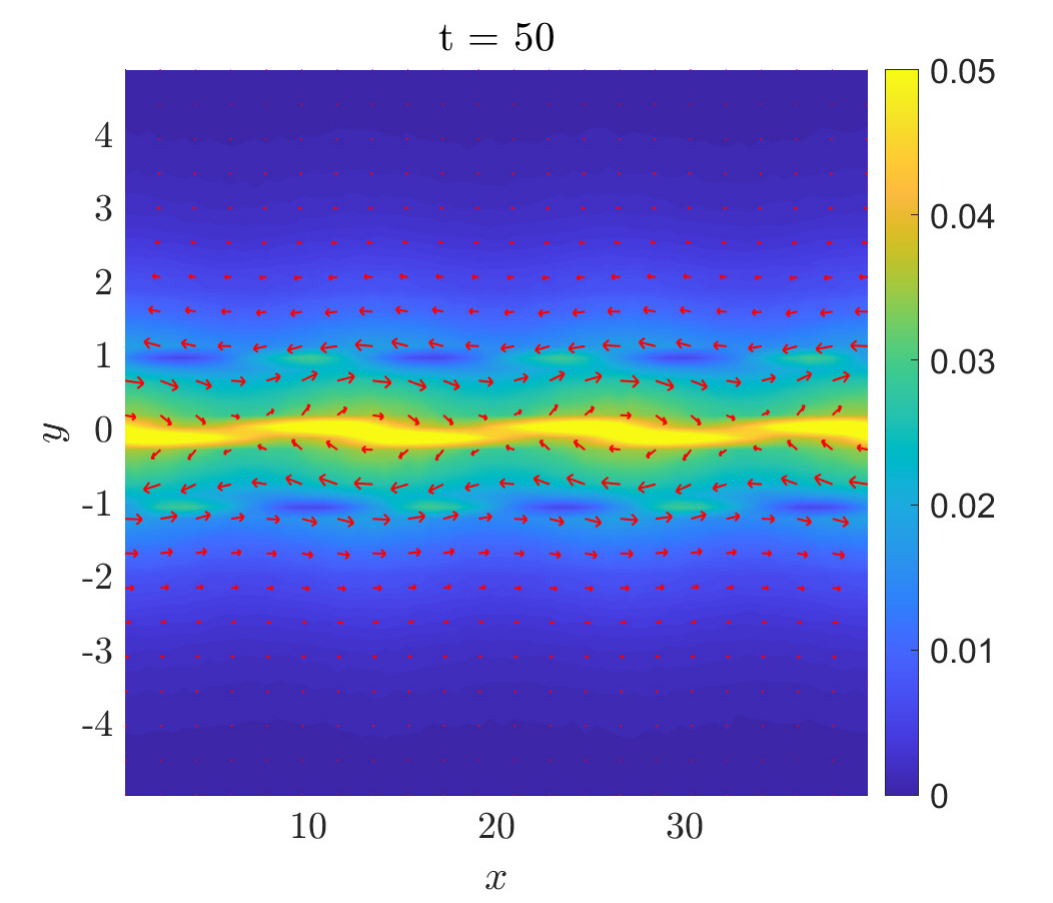} 	
	\includegraphics[width=0.328\linewidth]{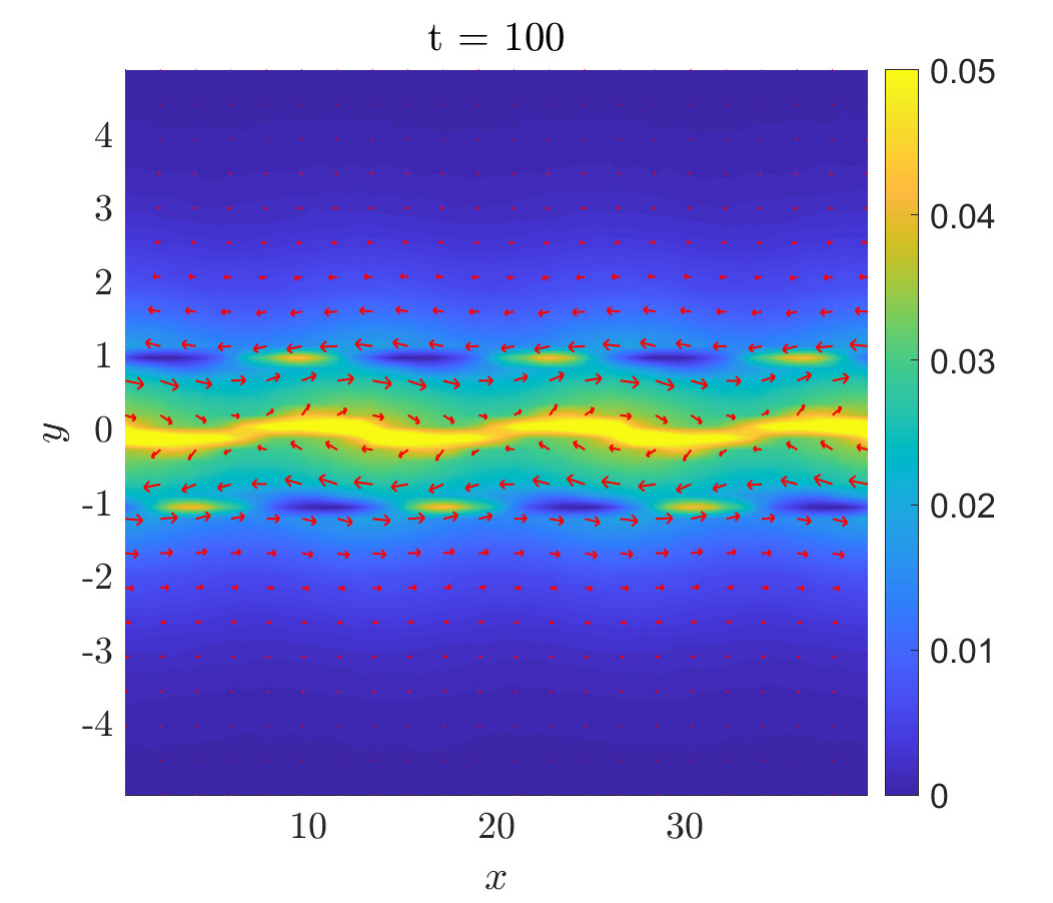} 	 	\\
	\includegraphics[width=0.328\linewidth]{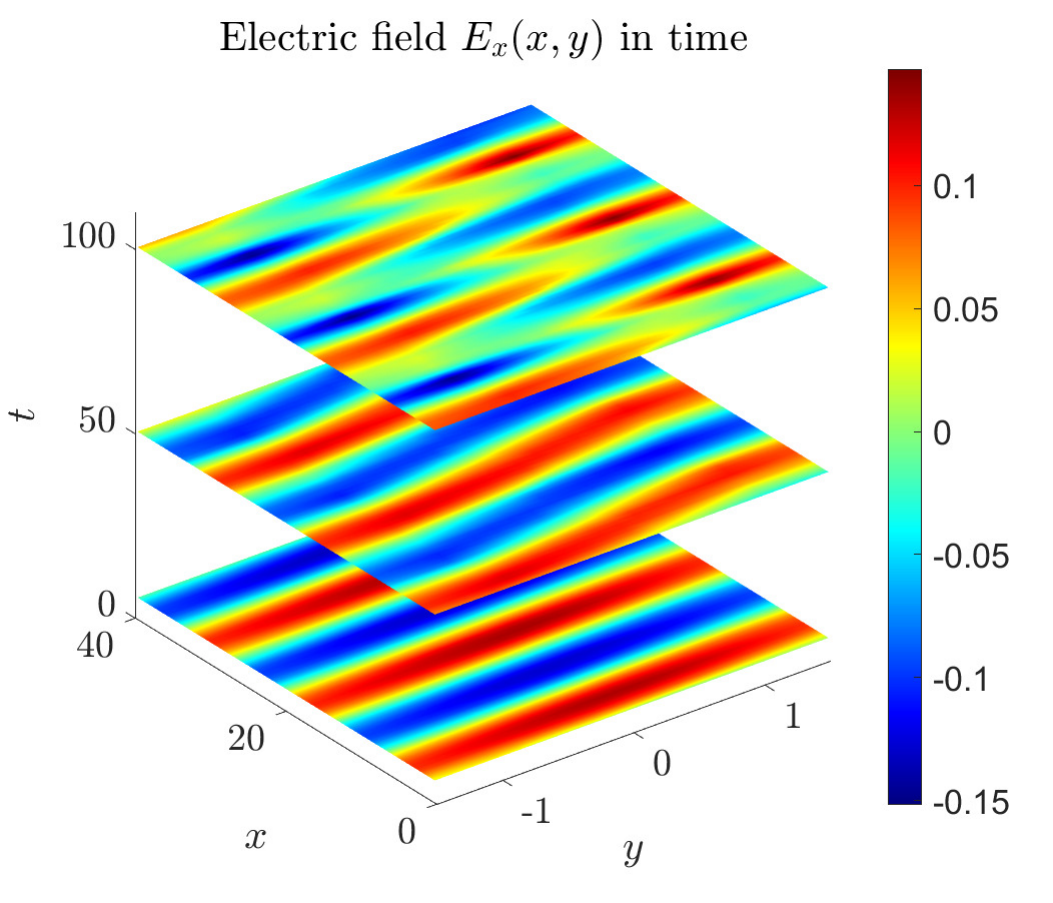} 	
	\includegraphics[width=0.328\linewidth]{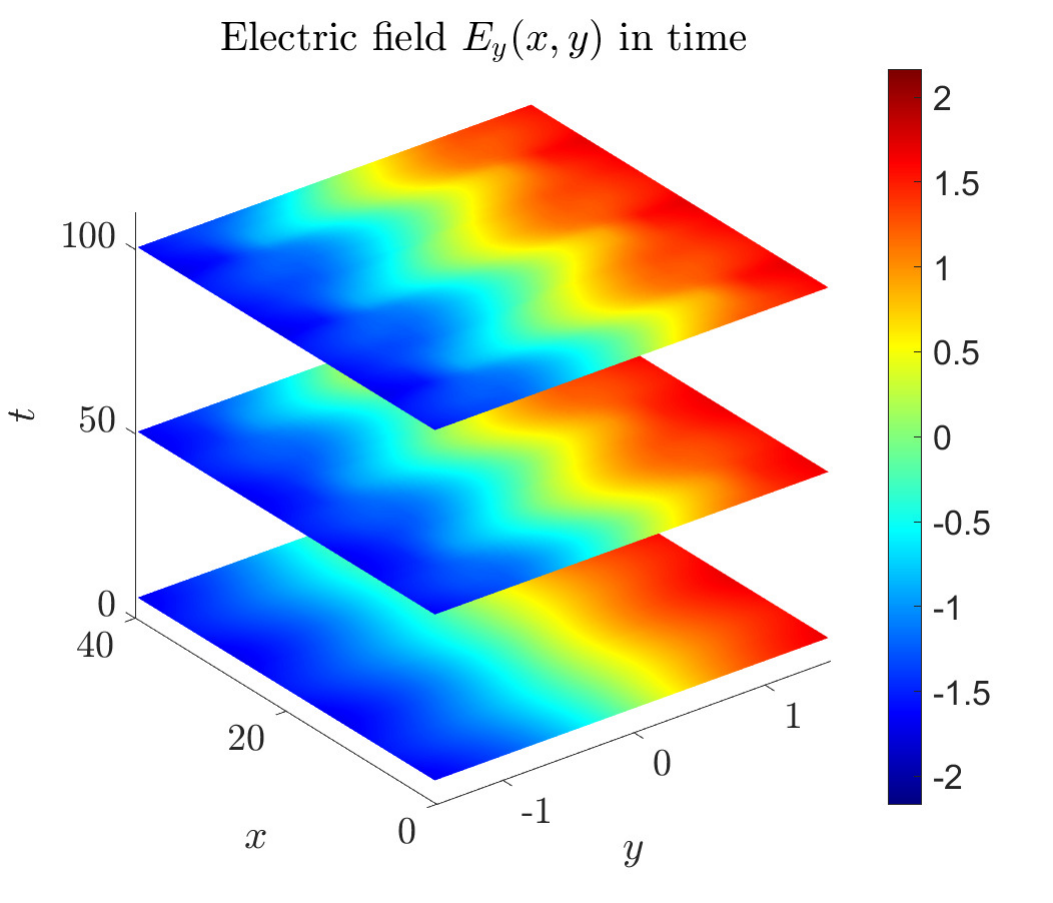} 	
	\includegraphics[width=0.328\linewidth]{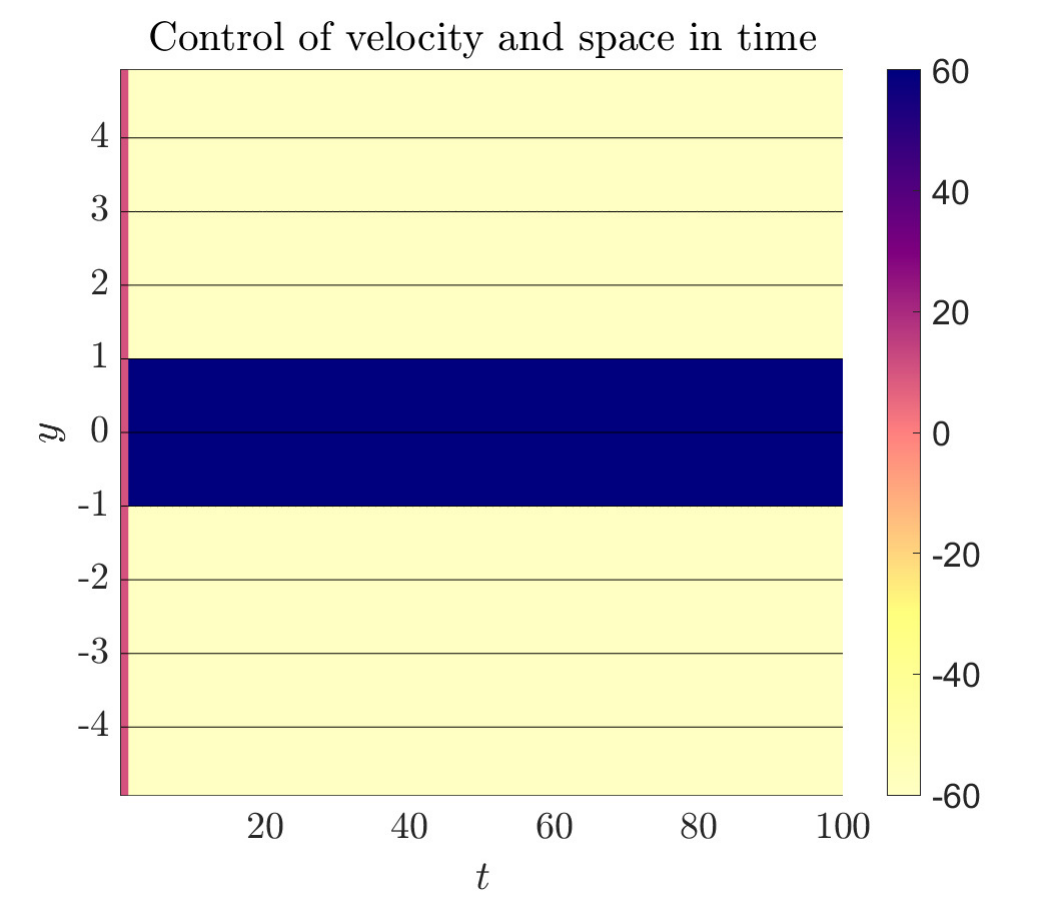} \\
	\includegraphics[width=0.328\linewidth]{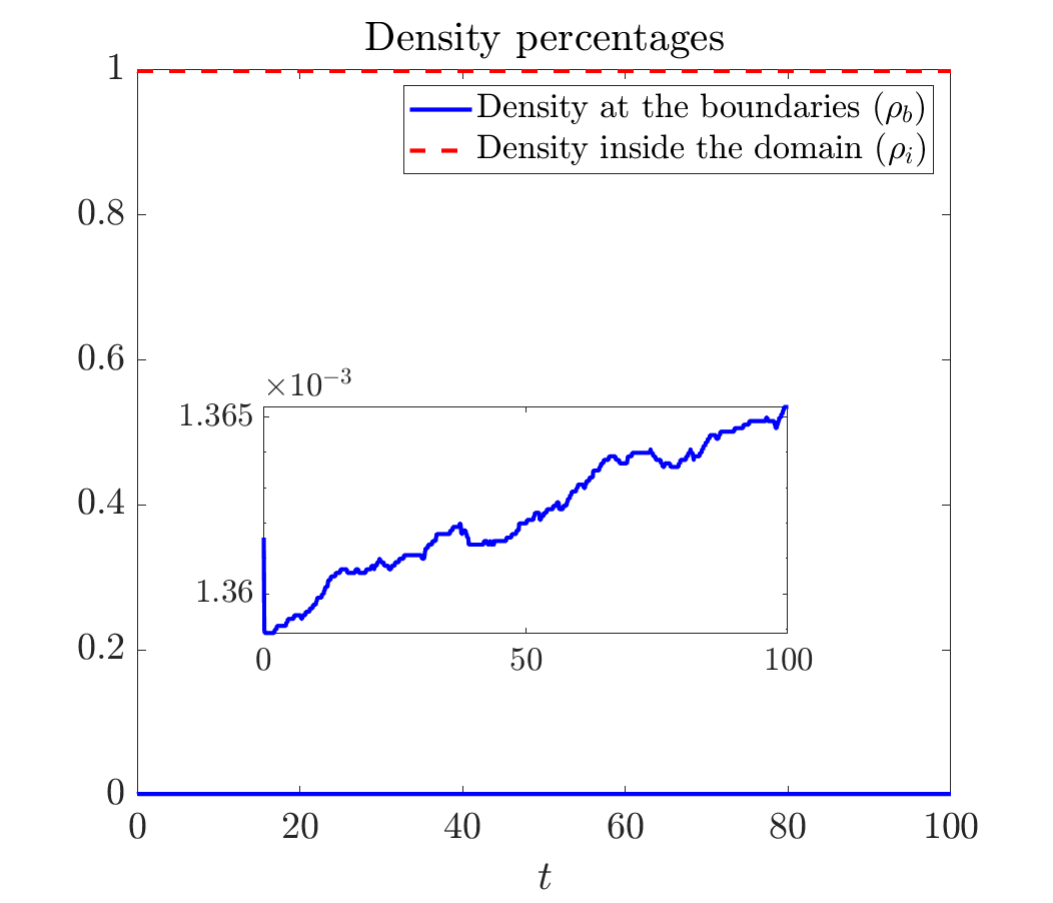} 	
	\includegraphics[width=0.328\linewidth]{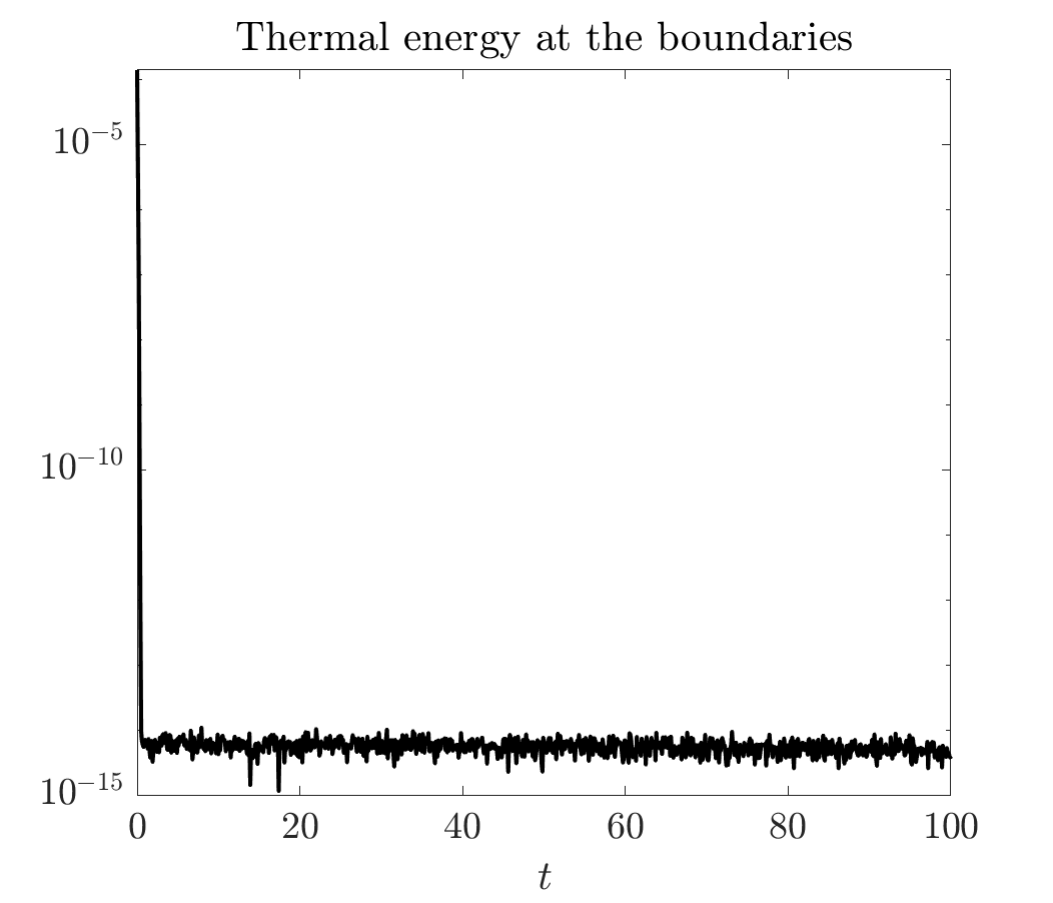} 	
	\caption{Kelvin-Helmholtz instability test: controlled case. First row: density profile at time $t=5$, $t=50$ and $t=100$. In red the velocity vector field. Second row: electric and magnetic field over time.  Third row: density percentages at the boundaries (on the left) and thermal energy at the boundaries (on the right).}
	\label{fig:kelvin_L2velocity}
\end{figure}
\begin{figure}[h!]
	\centering
	\includegraphics[width=0.322\linewidth]{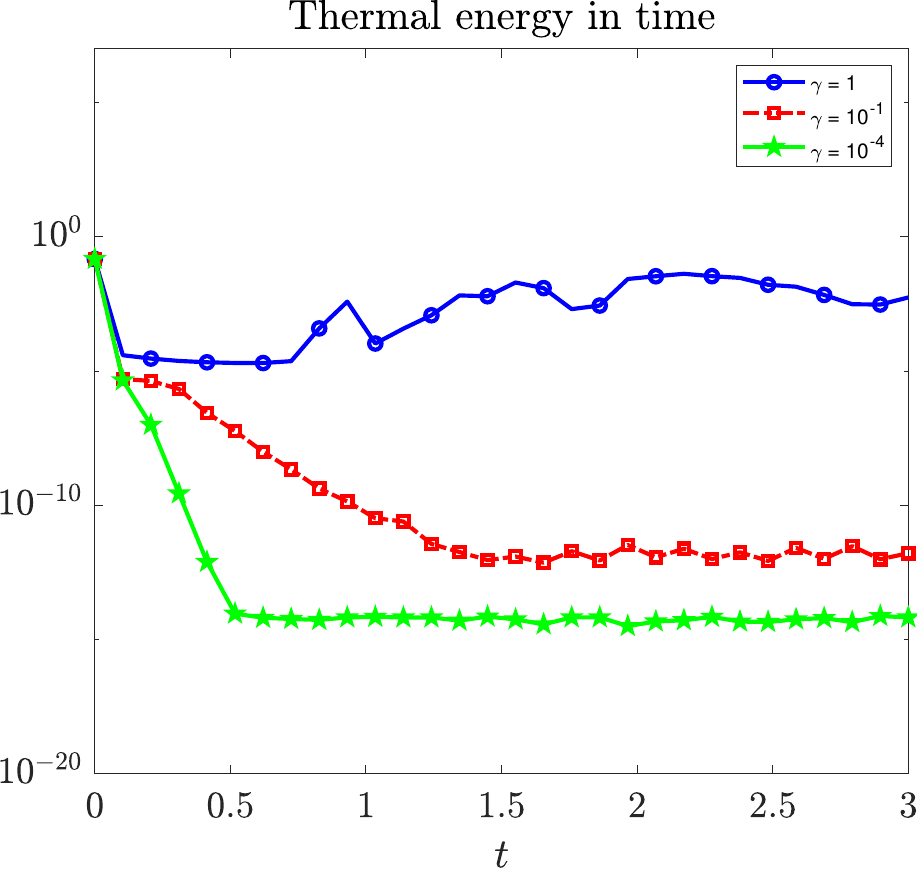} 
	\includegraphics[width=0.322\linewidth]{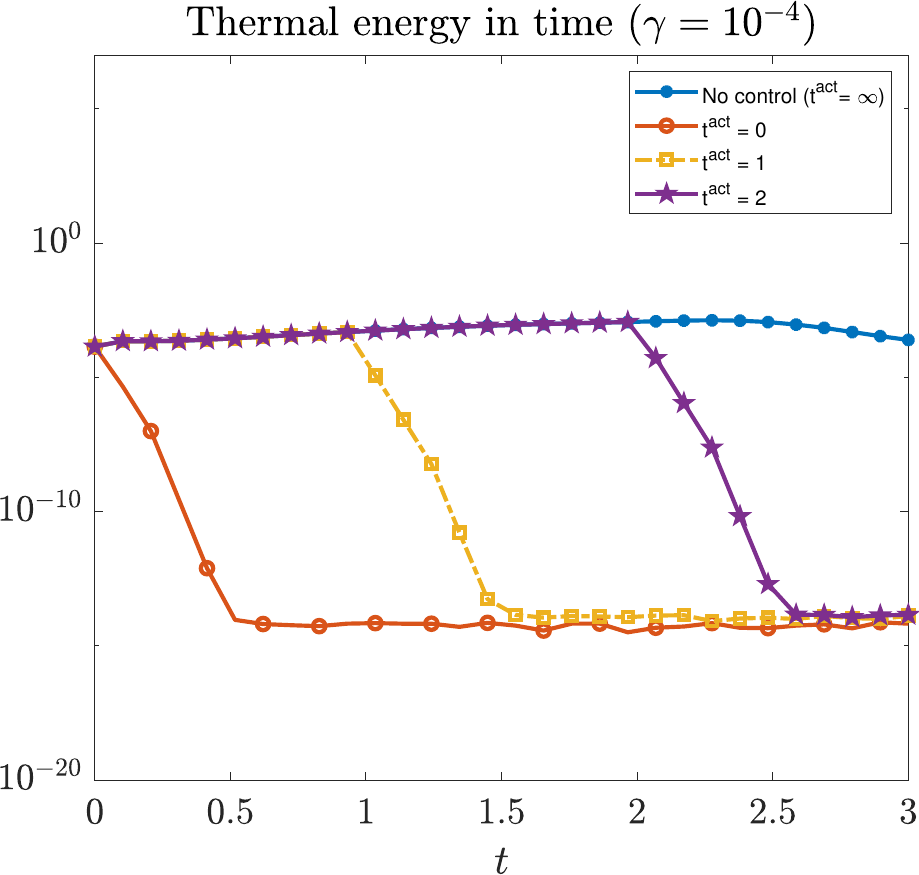} 	
	\includegraphics[width=0.328\linewidth]{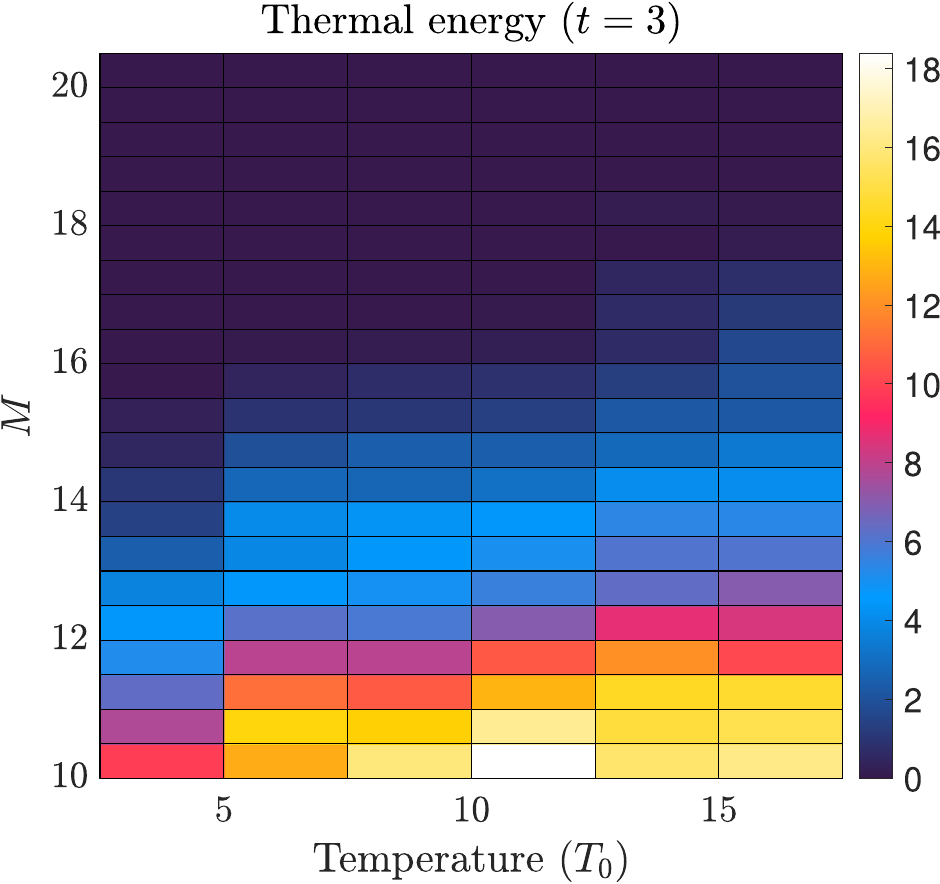} 
	\caption{Kelvin-Helmholtz instability test. Left: value of the thermal energy at the boundaries as the penalization term $\gamma$ varies. Middle: thermal energy at the boundaries assuming to activate the control at different times. Right: thermal energy at time $t=3$ as both the initial temperature and the control magnitude $M$ change.}
	\label{fig:comparison_L2_velocity_gamma} 
\end{figure}

Finally, in Figure \ref{fig:comparison_L2_velocity_gamma}, we compare the effectiveness of the control in confining the plasma in terms of thermal energy at the boundaries for different values of the penalization parameter $\gamma$ up to time $t=3$. The energy decreases proportionally to $\gamma$ showing that if one is willing to accept larger magnetic fields, the plasma can be more easily controlled. The central image in \ref{fig:comparison_L2_velocity_gamma} instead shows the value of the thermal energy at the boundaries when one assumes to activate the control at different times, namely $t^{act}=0$, $t^{act}=1$ and $t^{act}=2$. The energy starts immediately to decrease once the control is switched on,  i.e. there is no latency time. Finally, the right image shows the thermal energy at a given fixed time ($t=3$) assuming that both the initial temperature $T_0$ and the maximum magnitude of the control $M$ vary. The thermal energy decreases proportionally to the maximum allowed magnitude of the magnetic field control and increases as the initial temperature of the plasma increases, however showing that even for larger temperatures the control remains effective. 
 \section{Conclusions}\label{sec:conclusion}
 In this work, we have studied the possibility of using instantaneous control strategies to numerically describe and confine a plasma steered by the presence of an external magnetic field. The physical system is characterized by a kinetic equation of Vlasov type coupled with the Poisson equation to determine the intensity of the electric field. The Vlasov-Poisson system has been solved numerically using a second-order Particle-In-Cell method semi-implicit in time. In this setting, the optimal control problem has been approximated through an instantaneous control strategy based on a discretization of the cost functional over a single time step. In order to compute an explicit feedback control a further approximation based on a simpler time discretization of the dynamic has been adopted. Subsequently, this approximated feedback control was incorporated into the second-order PIC method.
An analytical comparison of the discretize and then optimize strategy with the optimize and then discretize alternative has been also presented. Finally, we evaluated the resulting method in various typical scenarios encountered in confinement problems. 
The numerical simulations showed that the approach outlined in this work effectively guides particles toward a desired configuration, paving the way for several future investigations. Specifically, we aim to extend this methodology to the Vlasov equation coupled with the Maxwell equations in a full three dimensional setting, to take into account the effect of collisions and to introduce uncertainty into the physical system. 
 
 \section*{Acknowledgments}
 This work has been written within the activities of GNCS and GNFM groups of INdAM (Italian National Institute of High Mathematics).
 GA thanks the support of MUR-PRIN Project 2022 PNRR No. P2022JC95T, ``Data-driven discovery and control of multi-scale interacting artificial agent systems", and of MUR-PRIN Project 2022 No. 2022N9BM3N  ``Efficient numerical schemes and optimal control methods for time-dependent partial differential equations" financed by the European Union - Next Generation EU.
 The research of LP has been supported by the Royal Society under the Wolfson Fellowship ``Uncertainty quantification, data-driven simulations and learning of multiscale complex systems governed by PDEs" and by MUR-PRIN Project 2022, No. 2022KKJP4X ``Advanced numerical methods for time dependent parametric partial differential equations with applications".
 The partial support by ICSC -- Centro Nazionale di Ricerca in High Performance Computing, Big Data and Quantum Computing, funded by European Union -- NextGenerationEU is also acknowledged.

\begin{appendices}
\section{ }\label{sec:appendix} 

In this appendix we report some validations test of the numerical methods, in particular we analyze the convergence rates on the so-called Diocotron instability test case. We first compute the numerical error and the rate of convergence in time for the two semi-implicit PIC schemes introduced in Section \ref{sec:Numerical_methods} in the uncontrolled case. Next, we present the numerical error as the number of particles $N$ increases for two different cases: a stochastic initialization and a deterministic initialization for the PIC strategy.
%
\paragraph{Diocotron instability: dynamics.} 
We focus on the solution of the Vlasov-Poisson equation with an external magnetic field $B=10$ in a two dimensional setting. We take as initial density \begin{equation}\label{eq:diocotron}
	f_0(\xx,\vv)	= \frac{1}{2\pi}\left( 1+\alpha \cos\left( k\arctan{\left( \frac{y}{x}\right) }\right) \right) e^{\left( -4\left( \sqrt{x^2+y^2}-6.5\right)^2\right)} e^{\left(- \frac{v_x^2+v_y^2}{2}\right)},
\end{equation} 
where $\alpha = 0.2$, $k = 7$, $\xx \in [-10,10]\times [-10,10]$, $\vv\in \mathbb{R}^{2}$. We set Dirichlet boundary conditions both in $x$ and in $y$.  We let the dynamics to evolve up to time $T=200$ with time step $h=0.1$. 
In Figure \ref{fig:initial_conf_diocotron} the initial density and three snapshots of the dynamics taken at time $t=50$, $t=100$, $t=200$ are shown. The presence of a strong magnetic field leads to the formation of vortices, as expected^^>\cite{filbet2016asymptotically}.
\begin{figure}[h!]
	\centering
		\includegraphics[width=0.328\linewidth]{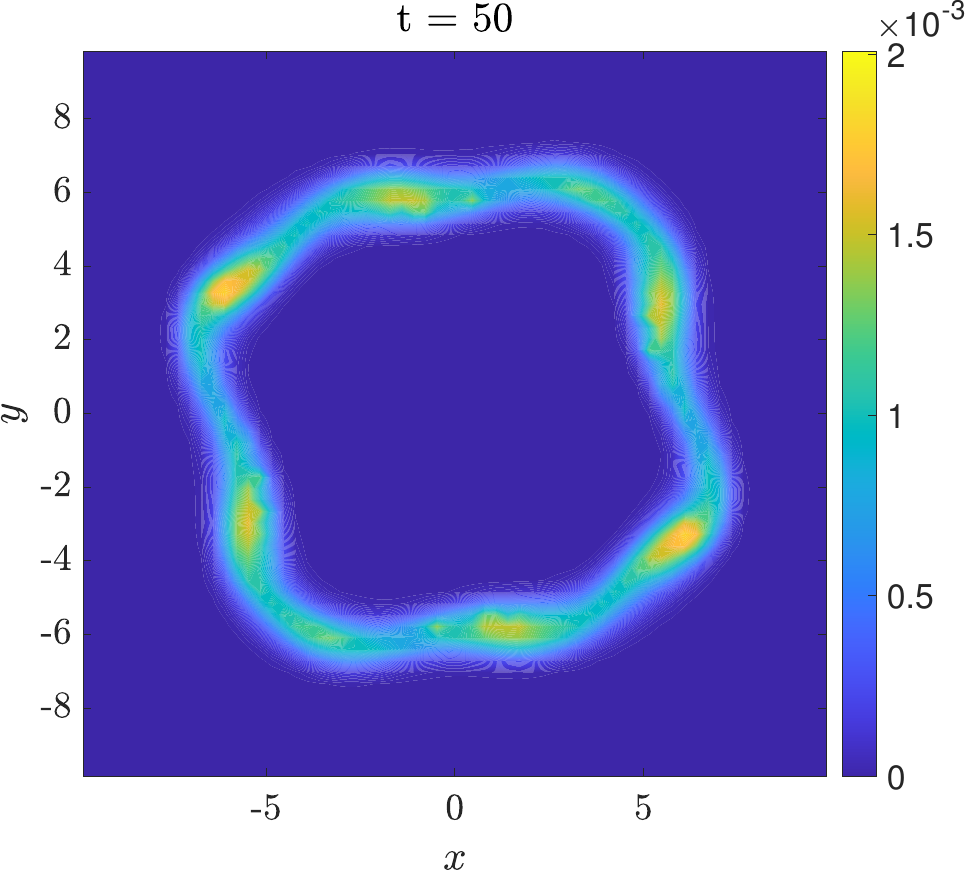}
	\includegraphics[width=0.328\linewidth]{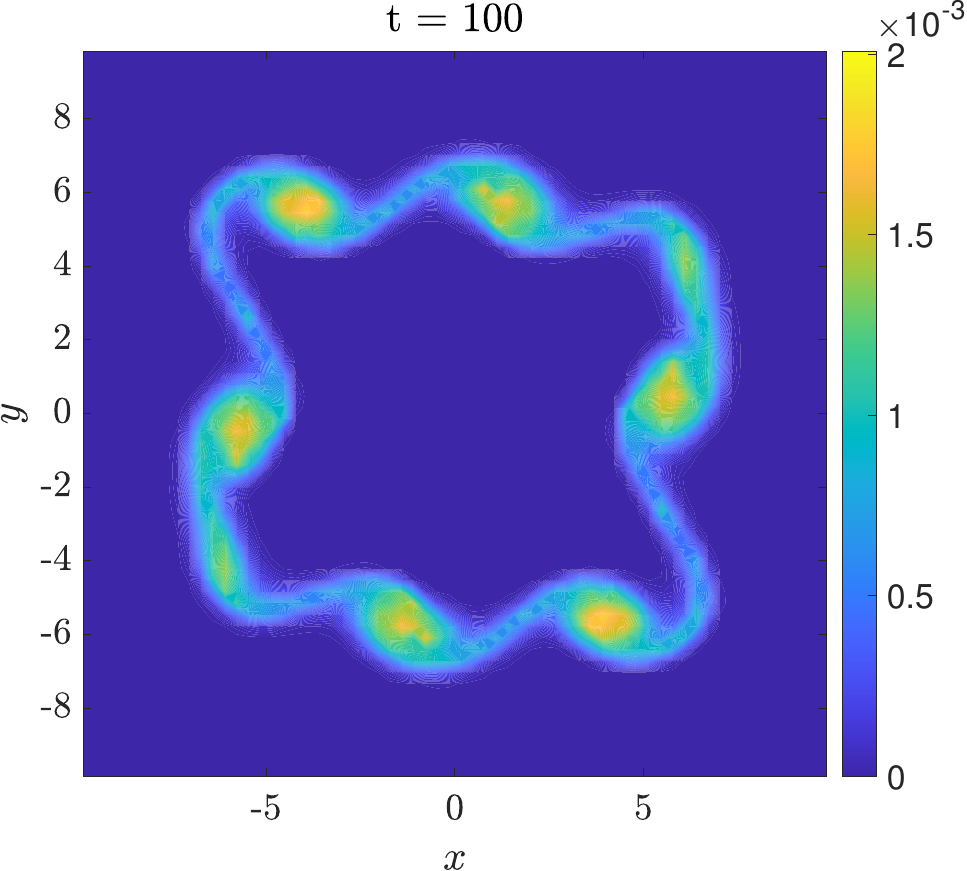}
	\includegraphics[width=0.328\linewidth]{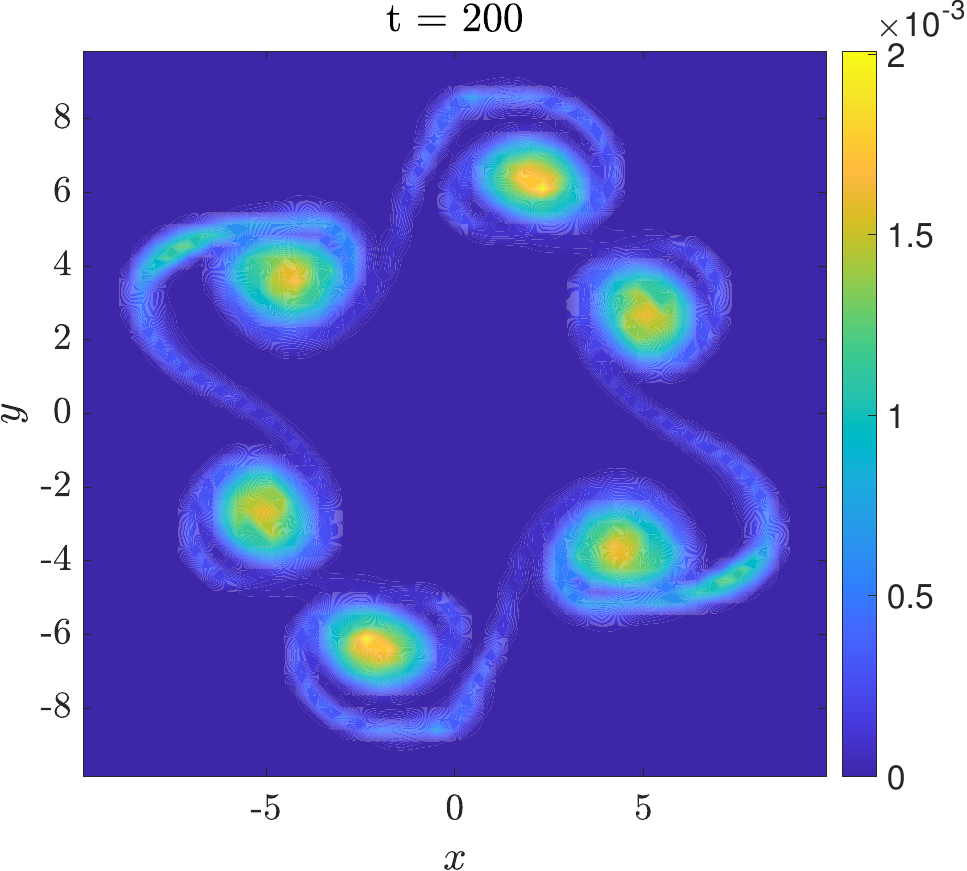}
	\caption{Diocotron instability. Snapshots of the dynamics taken at time $t=50$, $t=100$ and $t=200$, described in equation \eqref{eq:vlasov_poisson} with $B=10$.}
	\label{fig:initial_conf_diocotron}
\end{figure}
\paragraph{Diocotron instability: numerical errors.} 
Figure \ref{fig:error_time} shows the error in time
of the semi-implicit numerical schemes introduced in Section \ref{sec:Numerical_methods}. Denote by $U = [\xx_m,\vv_m]$ the position and velocity of each particle. The error in time is defined as 
\begin{equation}\label{eq:error_time}
	err_T = \Vert U^T_{rif} - U^T \Vert_{\infty},
\end{equation}
with $U_{rif}^T$ a reference solution at time $T = 1$ computed assuming $N_t =2^{10}$, and $U^T$ a solution computed  at time $T = 1$ with $N_t= 16,32,64,128$. As expected, the error of the first order scheme is proportional to $h$, while the one of the second order scheme is proportional to $h^2$, where $h$ is the time step size. 

\begin{figure}[h!]
	\centering
	\includegraphics[width=0.328\linewidth]{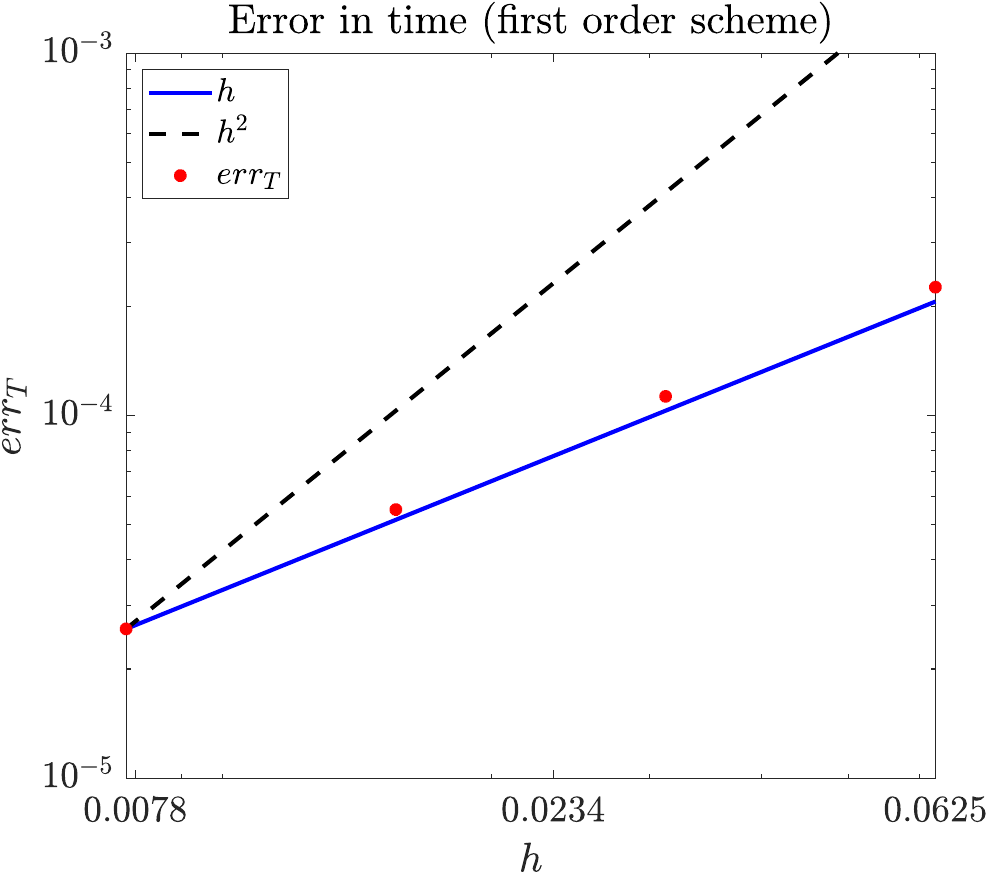}
		\includegraphics[width=0.328\linewidth]{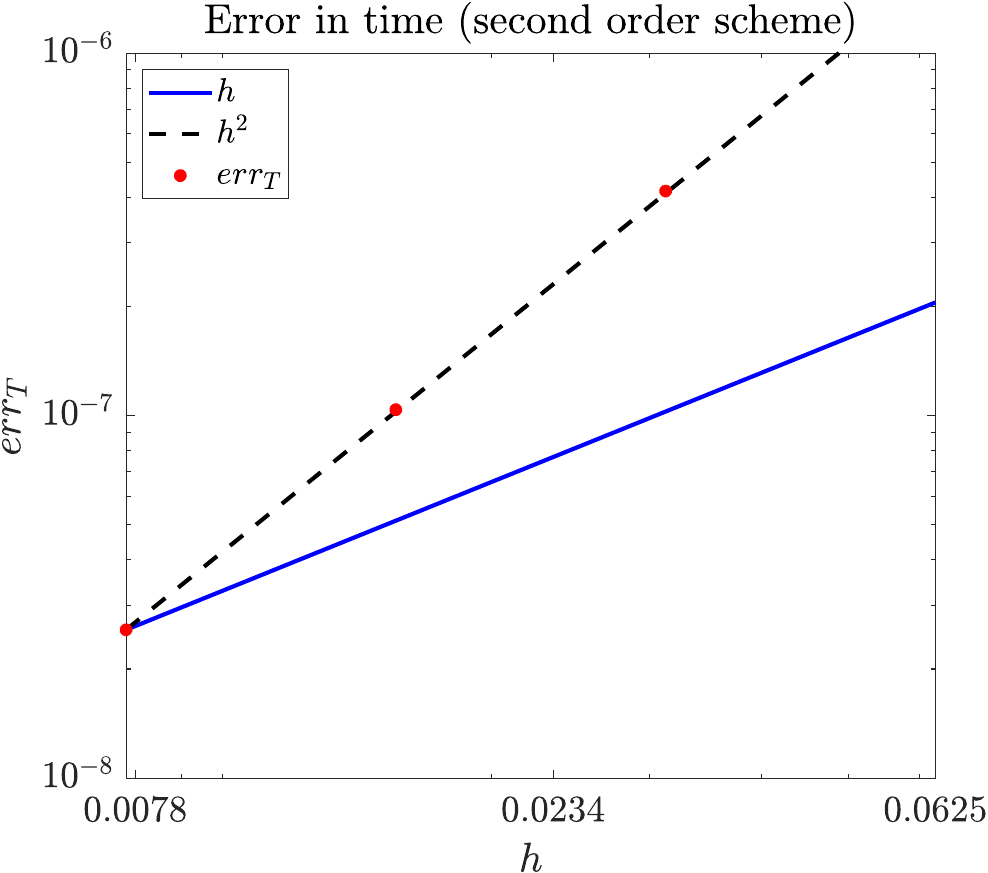}
	\caption{Diocotron instability. Error in time defined as in \eqref{eq:error_time}. On the left, the first order scheme. on the right, the second order scheme.}
	\label{fig:error_time}
\end{figure}
Finally, suppose the time step to be fixed and assume to vary the number of particles $N$. Compute the error as 
\begin{equation}\label{eq:error_N}
	err_N = \max_\vv \left\lbrace  \Vert f^{N_{rif}}(\xx,\vv) -f^N(\xx,\vv) \Vert_\infty\right\rbrace,
\end{equation}
where $ f^{N_{rif}}$ and $f^N$ represent the density function reconstructed with $N_{rif}$ and $N$ particles respectively, with $N_{rif}>> N$. Figure \ref{fig:error_N} shows the error for $N_{rif} = 10^7$ and $N=10^4,\ldots,10^6$. If we assume to generate the initial position and velocity of each particles randomly, then the error in the PIC scheme is proportional to $1/\sqrt{N}$, as show on the left. If instead we assume to generate particles following a deterministic method, i.e. defining a phase space mesh and selecting as the initial positions and velocities of the particles the corresponding center of the cells, then the error in the PIC method is proportional to $1/N$, as show on the right. 
\begin{figure}[h!]
	\centering
	\includegraphics[width=0.328\linewidth]{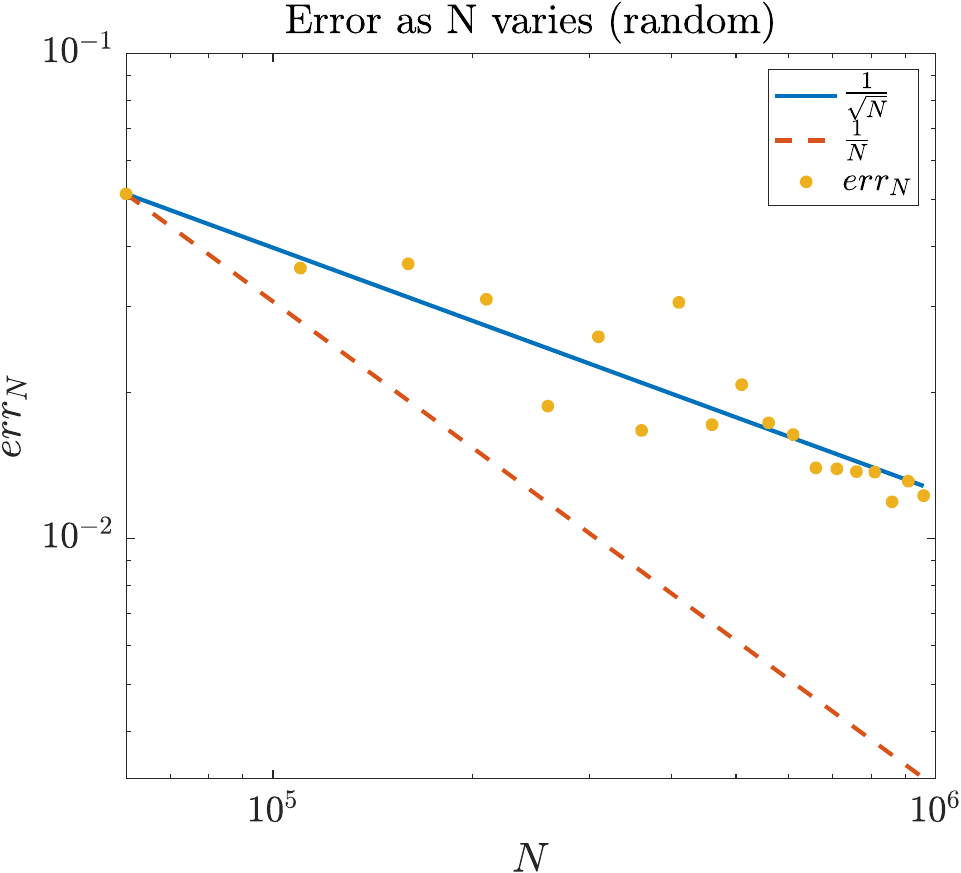}
	\includegraphics[width=0.328\linewidth]{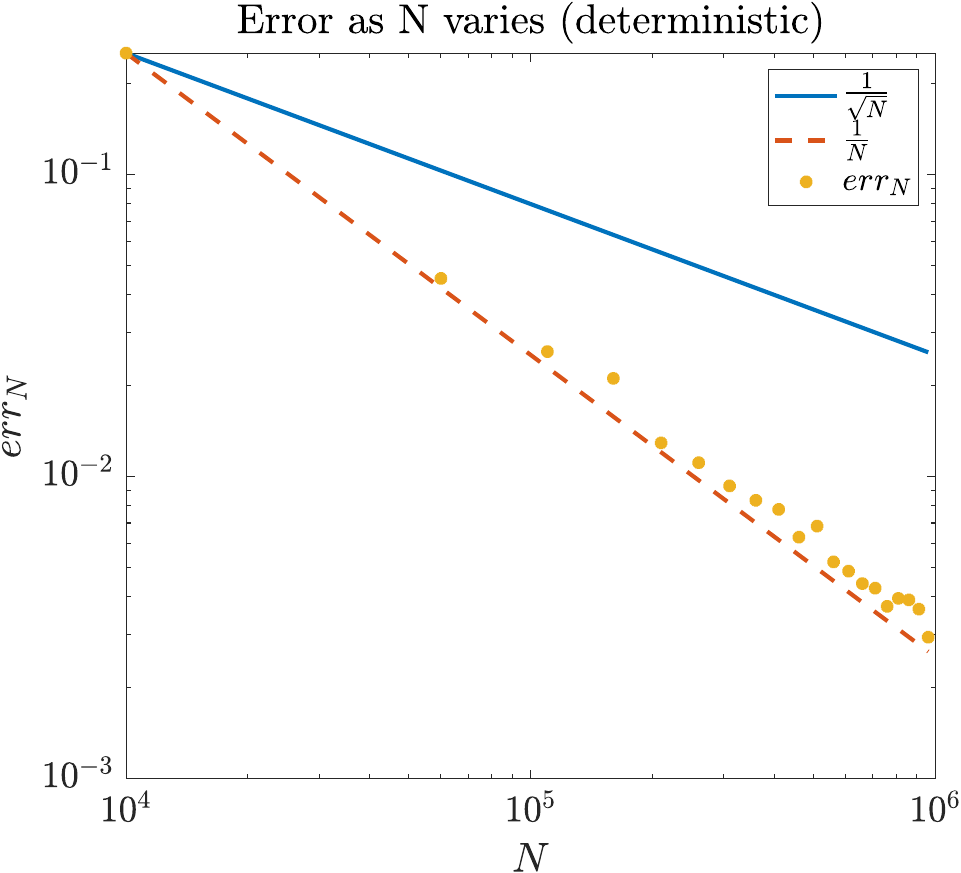}
	\caption{Diocotron instability. Error as a function of the number of particles $N$ defined as in \eqref{eq:error_N} obtained by setting the initial positions and velocities randomly (on the left) and following a deterministic method (on the right).}
	\label{fig:error_N}
\end{figure}
\end{appendices}
\bibliographystyle{abbrv}
\bibliography{biblio}
	\end{document}